\documentclass[12pt]{article}
\usepackage{bbm}
\usepackage{amsmath,amsthm,amssymb}
\usepackage{enumerate}
\usepackage{stmaryrd}
\usepackage{mathtools}
\usepackage[dvipsnames]{xcolor}
\usepackage{ulem}
\usepackage{authblk}
\usepackage{comment}
\usepackage[top=1in, bottom=1in, left=1.2in, right=1.2in]{geometry}
\newtheorem{thm}{Theorem}[section]

\newtheorem{lem}[thm]{Lemma}

\newtheorem{cor}[thm]{Corollary}
\newtheorem{prop}[thm]{Proposition}

\newtheorem{rem}[thm]{Remark}

\numberwithin{equation}{section}

\newcommand\vn{\omega}

\title{Homogenization of layered materials with rigid components in two-slip finite crystal plasticity}
\author[1]{Akira Ishikawa}
\author[2]{Karel Svadlenka}
\affil[1]{Department of Mathematics, Kyoto University, Japan} 
\affil[2]{Graduate School of Science, Tokyo Metropolitan University, Japan}

\begin{document}

\maketitle
\begin{abstract}
This paper is an extension of the result by Christowiak and Kreisbeck \cite{Christowiak}, which addresses the $\Gamma$-convergence approach to a homogenization problem for composite materials consisting of two distinct types of parallel layers. In \cite{Christowiak}, one of the layers of the material undergoes only local rotations while the other allows local rotation and plastic deformation along a single slip system. On the other hand, real materials show an interplay of multiple directions of slip. Here we obtain the $\Gamma$-limit for the problem where the number of slip directions is increased to two. When the slip systems are orthogonal, we derive the full homogenized energy based on generalized convex envelopes of the original energy density. Since these envelopes are not completely known for angles between slip directions differing from the right angle, we present a partial, conditional homogenization result in this general case. The analysis is based on a modification of the classical construction of laminate microstructures but several nontrivial difficulties arise due to  nonconvex constraints being present in the composite energy.
\end{abstract}

\newpage
\tableofcontents

\section*{List of symbols}

\begin{tabular}{ll}
$v^{\perp}$ $\quad$ & vector obtained by rotating vector $v\in\mathbb{R}^2$ by $\pi/2$ counterclockwise\\
$I$ $\quad$ & $2\times 2$ identity matrix \\
$SO(2)$ $\quad$ & the rotation group of $\mathbb{R}^2$, i.e., the group consisting of all orthogonal \\
& $2\times 2$ matrices with determinant $1$\\
$\mathbbm{1}_{X}$ $\quad$ & indicator function of a set $X$\\
$(a)_+^s$ $\quad$ & means $(\max\{0,a\})^s$, $a\in\mathbb{R}, s>0$ \\
$c$ $\quad$ & a generic positive constant, \\
$\;$ & independent of the involved parameters, unless stated otherwise \\
$e_i$ $\quad$ & the unit vector along the positive direction of coordinate axis $x_i$ in $\mathbb{R}^2$\\
$\partial_i$ $\quad$ & the first-order partial differential operator with respect to $x_i$\\
$L^p(D)$ $\quad$ & the Lebesgue space of functions with integrable $p$-powers for $1 \leq p < \infty$\\
$L^{\infty}(D)$ $\quad$ & the Lebesgue space of essentially bounded functions\\
$L^p_{\rm loc}(D)$ $\quad$ & the Lebesgue space of functions with locally integrable $p$-powers\\
$L_0^p(D)$ $\quad$ & the set of functions in $L^p(D)$ with vanishing mean\\
$W^{k,p}(D)$ $\;\;$ & the Sobolev space of functions with all $k$-th derivatives in $L^p(D)$\\
$\to$ $\quad$ & strong convergence in a normed linear space \\
$\rightharpoonup$ $\quad$ & weak convergence in a normed linear space
\end{tabular}

\newpage

\section{Introduction and main results}
\label{sec_intro}
A major task of modern material science is to develop new materials with properties exceeding the existing ones. 
Combination of two or more materials provides one possible way to do so.
It turns out that a clever choice of material components and their suitable spatial distribution may lead to a dramatic improvement in mechanical strength, ductility, stiffness, corrosion resistance and other characteristics.  

Investigations of the relation between properties and spatial arrangement of material components and the response of composite materials have been carried out thoroughly, especially by experimental, engineering and numerical methods (see, e.g., \cite{Low2021, Mantic2022, Milton2002} and references therein), while mathematical tools to address this relation have been developed alongside. 
The main approach of this mathematical research is homogenization, where the scale of the geometrical structure of composite material is taken zero to obtain an averaged or effective material model.
This model is usually simpler than the full model of composite material while retaining its important physical features and thus is suited for numerical implementations and in some cases even for theoretical analysis aimed at revealing the principles behind improved material properties.
A successful analytical tool introduced by  De Giorgi et al. \cite{DeGiorgi1975} for variational models based on energy minimization is the concept of $\Gamma$-convergence, providing a mathematically rigorous and physically consistent way to define limits of sequences of energy functionals.

The still active research of composite materials has been recently boosted by the discovery of so-called LPSO magnesium alloys and a possibility of a new principle of material strengthening related to kink-band formation in the periodic layered structure of the alloy \cite{Drozdenko2022, Hagihara2019, Plummer2021, Wadee2004}. 
The LPSO alloys shows a periodic milefeuille structure of soft pure magnesium layers and rigid layers where additive elements accumulate.  
In this paper, we contribute towards the mathematical understanding of the kink-band strengthening mechanism, building on the seminal work of Christowiak and Kreisbeck \cite{Christowiak}, where variational models of bilayered materials are studied (see also \cite{Bouchitte2002, Braides1995, Cherdantsev2012, Francfort2014} for a few examples of related results).
They consider the deformation of a planar material in which band-shaped rigid layers and softer layers with one available direction of plastic slip are periodically repeated. 
Since the actual LPSO alloy has multiple slip systems in its soft layers, which are thought to contribute to the strengthening, here we address the problem with two different slip systems in soft layers.

To introduce the setting of the problem, let $\Omega \subset \mathbb{R}^2$ be a simply connected Lipschitz domain representing (the reference configuration of) an elastoplastic material in plane undergoing deformation, and $u:\Omega \rightarrow \mathbb{R}^2$ be the function representing the deformation.
Let $\lambda \in (0, 1)$ represent the fraction of the softer layer in the periodic structure, and for the unit cell $Y = [0, 1)^2$ define the subsets
\begin{equation}
\label{Ysoft}
    Y_{soft} := [0, 1)\times [0, \lambda) \subset Y \qquad \text{and} \qquad Y_{rig} := Y \backslash Y_{soft} ,
\end{equation}
which represent the soft and rigid parts, respectively. 
We extend this unit structure to the whole $\mathbb{R}^2$ and scale the oscillation width with $\epsilon>0$, which represents the layer thickness. 
Then the hard and soft layers $\epsilon Y_{rig} \cap \Omega$ and $\epsilon Y_{soft} \cap \Omega$ are parallel to $e_1$, as shown in Fig. \ref{periodic-structure}.

\begin{figure}[!ht]
\begin{center}
\includegraphics[width=0.75\textwidth]{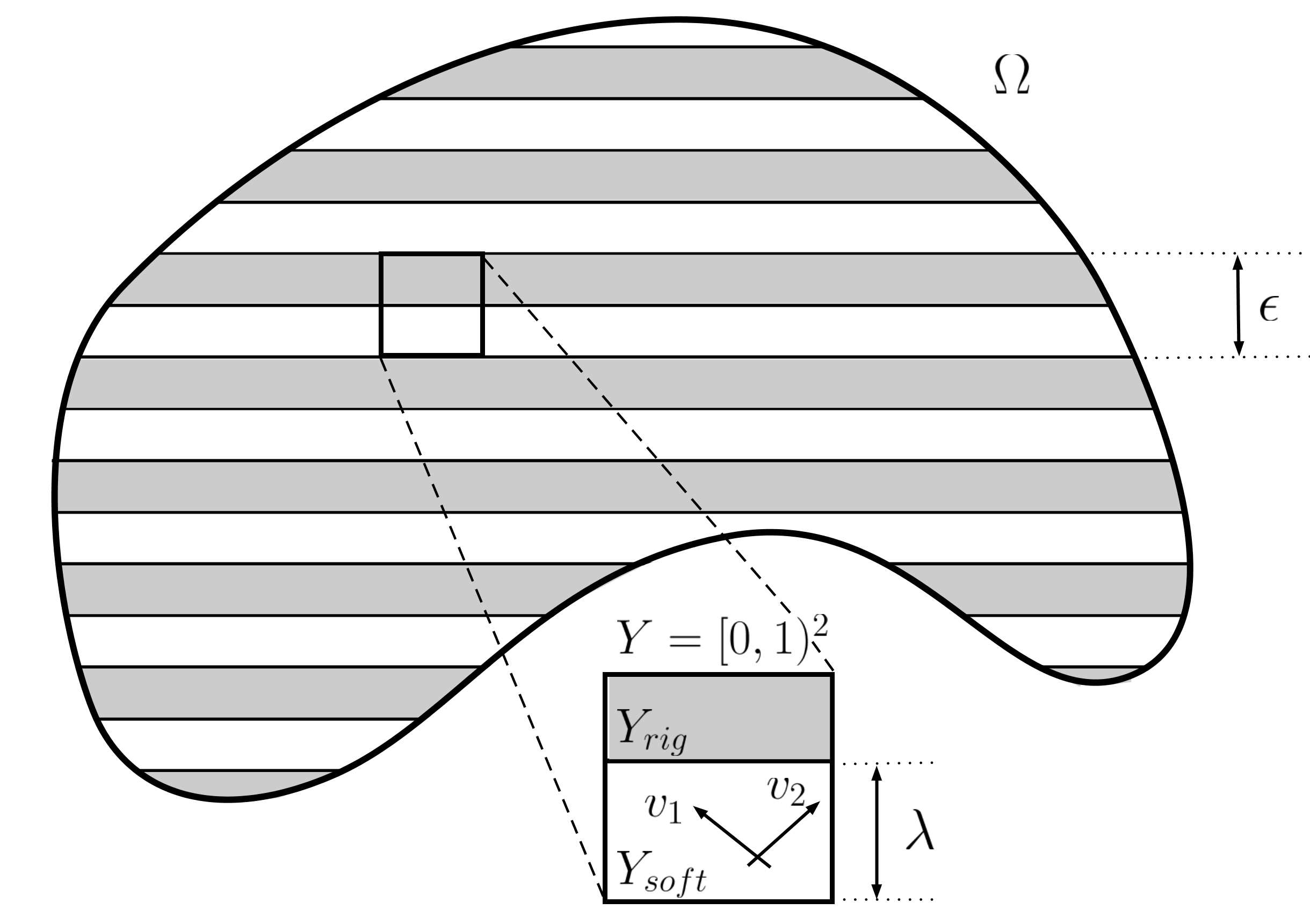}
\end{center}
\caption{A periodic structure with two layers of thin rigid and soft components is shown. There are two slip systems in the soft layer with slip directions $v_1$ and $v_2$.}
\label{periodic-structure}
\end{figure}

We use the Kr\"{o}ner-Lee multiplicative decomposition of the deformation gradient $\nabla u=F_{e}F_{p}$ as a fundamental assumption \cite{Kroner1960, Lee1969}.
Here, the elastic part $F_e$ describes local rotation and stretching of the
crystal structure, and the inelastic part $F_{p}$ accounts for local plastic deformation due to dislocations.
For simplicity, we consider only local rotations for $F_e$, so values of $F_e$ are restricted to $SO(2)$, which is a reasonable approximation for metallic materials \cite{Kruzik2019}. We also assume that there is no plastic deformation in the hard layer, i.e., $F_p = I$ on $\epsilon Y_{rig} \cap \Omega$. In the soft layers $\epsilon Y_{soft}\cap\Omega$ plastic glide can occur along one of two active slip systems $(v_1,v_1^{\perp}), (v_2,v_2^{\perp})$ with slip directions given by unit vectors $v_1,v_2\in\mathbb{R}^2$. Here $v^{\perp}$ denotes the counterclockwise rotation of $v\in\mathbb{R}^2$ by $\pi /2$. Hence, $F_p=I+\gamma_1 v_1\otimes v_1^{\perp}+\gamma_2 v_2\otimes v_2^{\perp}$ in $\epsilon Y_{soft}\cap\Omega$, where $\gamma_1, \gamma_2 \in\mathbb{R}$ corresponds to the amount of slip in each slip direction and satisfies $\gamma_1\gamma_2=0$ a.e. in $\Omega$ \cite{ContiTheil2005, Gurtin2000}. 
 
If the slip direction at a point is $v_i\;(i=1,2)$, the deformation gradient at that point is restricted to the set
\begin{align*}
\mathcal{M}_i := & \left\{F\in\mathbb{R}^{2\times 2}:F=R(I+\gamma_i v_i\otimes v_i^{\perp}), \, R\in SO(2), \, \gamma_i\in\mathbb{R}\right\}\\
= & \left\{F\in\mathbb{R}^{2\times 2}:\det F=1, \, |Fv_i|=1\right\}.
\end{align*}
In the rigid layer, both $\gamma_1$ and $\gamma_2$ vanish, so that the deformation gradient is restricted to $SO(2)= \mathcal{M}_1\cap\mathcal{M}_2$. 

In view of rigid elasticity in our model, and considering that no plastic slip is allowed in rigid layer, for energy density in the rigid layer we have
\begin{equation*}
W_{rig}(F)=
\begin{cases}
0&\text{if}\;F \in SO(2)\\
\infty & \text{otherwise}.
\end{cases}
\end{equation*}
This corresponds to the gradient being restricted to $SO(2)$ in $\epsilon Y_{rig}$.
On the other hand, using the two-slip model in \cite{Conti2013,Conti2015}, the energy density in the soft layer is
\begin{equation*}
W_{soft}(F)=
\begin{cases}
\gamma_1^2+\gamma_2^2&\text{if}\;F \in \mathcal{M}\\
\infty & \text{otherwise}.
\end{cases}
\end{equation*}
where $\mathcal{M} = \mathcal{M}_1\cup\mathcal{M}_2$.
Note that since slip in only one direction is allowed at every point, $\gamma_1^2+\gamma_2^2$ is equal either to $\gamma_1^2$ or to $\gamma_2^2$.
 
Summarizing the above, we can define the energy functional $E_{\epsilon}:W^{1,2}(\Omega;\mathbb{R}^2)\cap L_0^{2}(\Omega;\mathbb{R}^2)\rightarrow \mathbb{R}\cup \{\infty\}$ as 
\begin{equation*}
E_{\epsilon}(u):=
\begin{cases}
\displaystyle \int_{\Omega} \left(\gamma_1^2+\gamma_2^2 \right) \, dx & 
\text{if}\; u\in W^{1,2}\cap L_0^{2}(\Omega;\mathbb{R}^2), \; \text{where} \\ 
& \quad \nabla u=R(I+\gamma_1 v_1\otimes v_1^{\perp} + \gamma_2 v_2\otimes v_2^{\perp}),\\
& \quad R\in L^{\infty}(\Omega;SO(2)),\,\gamma_i\in L^2(\Omega), \\
& \quad \gamma_i=0\,\,\text{a.e. in}\,\epsilon Y_{rig}\cap \Omega,
 \; \gamma_1\gamma_2=0\,\, \text{a.e. in } \Omega, \\
\infty & \text{otherwise}.
\end{cases}
\end{equation*}

We remark that for orthogonal slip directions one can write
\begin{equation*}
\gamma_1^2+\gamma_2^2=|(\nabla u)v_1|^2+|(\nabla u) v_2|^2-2.
\end{equation*}
Hence, the functional $E_{\epsilon}$ can be rephrased as 
\begin{equation*}
E_{\epsilon}(u)=
\begin{cases}
\displaystyle\int_{\Omega}W(\nabla u) \,dx& 
\text{if}\;
u\in W^{1,2}\cap L_0^{2}(\Omega;\mathbb{R}^2),\,\nabla u \in \mathcal{M}\\
 & \quad \text{and} \; \nabla u \in SO(2) \; \text{in} \; \epsilon Y_{rig} \cap \Omega, \\
 \infty &  \text{otherwise},
\end{cases}
\end{equation*}
where 
\begin{equation}
\label{integrand}
W(F)
=\begin{cases}
|Fv_1|^2+|Fv_2|^2-2 & \text{if}\; \;F\in\mathcal{M},\\
    \infty & \text{otherwise}.
\end{cases}
\end{equation}

The main result of this work are the the following two homogenization theorems.
The first one fully describes the homogenized material for the orthogonal case of $v_2=v_1^{\perp}$, while the second one provides a partial analysis for the case of arbitrary angle between slip directions.

\begin{thm}[Homogenization via $\Gamma$-convergence -- orthogonal case]
\label{maintheorem} $\;$ \\
Let $v_2=v_1^{\perp}$. Then the family $(E_{\epsilon})_{\epsilon}$ $\Gamma$-converges as $\epsilon\to 0$ with respect to the strong $L^{2}(\Omega;\mathbb{R}^2)$-topology to the functional 
\begin{equation*}
E(u)=
\begin{cases}
\lambda\displaystyle\int_{\Omega}W_{hom}(N) \,dx& 
\text{if}\;
u\in W^{1,2}\cap L_0^{2}(\Omega;\mathbb{R}^2),\,\nabla u=R(I+\gamma e_1\otimes e_2),\\
& \quad R\in SO(2),\,\gamma\in L^2(\Omega),\\
    \infty &  \text{otherwise}.
\end{cases}
\end{equation*}
Here, $N$ is given by the formula $\lambda N+(1-\lambda)R=\nabla u$, and $W_{hom}:\mathbb{R}^{2\times 2}\rightarrow \mathbb{R}$ is defined by 
\begin{equation}
\label{gamma-limit-energydensity}
W_{hom}(F)
:=\begin{cases}
|Fv_1^{\perp}|^2-1 & \text{if}\;\;\det F=1,\;|Fv_1|\leq 1,\\
|Fv_2^{\perp}|^2-1 & \text{if}\;\;\det F=1,\;|Fv_2|\leq 1,\\
\chi(|Fv_3|) & \text{if}\;\;\det F=1,\;|Fv_1|,|Fv_2|>1 \text{ and } Fv_1\cdot Fv_2> 0,\\
\chi(|Fv_3^{\perp}|) & \text{if}\;\;\det F=1,\;|Fv_1|,|Fv_2|>1 \text{ and } Fv_1\cdot Fv_2< 0,\\
\infty & otherwise,
\end{cases}
\end{equation}
where $v_3=-(v_1+v_2)/|v_1+v_2|=-(v_1+v_2)/\sqrt{2}$, and the function $\chi$ is defined by 
\begin{equation*}
\chi(z)=\left(\left(2z^2-1 \right)_+^{1/2}-1 \right)_+^2, \qquad z\in \mathbb{R}.
\end{equation*}
The symbol $(a)_+^s$ for $a\in\mathbb{R}$ and $s>0$ means $(\max\{a,0\})^s$.
\end{thm}
The function $W_{hom}$ appearing in Theorem \ref{maintheorem} coincides with the rank-one convex envelope of $W$ in \eqref{integrand} obtained in \cite{Conti2013}. 
Note that if $\det F=1$ then $|Fv_1|,|Fv_2|> 1$ and $Fv_1\cdot Fv_2=0$ cannot happen at the same time (see proof of Lemma \ref{rank-one-connection}). Moreover, since $|Fv_3|\gtrless |Fv_3^{\perp}|$ is equivalent to $Fv_1\cdot Fv_2\gtrless 0$ (see Lemma \ref{exact}), we can simply write 
\begin{equation*} W_{hom}(F) = 
\chi(\max\{|Fv_3|,|Fv_3^{\perp}|\}) \qquad \text{if}\; \; \det F=1,\;|Fv_1|,|Fv_2|> 1.
\end{equation*}

\begin{thm}[Homogenization via $\Gamma$-convergence -- general case]
\label{maintheorem2} $\;$ \\
Let $(v_1, v_2)$ form a right-handed system with angle $2\theta \in (\pi/2,\pi)$.
Let the family $(E_{\epsilon})_{\epsilon}$ $\Gamma$-converge with respect to the strong $L^{2}(\Omega;\mathbb{R}^2)$-topology to an integral functional $E(u):L_0^2(\Omega;\mathbb{R}^2)\rightarrow [0,\infty]$ whose integrand is a function  of $\nabla u$. Then $E$ can be represented as
\begin{equation*}
E(u)=
\begin{cases}
\lambda\displaystyle\int_{\Omega}W_{hom}(N) \,dx& 
\text{if}\;
u\in W^{1,2}\cap L_0^{2}(\Omega;\mathbb{R}^2),\,\nabla u=R(I+\gamma e_1\otimes e_2),\\
& \quad R\in SO(2),\,\gamma\in L^2(\Omega),\\
    \infty &  \text{otherwise}.
\end{cases}
\end{equation*}
Here, $N$ is given by the formula $\lambda N+(1-\lambda)R=\nabla u$, and $W_{hom}:\mathbb{R}^{2\times 2}\rightarrow \mathbb{R}$ is given in part by
\begin{equation*}
W_{hom}(F)
:=\begin{cases}
|Fv_1^{\perp}|^2-1 & \text{if}\;\;\det F=1,\;|Fv_1|= 1,\\
|Fv_2^{\perp}|^2-1 & \text{if}\;\;\det F=1,\;|Fv_2|= 1,\\
h(|Fv_3|) & \text{if}\;\;\det F=1,\;|Fv_1|,|Fv_2|<1,\\
h(|Fv_3|) & \text{if}\;\;\det F=1,\;|Fv_1|,|Fv_2|>1 \text{ and } Fv_1\cdot Fv_2> 0,\\
h^{\perp}(|Fv_3^{\perp}|) & \text{if}\;\;\det F=1,\;|Fv_1|,|Fv_2|>1 \text{ and } Fv_1\cdot Fv_2< 0,\\
\infty & \text{if}\;\;\det F \neq 1,
\end{cases}
\end{equation*}
where $v_3 = -(v_1 + v_2)/|v_1 + v_2|$, and the functions $h,h^{\perp}$ are defined by
\begin{equation*}
h(z):= \left( \left(\tfrac{z^2}{\sin^2\theta} -1\right)_+^{1/2} - \tfrac{\cos\theta}{\sin\theta} \right)_+^2 , \quad
h^{\perp}(z):= \left(\left(\tfrac{z^2}{\cos^2\theta} -1\right)_+^{1/2} - \tfrac{\sin\theta}{\cos\theta} \right)_+^2 .
\end{equation*}
\end{thm}

We emphasize that this is a conditional and partial statement in the sense that it is assumed that the $\Gamma$-limit is an integral functional, while $W_{hom}(F)$ remains unknown for matrices $F$ with either $|Fv_1|>1$, $|Fv_2|<1$ or $|Fv_1|<1$, $|Fv_2|>1$.
Notice also that in the orthogonal case $\theta = \pi/4$, we always have $\min\{|Fv_1|,|Fv_2|\}\geq 1$ for $\det F=1$, and both functions $h$ and $h^{\perp}$ reduce to the function $\chi$ from Theorem \ref{maintheorem}.

We define the $\Gamma$-limit along a continuous parameter $\epsilon \rightarrow 0+$ by requiring $\Gamma$-convergence along arbitrary sequence $(\epsilon_j)_j$ such that $\epsilon_j \rightarrow 0+$.
Therefore, the proof of Theorem \ref{maintheorem} consists of showing the following three claims for any sequence $(\epsilon_j)_j$ in $\mathbb{R}$ such that $\epsilon_j \rightarrow 0+$ as $j\rightarrow \infty$. 
\begin{description}
\item[Compactness] (Section \ref{Compactness}): Let $(u_j)_j$ be a sequence in $L^2_{0}(\Omega;\mathbb{R}^2)$ such that $(E_{\epsilon_j} (u_j))_j$ is uniformly bounded. 
Then there exists a subsequence of $(u_j)_j$ whose limit in the strong $L^{2}(\Omega;\mathbb{R}^2)$-topology is $u$ satisfying $E(u) < \infty$.
\item[Liminf inequality] (Section \ref{Lower bound}): Let $(u_j)_j$ be a sequence in $L^2_{0}(\Omega;\mathbb{R}^2)$ with $u_j \rightarrow u$ in $L^2(\Omega;\mathbb{R}^2)$.
Then
\begin{equation}
\label{liminf-equation}
\liminf_{j\rightarrow \infty}E_{\epsilon_j}(u_j)\geq E(u).
\end{equation}
\item[Recovery sequence] (Section \ref{Recovery sequence}): For any $u \in L^2_{0}(\Omega;\mathbb{R}^2)$ with $E(u) < \infty$, there exists $(u_{j})_{j} \subset L^2_{0}(\Omega;\mathbb{R}^2)$ such that $u_{j} \rightarrow u$ in $L^2(\Omega;\mathbb{R}^2)$ and $\lim_{j\rightarrow \infty} E_{\epsilon_j}(u_{j}) = E(u)$.
\end{description}

The basic flow of the proof when the slip systems are orthogonal is similar to the one proposed in \cite{Christowiak}. In particular, in the proof of existence of recovery sequences, we approximate a limit by expressing $N$ which appears in Theorem \ref{maintheorem} as a rank-one convex combination of matrices with finite energy. Since the set of admissible deformations is widened by increasing the number of slip systems to two, the construction of rank-one convex combinations needs to be modified. In addition, due to the higher complexity of the homogenized functional, the simple approach adopted in \cite{Christowiak} of reducing the construction of recovery sequences for general admissible functions for the $\Gamma$-limit to the the case of piecewise affine functions does not work. Similarly, the fact that, unlike the problem with only one slip system, the homogenized functional $E$ has a different form from the original functional $E_{\epsilon}$ makes the proof of liminf inequality difficult. Namely, the different form of the homogenized functional does not allow to use the same method as in Corollary 2.5 of \cite{Christowiak} to pass to limit along a sequence of non-admissible deformations when proving the liminf inequality. Another essential complication in the proof of liminf inequality is that a sequence converging to an admissible deformation can choose from two slip directions at each point. 

In this paper, we addressed the above issues in the following way. When proving the existence of recovery sequences, we first show that, in general, if piecewise affine functions are dense in the set of admissible functions for the $\Gamma$-limit, it is sufficient to deal with the case of piecewise affine limit functions. 
Subsequently, the piecewise affine case is realized through the construction of rank-one convex combinations, referring to \cite{Conti2013}. 
In order to prove the liminf inequality, the energies of convergent sequences sliding in different directions at each point are collectively evaluated by a judiciously chosen function $f$, defined in Lemma \ref{continuity}. The important point is that $f$ can be chosen as a convex function.
This function $f$ plays an intermediary role to connect the different forms of the original energy density $W$ and its homogenization $W_{hom}$. 
We show that the extra terms arising due to this discrepancy between the densities can be estimated employing the function $f$.
The same approach generalizes to the problem with arbitrary angles between slip directions.

Section \ref{section_envelopes} is devoted to the attempt to find generalized convex envelopes \cite{Dacorogna} of the energy density function in the general case of arbitrary slip systems.
Due to the interference of the slip systems, the envelopes are identified using certain symmetries only in a certain subset of $\mathbb{R}^{2\times 2}$. 
Based on this information, we then reveal the form of the $\Gamma$-limit for arbitrary slip systems summarized in Theorem \ref{maintheorem2}. 
The incomplete knowledge of the envelopes is reflected in the partiality of the result and also forces us to assume the integral form of the homogenized functional.

We remark that routine technical modifications can be done to deal with the problem with $W_{soft}(F)=\gamma_1^p + \gamma_2^p$ for $F\in\mathcal{M}$, where $p>2$.

On the other hand, the variant of our problem, which consists in adding a linear term $|\gamma_1|+|\gamma_2|$ to the energy density, remains fully open for the setting of more than one slip system. 
This term represents dissipative contribution and is thus vital for addressing the evolution problem \cite{Carstensen2002, Davoli2015, Ortiz1999}, for example, in the framework of rate-independent systems \cite{Mielke2015}.
For the problem with one slip system, the authors of \cite{Christowiak} were able to identify the $\Gamma$-limit in the particular case of slip direction parallel to the layers by noticing a special structure leading to strong convergence of rotations  
and hence weak convergence of slips. However, for two slip systems, the trick to show strong convergence of rotations is not available.

\section{Homogenization for orthogonal slips} 

In this section, we present the details of proof of the main result, Theorem \ref{maintheorem}, which deals with the special case of orthogonal slip systems.
We start with compactness, continue with the proof of liminf inequality and finish with a construction of recovery sequences, as outlined after the statement of main theorems in Section \ref{sec_intro}.

\subsection{Compactness }
\label{Compactness}

To begin with, we look at the asymptotic behavior of material with stiff layers in the limit of vanishing layer width $\epsilon$. It has been described fully in the following fundamental proposition quoted from \cite{Christowiak}.  

\begin{prop}[Weak limit of deformation of materials with stiff layers]
\label{Asymptotic-rigidity} $\;$ \\
Let \(\Omega\subset \mathbb{R}^2\) be a bounded
Lipschitz domain. Suppose that the sequence \((u_{\epsilon})_{\epsilon}\subset W^{1,2}(\Omega;\mathbb{R}^2)\) satisfies \(u_{\epsilon}\rightharpoonup u\) in
\(W^{1,2}(\Omega;\mathbb{R}^2)\) as \(\epsilon \rightarrow 0\) for some \(u \in W^{1,2}(\Omega;\mathbb{R}^2)\) with \( \det \nabla u=1\,\) a.e. in $\Omega$, and
\begin{equation*}
\nabla u_{\epsilon}\in SO(2)\quad \text{a.e. in }\, \Omega \cap \epsilon Y_{rig}\quad \forall \epsilon >0,
\end{equation*}
with \(Y_{rig}\) defined in \eqref{Ysoft}. Then there exists a matrix \(R\in SO(2)\) and a function \(\gamma\in L^2(\Omega)\) such that
\begin{equation}
\label{lim-form.1}
\nabla u=R(I+\gamma e_1\otimes e_2).
\end{equation}
Furthermore,
\begin{equation}
\label{lim-form.2}
\nabla u_{\epsilon}\mathbbm{1}_{\epsilon Y_{rig}\cap \,\Omega} \rightharpoonup \vert Y_{rig} \vert R \quad in \; L^2(\Omega;\mathbb{R}^2).
\end{equation}
\end{prop}

\begin{rem}
\label{independent-of-x_1}
The function $\gamma$ in \eqref{lim-form.1} is independent
of $x_1$ in the sense that its distributional derivative $\partial_1 \gamma$ vanishes.\cite{Christowiak}
\end{rem}

Since any weakly convergent sequence \((u_{\epsilon})_{\epsilon}\) of bounded energy \((E_{\epsilon})_{\epsilon}\) fulfills the assumptions of this proposition, one obtains a necessary condition on the form of deformations for the homogenized material. Note that the conclusion of the proposition is independent of the structure of slip systems in soft layers.

To prove compactness, we recall that the original energy has the form 
\begin{equation*}
E_{\epsilon}(u)=
\begin{cases}
\displaystyle\int_{\Omega}W(\nabla u) \,dx& 
\text{if}\;
u\in W^{1,2}\cap L_0^{2}(\Omega;\mathbb{R}^2) \;\; \text{and} \;\; \nabla u\in K_{\epsilon},\\
    \infty &  \text{otherwise},
\end{cases}
\end{equation*}
for a set $K_{\epsilon} \subset L^2(\Omega;\mathbb{R}^{2\times 2})$, where $W:\mathbb{R}^{2\times 2}\rightarrow [0,\infty]$ has quadratic growth, i.e.,  $W(F) \geq \mu|F|^2 - c$ for some $\mu>0$. 
When $(u_j)_j$ is a sequence satisfying $E_{\epsilon_j}(u_j)<c$ for some $c>0$ and all $j$, this implies that $(\|\nabla u_j\|_{L^2})_j$ is uniformly bounded.
Since $\int_{\Omega}u_j \, dx=0$,  $(u_j)_j$ is uniformly bounded in $W^{1,2}(\Omega;\mathbb{R}^2)$ due to Poincar\'{e}-Wirtinger inequality. Therefore, we can choose a subsequence $(u_j)_j$ (not relabeled) to be weakly convergent in $W^{1,2}(\Omega;\mathbb{R}^2)$.
Finally, using Rellich-Kondrachov theorem, we find that $(u_j)_j$ converges also strongly in $L^2(\Omega;\mathbb{R}^2)$ to a limit $u\in W^{1,2}(\Omega;\mathbb{R}^2)\cap L^2_{0}(\Omega;\mathbb{R}^2)$.
To see that $E(u)$ is finite, it is enough to note that the convergence $u_j \rightharpoonup u$ in $W^{1,2}(\Omega;\mathbb{R}^2)$ implies $\det \nabla u_j \overset{*}{\rightharpoonup} \det \nabla u$ in the sense of measures and to apply Proposition \ref{Asymptotic-rigidity}, cf. \cite{Morrey1966, Reshetnyak1968}.

\noindent
\subsection{Lower bound }
\label{Lower bound}

The following fact, which is a generalization of Corollary 2.5 from \cite{Christowiak}, plays an important role in the proof of the liminf inequality \eqref{liminf-equation}.

\begin{prop}
\label{liminf_affine}
Let $\Omega$ be a cube, to say $\Omega = Q = (0,l)^2$ for $l > 0$, and let $(u_{\epsilon})_{\epsilon}\subset W^{1,2}(\Omega;\mathbb{R}^2)$ 
be such that 
$E_{\epsilon}(u_{\epsilon})\leq c$
 for all $\epsilon>0$ and $u_{\epsilon}\rightharpoonup u$ in 
 $W^{1,2}(\Omega;\mathbb{R}^2)$ for some 
 $u\in W^{1,2}(\Omega;\mathbb{R}^2)$ 
 with gradient of the form \eqref{lim-form.1}. 
 If, in addition, $u$
is finitely piecewise affine, then
\begin{equation*}
\liminf_{\epsilon \rightarrow 0} E_{\epsilon}(u_{\epsilon}) \geq \lambda\int_{\Omega}W_{hom}\left(\frac{1}{\lambda}(\nabla u-(1-\lambda)R) \right) \,dx,
\end{equation*}
where $W_{hom}$ is defined in \eqref{gamma-limit-energydensity}.
\end{prop}

In order to prove the assertion, we prepare a lemma, which will be used also later on.

\begin{lem}
\label{continuity}
Define the function $f:\mathbb{R}^{2\times 2}\rightarrow [0,\infty)$ by
\begin{equation*}
f(F):= \max\left\{\left(|Fv_1|^2-1\right)_{+}, \, \left(|Fv_2|^2-1\right)_{+},\, \chi\left(\max\{|Fv_3|,|Fv_3^{\perp}|\}\right)\right\}.
\end{equation*}
Then $f$ is convex and coincides with $W$ on $\mathcal{M}$ and with $W_{hom}$ on the set $\mathcal{N}:=\{F\in\mathbb{R}^{2\times 2}: \det F=1\}.$ Thus, $W_{hom}$ is continuous on $\mathcal{N}$.
\end{lem}
\begin{proof}
Since the maximum of convex functions is convex and $(|Fv_1|^2-1)_{+},(|Fv_2|^2-1)_{+}$ are convex functions, it is only necessary to confirm that the third term is a convex function. This can be seen from the fact that $\chi(z)$ is nondecreasing for $z\geq 0$ and $F \mapsto |Fv|$ is convex for any $v\in\mathbb{R}^2$. 

If $F\in\mathcal{M}$ by symmetry we can assume that $F=R(I+\gamma v_1\otimes v_2)$. Then $|Fv_1|^2-1=0$, $|Fv_2|^2-1=\gamma^2$ and a short calculation shows that $\chi(\max\{|Fv_3|,|Fv_3^{\perp}|\})=\chi(\sqrt{1+|\gamma|+\gamma^2/2})=\gamma^2$, which implies that $f(F)=\gamma^2=W(F)$.
The coincidence of $f$ and $W_{hom}$ on $\mathcal{N}$ follows from the results in \cite{Conti2013} (Step 2 in the proof of Theorem 1.1). 
\end{proof}

\begin{proof}[Proof of Proposition \ref{liminf_affine}]
In the following we can assume that $u$ is affine in $\Omega$. If otherwise, we can slightly generalize $\Omega$ to be a cuboid instead of a cube, and apply the same argument to each affine part of $u$. Indeed, since by Remark \ref{independent-of-x_1} the limit is a function of $x_2$ only, each piecewise affine region is a cuboid. 

Since $\nabla u_{\epsilon}\rightharpoonup \nabla u\;$in $L^2(\Omega;\mathbb{R}^{2\times 2})$, and, by \eqref{lim-form.2}, $\nabla u_{\epsilon}\mathbbm{1}_{\epsilon Y_{rig}}\rightharpoonup (1-\lambda)R$ in $L^2(\Omega;\mathbb{R}^{2\times 2})$,
it follows that
\begin{equation}
\label{lim-of-Ysoft}
\nabla u_{\epsilon}\mathbbm{1}_{\epsilon Y_{soft}}\rightharpoonup \nabla u -(1-\lambda)R\,\quad \text{in} \; L^2(\Omega;\mathbb{R}^{2\times 2}).
\end{equation}

Define the sets $\Omega_{\epsilon}^{i} \,$and$\,\Omega_{\epsilon}^0$ via 
\begin{equation*}
\Omega_{\epsilon}^{i}:=(0,l)\times((i-1)\epsilon,i\epsilon )\;\;(i=1\;,\ldots,\; \lfloor l/ \epsilon \rfloor),\quad\Omega_{\epsilon}^0:=\Omega\backslash \bigcup_{i=1}^{\lfloor l/ \epsilon \rfloor}\,\Omega^{i}_{\epsilon} .
\end{equation*} 
Then $|\Omega_{\epsilon}^0|<l\epsilon$ and since $ |\int_{\Omega'} \nabla u_{\epsilon} \,  dx|\leq |\Omega'|^{1/2}\|\nabla u_{\epsilon}\|_{L^2(\Omega)}$ for $\Omega'\subset \Omega_{\epsilon}^0$, we deduce
\begin{equation}
\label{error-term}
\int_{\Omega^0_{\epsilon}\cap \epsilon Y_{soft}}\nabla u_{\epsilon} \, dx\rightarrow 0 \quad \text{ as }\epsilon \rightarrow 0.
\end{equation}
\sloppy In the same fashion, $\max\{(|\nabla u_{\epsilon}v_1|^2-1)_{+},(|\nabla u_{\epsilon}v_2|^2-1)_{+}\}\leq|\nabla u_{\epsilon}|^2$ and $\chi(\text{max} \{|\nabla u_{\epsilon}v_3|, |\nabla u_{\epsilon}v_3^{\perp})|\}) \leq 2|\nabla u_{\epsilon}|^2$ imply 
\begin{equation*}
\int_{\Omega^0_{\epsilon}}f(\nabla u_{\epsilon}) \, dx\leq \int_{\Omega^0_{\epsilon}}2\;|\nabla u_{\epsilon}|^2 dx\rightarrow 0 \quad \text{ as }\epsilon \rightarrow 0.
\end{equation*}
Considering that $f(\nabla u_{\epsilon})=0$ a.e. in $\epsilon Y_{rig}\cap \Omega$ and that $W(\nabla u_{\epsilon})=f(\nabla u_{\epsilon})$ a.e. in $\epsilon Y_{soft}\cap \Omega$ by Lemma \ref{continuity}, we have
\begin{align*}
\liminf_{\epsilon \rightarrow 0}E_{\epsilon}(u_{\epsilon})
& = \liminf_{\epsilon \rightarrow 0}\int_{\epsilon Y_{soft}\cap \Omega} \;W(\nabla u_{\epsilon})\; dx
=\liminf_{\epsilon \rightarrow 0}\int_{\Omega} f(\nabla u_{\epsilon})\,dx\\
&=\liminf_{\epsilon \rightarrow 0}\left(\sum_{i=1}^{\lfloor l/ \epsilon \rfloor}\int_{\Omega^{i}_{\epsilon}\cap \;\epsilon Y_{soft}} f(\nabla u_{\epsilon}) \,dx+\int_{\Omega^0_{\epsilon}}f(\nabla u_{\epsilon})\,dx\right)\\
&\geq \liminf_{\epsilon \rightarrow 0}\,\lambda\epsilon l\sum_{i=1}^{\lfloor l/ \epsilon \rfloor}f\left(\frac{1}{\lambda \epsilon l}\int_{\Omega^{i}_{\epsilon}\cap \;\epsilon Y_{soft}} \nabla u_{\epsilon}\,dx\right)\\
&\geq \liminf_{\epsilon \rightarrow 0}\,\lambda\epsilon l\cdot \lfloor  l/ \epsilon\rfloor \cdot f\left(\frac{1}{\lfloor  l/ \epsilon\rfloor}\sum_{i=1}^{\lfloor l/ \epsilon \rfloor}\frac{1}{\lambda \epsilon l}\int_{\Omega^{i}_{\epsilon}\cap \;\epsilon Y_{soft}} \nabla u_{\epsilon}\,dx\right),
\end{align*}
where we have used the continuous and discrete form of Jensen's inequality in the last two steps, respectively.
From \eqref{error-term}, we may continue as
\begin{align*}
&= \liminf_{\epsilon \rightarrow 0}\,\lambda|\Omega|\cdot f\left(\frac{1}{\lfloor  l/ \epsilon\rfloor}\,\frac{1}{\lambda \epsilon l}\int_{\Omega\cap \;\epsilon Y_{soft}} \nabla u_{\epsilon}\,dx\right)\\
&=\lambda|\Omega|\cdot f\left(\frac{1}{\lambda }(\nabla u-(1-\lambda)R)\right)\\
&=\lambda|\Omega|\cdot W_{hom}\left(\frac{1}{\lambda }(\nabla u-(1-\lambda)R)\right) .
\end{align*}
Here, we employed \eqref{lim-of-Ysoft} and the continuity of $f$ in the second equality.
\end{proof}

We proceed to the proof of the liminf inequality.
Let $\Omega$ be a cube, that is, $\Omega = Q = (0, l)^2$ for $ l > 0$.  The proof for the general case is done by approximating $\Omega$ by a finite number of cubes and taking the supremum in the resulting liminf inequality over all such approximations. 

Let $(\epsilon_j)_j$ fulfill $\epsilon_j\rightarrow 0$ as $j\rightarrow \infty$, and let $(u_j)_j$ be a sequence  with $u_j\rightarrow u$ as $j\rightarrow \infty$  in $L^2(\Omega;\mathbb{R}^2)$ such that $(E_{\epsilon_j}(u_j))_j$ is bounded.
Then, similarly as in the compactness proof (Section \ref{Compactness}), the limit $u\in W^{1,2}(\Omega;\mathbb{R}^2)$ satisfies $\nabla u=R(I+\gamma e_1\otimes e_2)$ for some $R\in SO(2)$ and $\gamma\in L^2(\Omega)$.

If $\gamma$ is piecewise constant, Corollary \ref{liminf_affine} implies the liminf inequality, i.e.,
\begin{equation*}
\liminf_{j\rightarrow \infty} E_{\epsilon_j}(u_{j}) \geq \int_{\Omega}W_{hom}\left(\frac{1}{\lambda}(\nabla u-(1-\lambda)R)\right) \,dx = E(u).
\end{equation*}
The general case $\gamma\in L^2(\Omega)$ can be reduced to the previous one through approximating $\gamma$ by piecewise constant functions. Since $\gamma$ is essentially a function of $x_2$ only (cf. Remark \ref{independent-of-x_1}), one can approximate $\gamma$ in one-dimension and extend it constantly in the $x_1$-direction.
Let $\gamma^{(\epsilon)}\in C^{\infty}_0(0,l)$ be such a one-dimensional proxy for $\gamma$ fulfilling $\|\gamma^{(\epsilon)}-\gamma\|_{L^2(\Omega)}\leq \epsilon$, and $(\zeta_k^{(\epsilon)})_k\subset L^2(\Omega)$ be a approximation by simple functions for $\gamma^{(\epsilon)}$ such that $\|\zeta^{\epsilon}_{k}-\gamma^{(\epsilon)}\|<1/k$. 
By a diagonal argument we obtain a subsequence $\zeta_k:=\zeta_k^{(\epsilon_k)}$ that satisfies 
\begin{equation*}
\zeta_k\rightarrow \gamma\quad\text{in}\,L^2(\Omega)\quad \text{as} \;\;k\rightarrow \infty, \qquad \zeta_k=\sum^{n_k}_{i=1}\zeta_{k,i}\mathbbm{1}_{(t_{k,i-1},t_{k,i})},
\end{equation*}
where $(t_{k,i-1},t_{k,i})_i$ are divisions $0=t_{k,0}<t_{k,1}<\cdots< t_{k,n_k}=l$ of the interval $(0,l)$. This approximation can be constructed so that the subdivisions are nested.
For $\zeta_k$ obtained in this way, let $w_k\in W^{1,2}(\Omega;\mathbb{R}^2)\cap L^2_0(\Omega;\mathbb{R}^2)$ be given by $\nabla w_k=R(I+\zeta_k e_1\otimes e_2)$ for
$k\in\mathbb{N}$.

Adapting the method from \cite{Muller1987}, we construct $(\vn_j)_j,(\vn_{k,j})_j\subset W^{1,2}\cap L^2_0(\Omega;\mathbb{R}^2)$ weakly converging to $u$ and $w_k$, respectively:
\begin{equation}
\label{weak}
\vn_j\rightharpoonup u,\quad \vn_{k,j}\rightharpoonup w_{k} \;\; \forall k, \quad\text{both in }W^{1,2}(\Omega;\mathbb{R}^2)\quad \text{as} \;\; j\rightarrow \infty.
\end{equation}
Furthermore, we require that $(\vn_j)_j,(\vn_{k,j})_j$ are defined so that 
\begin{equation}
\label{equality-in-rigid-layer}
\nabla \vn_{k,j}=\nabla \vn_j\quad \text{in}\;\epsilon_j Y_{rig}\cap \Omega \quad \text{for all} \; j,k\in \mathbb{N},
\end{equation}
and so that there is a nondecreasing sequence of natural numbers $(k(j))_j$ diverging to $\infty$ as $j\rightarrow \infty$, for which
\begin{equation}
\label{boundedness-of-error}
\|\nabla \vn_{k,j}-\nabla \vn_j\|_{L^2(\Omega;\mathbb{R}^{2\times 2})}\leq \frac{1}{\lambda} \|\zeta_k -\zeta_{k(j)}\|_{L^2(\Omega)}+\frac{c}{\sqrt{k(j)}}
\end{equation}
holds for all $k\in\mathbb{N}$ and all $j\geq j_0$, where $j_0$ may depend on $k$. Here, $c$ is a constant independent of $j,k$. 

To satisfy these conditions, we define $(\vn_{k,j})_j$ with zero mean by
\begin{equation}
\label{nablavkj}
\nabla \vn_{k,j}=R+\sum_{i=1}^{n_k}(N_{k,i}-R)\mathbbm{1}_{\epsilon_j Y_{soft}\cap\Omega}
\mathbbm{1}_{(0,l)\times(\lceil\epsilon^{-1}_j t_{k,(i-1)}\rceil\epsilon_j,\lfloor\epsilon^{-1}_j t_{k,i}\rfloor\epsilon_j)},
\end{equation}
with $N_{k,i}\in \mathcal{N}$ defined from $\zeta_{k,i}$ by
\begin{equation}
\label{Nki}
\lambda N_{k,i}+(1-\lambda)R=R(I+\zeta_{k,i} e_1\otimes e_2).
\end{equation}
Note that $\vn_{k,j}$ is well defined by its gradient since $N_{k,i}$ and $R$ rank-one connect along all segments $[0,l] \times m \epsilon_j$, $m=1, \dots, \lfloor l/\epsilon_j \rfloor$ (see Fig. \ref{recovery-sequence}).

Thanks to the averaging lemma on weak convergence of periodic functions, combined with the fact that the measure of the region around the lines $[0,l]\times t_{k,i}$, where $\nabla \vn_{k,j}$ deviates from a periodic function, vanishes for $j\to\infty$ and fixed $k$, we have 
\begin{equation*}
\nabla \vn_{k,j}\rightharpoonup \sum_{i=1}^{n_k}(\lambda N_{k,i}+(1-\lambda)R)\mathbbm{1}_{(0,l)\times (t_{k,i-1},t_{k,i})}=\nabla w_k \quad\text{in} \;\; L^2(\Omega;\mathbb{R}^{2\times 2})
\end{equation*}
for every $k\in\mathbb{N}$ as $j\to\infty$.
Thus the second condition in \eqref{weak} is fulfilled.

\begin{figure}[!ht]
\begin{center}
\includegraphics[width=0.9\textwidth]{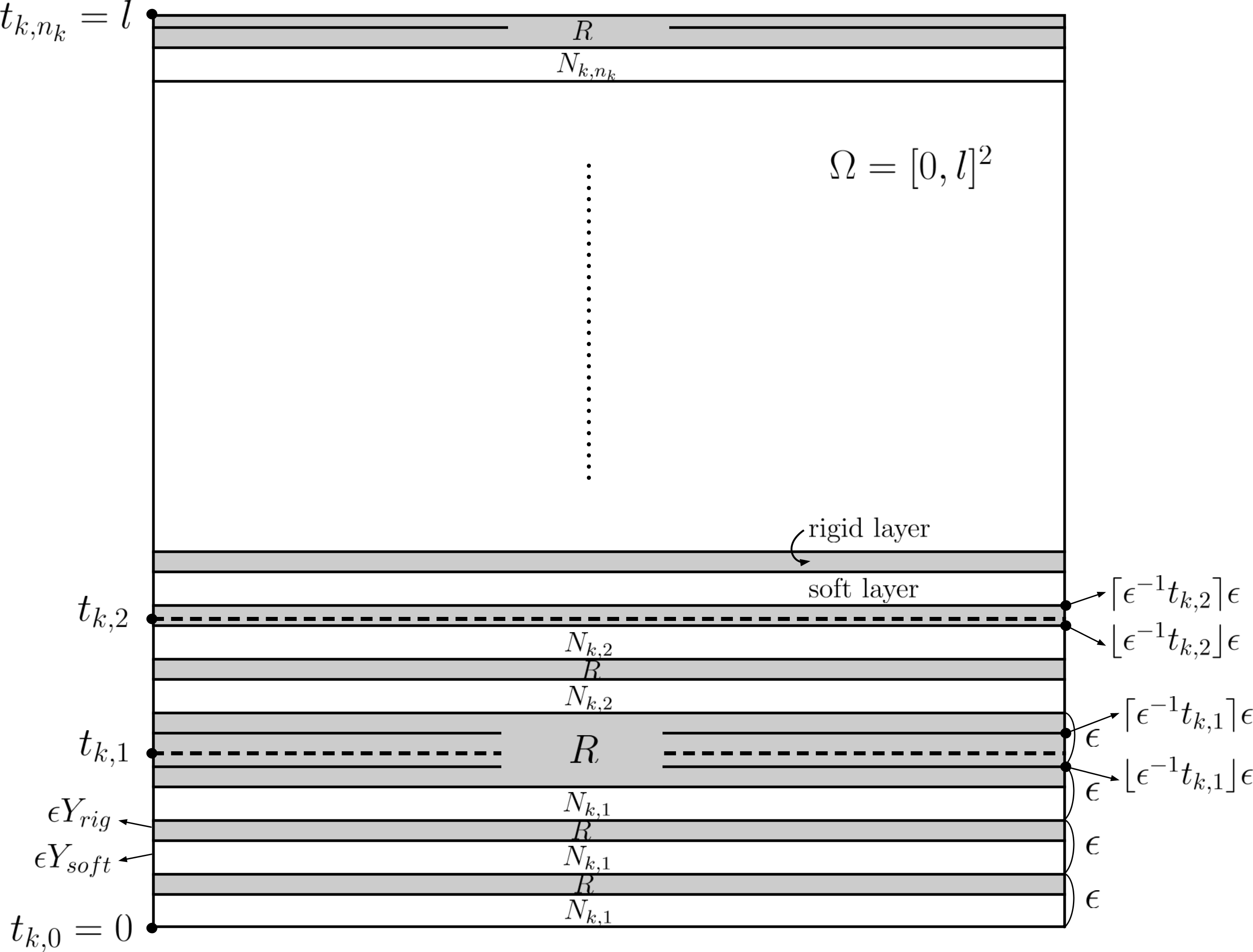}
\end{center}
\caption{A sketch of the function $\nabla \vn_{k,j}$ defined in \eqref{nablavkj}. For simplicity we omit the index $j$ in $\epsilon_j$.
As shown by gray shading, the value of $\nabla \vn_{k,j}$ becomes $R$ not only in rigid layers but also in certain soft layers, namely those which contain a nodal point $t_{k,i}$.}
\label{recovery-sequence}
\end{figure}

The idea to realize also the first condition in \eqref{weak}, together with \eqref{equality-in-rigid-layer} and \eqref{boundedness-of-error} is to use a diagonal argument to define $\vn_j:=\vn_{k(j),j}$. Then  \eqref{equality-in-rigid-layer} trivially holds. 
Furthermore, to achieve \eqref{boundedness-of-error}, we set $\tau_k:=\min_h (t_{k,h}-t_{k,h-1})$ and estimate
\begin{align*}
& \|\nabla \vn_{k,j}-\nabla  \vn_{j}\|_{L^2(\Omega;\mathbb{R}^{2\times 2})} = \|\nabla \vn_{k,j}-\nabla  \vn_{k(j),j}\|_{L^2(\Omega)} \\
& \; \leq \,\left\| \sum_{i=1}^{n_k}\sum_{h=1}^{n_{k(j)}}(N_{k,i}-N_{k(j),h})\mathbbm{1}_{(0,l)\times(t_{k(j),h-1},t_{k(j),h})}\mathbbm{1}_{(0,l)\times(t_{k,i-1},t_{k,i})}\right\|_{L^2(\Omega)}\\
& \;\quad +\sqrt{\frac{2\epsilon_{j}}{\tau_{k(j)}}}\left\|\sum_{i=1}^{n_k}\sum_{h=1}^{n_{k(j)}}(N_{k,i}-R)\mathbbm{1}_{(0,l)\times(t_{k(j),h-1},t_{k(j),h})}\mathbbm{1}_{(0,l)\times(t_{k,i-1},t_{k,i})}\right\|_{L^2(\Omega)}.
\end{align*} 
The second term accounts for the discrepancy between $\nabla \vn_{k,j}$ and $\nabla \vn_{k(j),j}$ possibly occurring in an $\epsilon_j$-neighborhood of the jump lines $t_{k(j),h}$, where we assume that $j$ is large enough so that $k(j)>k$. 
Note that since the divisions $(t_{k,i})_i$ are nested when $k$ increases, the subintervals corresponding to $\zeta_{k(j)}$ are then a subdivision of those corresponding to $\zeta_k$. Accordingly,
\begin{align}
& \|\nabla \vn_{k,j}-\nabla  \vn_{j}\|_{L^2(\Omega;\mathbb{R}^{2\times 2})} \notag \\
& \; \leq \frac{1}{\lambda}\;\left\|\sum_{i=1}^{n_k}\sum_{h=1}^{n_{k(j)}}(\zeta_{k,i}-\zeta_{k(j),h})\mathbbm{1}_{(0,l)\times(t_{k(j),h-1},t_{k(j),h})}\mathbbm{1}_{(0,l)\times(t_{k,i-1},t_{k,i})}\right\|_{L^2(\Omega)} \notag \\
& \;\quad + \frac{1}{\lambda} \sqrt{\frac{2\epsilon_{j}}{\tau_{k(j)}}}\left\|\sum_{i=1}^{n_{k}}\zeta_{k,i}\mathbbm{1}_{(0,l)\times(t_{k,i-1},t_{k,i})}\right\|_{L^2(\Omega)} \notag \\
\label{difnablav}
& \; =\frac{1}{\lambda}\;\|\zeta_k-\zeta_{k(j)}\|_{L^2(\Omega)}+ \frac{1}{\lambda} \sqrt{\frac{2\epsilon_{j}}{\tau_{k(j)}}}\|\zeta_{k}\|_{L^2(\Omega)}.
\end{align}
Therefore, to have \eqref{weak} and \eqref{boundedness-of-error}, it is enough to find $(k(j))_j$ so that
\begin{equation*} \frac{\epsilon_{j}}{\tau_{k(j)}} \leq \frac{1}{k(j)} \;\; \forall j \qquad \text{and} \qquad \vn_{k(j),j} \rightharpoonup u \quad\text{in} \; W^{1,2}(\Omega;\mathbb{R}^2)\;\; \text{as} \; j\rightarrow \infty.\end{equation*}
We next construct such a sequence $(k(j))_j$.
Noting that $(\vn_{k,j})_{k,j}$ and $(w_k)_k$ are uniformly bounded in $W^{1,2}(\Omega;\mathbb{R}^2)$, we let $d$ denote a distance metrizing  the weak topology of $W^{1,2}(\Omega;\mathbb{R}^2)$ in some closed sphere containing $(\vn_{k,j})_{k,j}$, $(w_k)_k$ and $u$.
Taking into account the weak convergence of $\vn_{p,j}$ to $w_p$ and the fact that $(\epsilon_j/\tau_p)_j$ is a decreasing sequence for fixed $p$, we can find an increasing sequence $(j(p))_p\subset\mathbb{N}$ such that $j(1)=1$ and for $p\geq 2,$
\begin{equation*}d\left( \vn_{p,j},w_p \right) \leq \frac{1}{p} \quad \text{and} \quad \frac{\epsilon_j}{\tau_p} \leq \frac{1}{p} \qquad \text{when} \;\; j \geq j(p) .\end{equation*}
Then we define $k(j)$ for $j\in\mathbb{N}$ as follows:
\begin{equation*} k(j)=p \qquad \text{if} \;\; j(p) \leq j \leq j(p+1)-1 . \end{equation*}

For any $\delta>0$, take $p\in \mathbb{N}$ sufficiently large so that $1/p< \min\{1,\delta\}$ and so that for any $k\geq p$ one has $d(w_k,u)<\delta$. Then $j\geq j(p)$ implies
\begin{equation*}
d(\vn_{k(j),j},u)\leq d(\vn_{k(j),j},w_{k(j)}) + d(w_{k(j)},u)<2\delta ,
\end{equation*}
and thus $\vn_j=\vn_{k(j),j}\rightharpoonup u$ as $j\rightarrow \infty$. In addition, by construction of $(k(j))_j$, 
$\epsilon_j/\tau_{k(j)}<1/k(j)$ for $j\geq j(2)$, which in combination with \eqref{difnablav} yields \eqref{boundedness-of-error}.

Now we set
\begin{equation*}
z_{k,j}=u_j-\vn_j+\vn_{k,j}
\end{equation*}
for $j,k\in\mathbb{N}$. By \eqref{equality-in-rigid-layer}, $f(\nabla z_{k,j})=0$ a.e. in$\,\epsilon_j Y_{rig}\cap \Omega$, where $f$ is the function defined in Lemma \ref{continuity}. By Proposition \ref{Asymptotic-rigidity}, for each $k\in\mathbb{N}$,
\begin{equation*}
\nabla z_{k,j}\mathbbm{1}_{\epsilon_j Y_{rig}}\rightharpoonup (1-\lambda)R\quad\text{in}\,L^2(\Omega;\mathbb{R}^{2\times 2})\quad \text{as} \; j\rightarrow \infty.
\end{equation*}
Also, thanks to \eqref{weak}, $z_{k,j}$ weakly converges to the piecewise affine function $w_k$ in $W^{1,2}(\Omega;\mathbb{R}^2)$ as $j\rightarrow\infty$. 

We now show the liminf inequality using the proof of Corollary \ref{liminf_affine} and the following lemma, whose proof can be found in Appendix \ref{prooflemFAD}. 

\begin{lem}
\label{lemFAD}
Let $F=A+D$, where $D=R\gamma_1 e_1\otimes e_2$ and either $A=R(I+\gamma_2 v_2\otimes v_1)$ or $A=R(I+\gamma_2 v_1\otimes v_2)$ for some $\gamma_1, \gamma_2$ and $R\in SO(2)$.
Then there exists a constant $c$, such that
\begin{equation*} 
f(F) \leq |F|^2 - 2 +c\left(\sqrt{|D|}+|D|\right)\left( \sqrt{|A|}+ |A|+\sqrt{|D|}+|D|\right) .
\end{equation*}
\end{lem}

In view of \eqref{nablavkj}, we know that $\nabla \vn_{k,j}$ is equal to $R$ in rigid layers and either $N_{k,i}=R(I+\frac{\zeta_{k,i}}{\lambda} e_1\otimes e_2)$ or $R$ in soft layers. 
Since $\vn_j = \vn_{k(j),j}$ is a subsequence, it satisfies the same property and thus
\begin{equation*} \nabla \vn_j - \nabla \vn_{k,j} = \left\{ \begin{array}{ll} 0 \qquad & \text{in} \; \Omega \cap \epsilon_j Y_{rig} \\ R\frac{\gamma}{\lambda} e_1\otimes e_2  \qquad & \text{in} \; \Omega \cap \epsilon_j Y_{soft} \end{array} \right.\end{equation*}
for some $\gamma\in\mathbb{R}$.

Hence, we may apply Lemma \ref{lemFAD} to $F:=\nabla z_{k,j}$ decomposed as $F=A+D$, where $A:=\nabla u_j \in \mathcal{M}$ and $D:=\nabla \vn_{k,j}-\nabla \vn_j$. 
Combined with H\"older's inequality, with the boundedness of $ \| \nabla u_j \|_{L^2(\Omega; \, \mathbb{R}^{2\times 2})}$, $\| \zeta_k \|_{L^2(\Omega)}$ and with the estimate \eqref{boundedness-of-error}, this yields 
\begin{align*}
& \int_{\Omega} f(\nabla z_{k,j}) \, dx \leq \int_{\Omega} \left( | \nabla z_{k,j}|^2 -2 \right) \, dx + c \| \nabla \vn_{k,j}-\nabla \vn_j \|_{L^2(\Omega; \, \mathbb{R}^{2\times 2})}^{1/2}  \\
& \qquad \leq \int_{\Omega} \left( | \nabla z_{k,j}|^2 -2 \right) \, dx + c \left( \|\zeta_k-\zeta_{k(j)}\|_{L^2(\Omega)}+\frac{1}{\sqrt{k(j)}}\right)^{1/2} .
\end{align*}

Using triangle inequality and the uniform boundedness of $\|\nabla \vn_{j}-\nabla \vn_{k,j}\|_{L^2}$, $\|\nabla z_{k,j}\|_{L^2}$ and $\|\zeta_k-\zeta_{k(j)}\|_{L^2}$, we get 
\begin{align*}
\liminf_{j\rightarrow \infty}E_{\epsilon_j}(u_j)&=\liminf_{j\rightarrow \infty}\int_{\Omega} \left( |\nabla z_{k,j}+(\nabla \vn_{j}-\nabla \vn_{k,j})|^2-2 \right) \,dx  \\
&\geq\liminf_{j\rightarrow \infty} \left( \int_{\Omega} \left( |\nabla z_{k,j}|^2-2 \right)\, dx - \,c\,\|\nabla \vn_{j}-\nabla \vn_{k,j}\|_{L^2(\Omega; \, \mathbb{R}^{2\times 2})} \right) \\
&\geq\liminf_{j\rightarrow \infty}  \left(\int_{\Omega} f(\nabla z_{k,j}) \, dx -  c \left( \|\zeta_k-\zeta_{k(j)}\|_{L^2(\Omega)}+\frac{1}{\sqrt{k(j)}}\right)^{1/2}\right)  \\
&\geq\liminf_{j\rightarrow \infty} \int_{\Omega} f(\nabla z_{k,j}) \, dx -  c \| \zeta_k-\gamma \|_{L^2(\Omega)}^{1/2} \\
&\geq E(w_k)- c \| \zeta_k-\gamma \|_{L^2(\Omega)}^{1/2} ,
\end{align*}
where the last estimate can be shown as in the proof of Corollary \ref{liminf_affine} due to the weak convergence of $(z_{k,j})_j$ to the piecewise affine function $w_k$ together with the convergence $\nabla z_{k,j}\mathbbm{1}_{\epsilon_j Y_{rig}}\rightharpoonup (1-\lambda)R$.
Therefore, taking limes inferior as $k\rightarrow \infty$, we arrive at
\begin{equation*}
\liminf_{j\rightarrow \infty}E_{\epsilon_j}(u_j)\geq E(u)
\end{equation*}
because $E(w_k)=\int_{\Omega} f(\nabla w_k) \, dx$, which is a lower semicontinuous functional thanks to the convexity of $f$ \cite{Benesova2017}. 
\qed

\noindent
\subsection{Recovery sequence}
\label{Recovery sequence}

Let $u\in W^{1,2}(\Omega;\mathbb{R}^2)\cap L^2_{0}(\Omega;\mathbb{R}^2)$ be such that $\nabla u=R(I+\gamma e_1\otimes e_2)=\lambda N +(1-\lambda)R$ with $R\in SO(2)$, $N\in\mathcal{N}$ and $\gamma\in L^2(\Omega).$
The goal of this section is to show the existence of a sequence $(u_j)_j \subset L_0^2(\Omega;\mathbb{R}^2)$ such that $u_j\to u$ in $L^2(\Omega;\mathbb{R}^2)$ and $E_{\epsilon_j}(u_j)\to E(u)$.
The idea of constructing such a sequence would be to take a function $w\in W^{1,\infty}\cap L^{\infty}_0 (Y;\mathbb{R}^2)$ with 
$\nabla w=R\mathbbm{1}_{ Y_{rig}}+N\mathbbm{1}_{Y_{soft}}$ in $Y$, extend it periodically and set $\nabla u_{\epsilon}(x)=\nabla w(\frac{x}{\epsilon})$ for $x\in\Omega$. Then the laminate $\nabla u_{\epsilon}$ converges weakly to $\lambda N +(1-\lambda)R$, as desired. However, since $N$ may not belong to the admissible set $\mathcal{M}$, this construction needs to be refined by grafting into the soft layer another laminate between two matrices from $\mathcal{M}$ to replace the matrix $N$ (see Fig. \ref{simplelaminate}).

To realize this goal, we first prove in Section \ref{sec_rs1} that it is in fact sufficient to construct recovery sequences for piecewise affine limit functions $u$, i.e., piecewise constant functions $\gamma$.
Then we address the construction of suitable laminates approximating $N$ in Section \ref{sec_rs2} and, finally, construct recovery sequences for constant and piecewise constant functions $\gamma$.

\subsubsection{Recovery sequence for general $\gamma \in L^2(\Omega)$}
\label{sec_rs1}

\begin{lem}
\label{generalization}
Let $(E_{\epsilon})_{\epsilon}:W^{1,2}\cap L_0^2(\Omega;\mathbb{R}^{2})\rightarrow [0,\infty]$ be a functional family and let $E:L_0^2(\Omega;\mathbb{R}^2)\rightarrow [0,\infty]$ have the form
\begin{equation*}
E(u)=
\begin{cases}
\displaystyle\int_{\Omega}W_{hom}(\nabla u) \, dx & \text{if}\;\;
u\in K, \\ 
\infty &  \text{otherwise}.
\end{cases}
\end{equation*}
Here, $K$ is a subset of $W^{1,2}\cap L_0^{2}(\Omega;\mathbb{R}^2)$ such that the set of finitely piecewise affine functions belonging to $K$ is dense in $K$.
Further, the integrand $W_{hom}:\mathbb{R}^{2\times 2}\rightarrow [0,\infty]$ is a continuous function on $\mathcal{L}:=\{\nabla u(x): \, u\in K, x\in \Omega\}$ with the growth $W_{hom}(A)\leq \mu|A|^2 + c$ for some $\mu, \,c>0$ and all $A\in\mathcal{L}$.

If for any finitely piecewise affine $u\in K$ there exists $(u_j)_j \subset W^{1,2}\cap L_0^{2}(\Omega;\mathbb{R}^2)$ such that $u_j \rightharpoonup u$ in $W^{1,2}(\Omega;\mathbb{R}^2)$ and $E_{\epsilon_j}(u_j)\rightarrow E(u)$, then such a sequence exists for any $u \in K$.
\end{lem}

\begin{proof}
For a given $u\in K$, let $(w_k)_k\subset K\,\cap\, W^{1,\infty}(\Omega;\mathbb{R}^2)$ be a sequence of finitely piecewise affine functions such that $w_k\rightarrow u$ in $W^{1,2}(\Omega;\mathbb{R}^2)$ as $k\rightarrow \infty$.
We show that 
\begin{equation}
\label{energy-converge-Linfty}
\lim_{k\rightarrow\infty}E(w_k)=E(u).
\end{equation}

For $\alpha$ with $0<\alpha\leq 1$ we define 
\begin{equation}
\label{subset-of-Omega}
\Omega^{\alpha}_{k}:=\{x\in\Omega :\,|(\nabla w_k-\nabla u)(x)|>\alpha\} \;\; \text{ and } \;\; \Omega^{\alpha,0}_{k}:=\Omega\backslash \Omega^{\alpha}_k.
\end{equation}
Then 
\begin{equation}
\label{containment}
\lim_{k\rightarrow \infty}|\Omega^{\alpha}_k|=0.
\end{equation}
Now, let $k\in \mathbb{N}$ be fixed and consider $u$ as a function over $\Omega^{1,0}_k$.
By assumption, $W_{hom}$ is uniformly continuous on the set $\mathcal{L} \,\cap\, \{F\in \mathbb{R}^{2\times 2}:|F|\leq \|\nabla u\|_{L^{\infty}(\Omega^{1,0}_k;\mathbb{R}^{2\times 2})}+1\}$. 
Note that by \eqref{subset-of-Omega}, $\nabla u(x)$ and $\nabla w_{k'}(x)$ belong to this set for all $x \in \Omega^{1,0}_k\cap \Omega^{\alpha,0}_{k'}$ and for all $k'\in\mathbb{N}$.
Therefore, for any $\varepsilon>0$ one can find a sufficiently small $\alpha(k,\varepsilon)$ such that for every $k'\in\mathbb{N}$,
\begin{equation*}
|W_{hom}(\nabla w_{k'})-W_{hom}(\nabla u)|<\varepsilon\quad \text{in}\;\Omega^{1,0}_k \cap \Omega^{\alpha,0}_{k'}.
\end{equation*}
It follows that
\begin{equation}
\label{first-converging}
\int_{\Omega^{1,0}_k\cap\, \Omega^{\alpha,0}_{k'}}W_{hom}(\nabla w_{k'}) \, dx <\varepsilon |\Omega^{1,0}_k\cap \Omega^{\alpha,0}_{k'}| + \int_{\Omega^{1,0}_k\cap\, \Omega^{\alpha,0}_{k'}}W_{hom}(\nabla u) \, dx.
\end{equation}
Hence, by the growth assumption on $W_{hom}$,
\begin{equation}
\label{Ewkp}
E(w_{k'}) 
 \leq \int_{\Omega^1_k\cup\,(\Omega^{1,0}_k\cap \,\Omega^{\alpha}_{k'})} \left( \mu |\nabla w_{k'}|^2+c \right) \, dx + \int_{\Omega^{1,0}_k\cap\, \Omega^{\alpha,0}_{k'}}W_{hom}(\nabla w_{k'}) \, dx .
\end{equation}
The convergence $\nabla w_{k'}\rightarrow \nabla u$ in $L^2(\Omega;\mathbb{R}^{2\times 2})$ together with \eqref{containment} imply that the first term on the right side converges to $\int_{\Omega^1_k} (\mu |\nabla u|^2+c )\, dx$ when $k'\rightarrow \infty$.
Similarly, in view of \eqref{first-converging},
\begin{equation*}
\limsup_{k'\rightarrow \infty} \int_{\Omega^{1,0}_k\cap\, \Omega^{\alpha,0}_{k'}}W_{hom}(\nabla w_{k'}) \, dx \leq \varepsilon|\Omega ^{1,0}_k| + \int_{\Omega^{1,0}_k} W_{hom}(\nabla u)\,dx ,
\end{equation*}
so that \eqref{Ewkp} yields 
\begin{equation*}
\limsup_{k'\rightarrow \infty} E(w_{k'})\leq \int_{\Omega^1_k} (\mu |\nabla u|^2+c )\, dx + \varepsilon|\Omega| + \int_{\Omega} W_{hom}(\nabla u)\,dx.
\end{equation*}
Taking $\varepsilon\to 0+$ and $k\rightarrow \infty$, we get from \eqref{containment},
\begin{equation*}
\limsup_{k'\rightarrow \infty} E(w_{k'})\leq \int_{\Omega} W_{hom}(\nabla u)\,dx.
\end{equation*}
An analogous argument, exploiting the nonnegativity of $W_{hom}$ in place of the growth condition, leads to the opposite inequality for the lower limit, and we conclude
\begin{equation}
\label{limEwp}
\lim_{k'\rightarrow \infty}E(w_{k'})=E(u).
\end{equation}

\sloppy By assumption, we can obtain a recovery sequence $(w_{k,j})_j$ for $w_k$ such that $\lim_{j\rightarrow 0}E_{\epsilon_j}(w_{k,j})=E(w_k)$ and $w_{k,j}\rightharpoonup w_k$ in $W^{1,2}(\Omega;\mathbb{R}^2)$ as $j\to\infty$ 
for all $k\in \mathbb{N}$. By Rellich-Kondrachov theorem, $(w_{k,j})_j$ can be assumed to converge to $w_k$ strongly in $L^2(\Omega;\mathbb{R}^2)$. Combining this with $\lim_{k\to\infty}\|w_k-u\|_{L^2(\Omega;\mathbb{R}^2)}=0$ and with \eqref{energy-converge-Linfty}, we have
\begin{equation*}
\lim_{k\rightarrow\infty}\lim_{j\rightarrow 0} \left( \|w_{k,j}-u\|_{L^2(\Omega;\mathbb{R}^2)}+|E_{\epsilon_j}(w_{k,j})-E(u)| \right) =0.
\end{equation*}
A simplified version of the diagonalization argument introduced in the proof of the liminf inequality (see also \cite{Attouch1984}, Corollary 1.18), yields the existence of a sequence $(u_{j})_j$ such that $u_{j}=w_{k(j),j}$ and 
\begin{equation*}
u_{j}\rightarrow u \;\; \text{in}\; L^2(\Omega;\mathbb{R}^2)\qquad\text{and}\qquad E_{\epsilon_j}(u_{j})\rightarrow E(u)
\end{equation*}
as $j\rightarrow 0$. It remains to show that $u_j$ converges to $u$ weakly in $W^{1,2}(\Omega;\mathbb{R}^2)$. This follows by noting that any subsequence of $(u_j)_j$ contains a subsequence weakly converging to $u$.
\end{proof}

We verify that $W _{hom}$ given by \eqref{gamma-limit-energydensity} satisfies the assumptions of Lemma \ref{generalization}.
Indeed, since $K=\{u\in W^{1,2}\cap L_0^{2}(\Omega;\mathbb{R}^2):\,\nabla u=R(I+\gamma e_1\otimes e_2),R\in SO(2),\gamma\in L^2(\Omega)\}$, taking an arbitrary $u\in K$ we see that $\nabla u=R(I+\gamma e_1\otimes e_2)$ for some $R\in SO(2)$ and this $\gamma$ is independent of $x_1$ due to $0=\operatorname{curl} \nabla u = - \frac{\partial \gamma}{\partial x_1} Re_1$ (see Remark \ref{independent-of-x_1}). Hence we may use a one-dimensional approximation of $\gamma$ by a sequence of simple functions $(\zeta_k)_k$, i.e., $\zeta_k\rightarrow \gamma$ in $L^{2}(\Omega)$. Then the vanishing curl of $R(I+\zeta_k e_1\otimes e_2)$ guarantees the existence of a piecewise affine $w_k\in W^{1,2}(\Omega;\mathbb{R}^2)$ with $\nabla w_k=R(I+\zeta_k e_1\otimes e_2)$ and with mean value zero.
In combination with  Poincar\'{e}'s inequality we have $w_k\rightarrow u$ in $W^{1,2}(\Omega;\mathbb{R}^{2\times 2})$, as desired. 
Moreover, the continuity of $W_{hom}$ on $\mathcal{L}$ follows from Lemma \ref{continuity} and the relation $\mathcal{L} \subset \mathcal{N}$.

\begin{rem}
\label{rem_estimateWhom}
For the specific form of energy that we have, one can directly show that it is locally Lipschitz continuous in $\gamma$ (see Appendix \ref{appendix_Whom} for details). Namely, there is a constant $c$ depending only on $\lambda$ such that
\begin{equation*} 
|W_{hom}(\gamma_1) - W_{hom}(\gamma_2)| \leq c(1+|\gamma_1|+|\gamma_2|) |\gamma_1-\gamma_2| \qquad \forall \gamma_1, \gamma_2 \in \mathbb{R},
\end{equation*}
where we abuse notation to write $W_{hom}(\gamma)$ instead of $W_{hom}(\nabla u)$ when $\nabla u = R(I+\gamma e_1\otimes e_2)$.
This fact allows for a simpler proof in the sense that the argument leading to \eqref{limEwp} can be skipped. 
\end{rem}

\subsubsection{Laminate constructions}
\label{sec_rs2}

Thanks to Lemma \ref{generalization}, it suffices to address the case of piecewise affine functions $u$, or equivalently, piecewise constant functions $\gamma$.
As mentioned above, the underlying idea requires the construction of simple laminates between matrices with finite energy $E_{\epsilon}$. For this purpose, in this section we first prove that the domain $\mathcal{N}$ of the homogenized functional $E$, i.e., the set of matrices with unit determinant, is contained in the rank-one convex hull of $\mathcal{M}$, that is, any $N\in\mathcal{N}$ can be represented by a convex combination of matrices in $\mathcal{M}$ that are rank-one connected. In this way, we can construct a simple laminate whose gradient lies in $\mathcal{M}$ and which, with period approaching to zero, weakly converges to a given $N\in\mathcal{N}$.

\begin{lem}
\label{rank-one-connection}
For a given $N\in\mathcal{N}$, there are $F_{+},\;F_{-} \in \mathcal{M}$ satisfying
$\operatorname{rank} (F_+-F_-)=1$ and $\mu \in [0,1]$ such that
\begin{equation}
\label{NFpFm}
N=\mu F_{+} +(1-\mu) F_{-}
\end{equation}
and
\begin{enumerate}[(i)]
\item if $\,\min\{|Nv_1|, |Nv_2|\}>1$ and $Nv_1\cdot Nv_2>0$ then \\
$N(v_1+v_2)=F_{+}(v_1+v_2)=F_{-}(v_1+v_2)$; 
\item if $\,\min\{|Nv_1|, |Nv_2|\}>1$ and $Nv_1\cdot Nv_2<0$ then \\
$N(v_1-v_2)=F_{+}(v_1-v_2)=F_{-}(v_1-v_2)$;
\item if $\,|Nv_1|\leq 1$ then $Nv_2=F_{+}v_2=F_{-}v_2$;
\item if $\,|Nv_2|\leq 1$ then $Nv_1=F_{+}v_1=F_{-}v_1$.
\end{enumerate}
\end{lem}
\begin{proof}
First, let us confirm that the above cases (i)--(iv) cover all possibilities, namely, that $\min\{|Nv_1|, |Nv_2|\}>1$ and $Nv_1\cdot Nv_2=0$ cannot hold at the same time. Indeed, if $Nv_1\cdot Nv_2=0$ then $Nv_1$ and $Nv_2$ are orthogonal, and since $\det N=|Nv_1||Nv_2|=1$, we get $\min\{|Nv_1|, |Nv_2|\}\leq 1$, which contradicts our assumption.

To begin with, we focus on the case (i) and introduce the function $\varphi(t), t\in\mathbb{R}$ following the idea of \cite{Conti2013}:
\begin{equation*}
\varphi(t):=|N_t(v_1-v_2)|^2-|N_t(v_1+v_2)|^2,\quad \text{where} \;\; N_t:=N(I+t(v_1+v_2)\otimes(v_1-v_2)).
\end{equation*}
$\varphi$ is a quadratic function with $\varphi(0)=-4Nv_1\cdot Nv_2<0$, and while $N_t(v_1+v_2)$ is a constant with respect to $t$, the term $|N_t(v_1-v_2)|$ diverges to $+\infty$ when $t\rightarrow \pm\infty$. Hence, we can determine two finite values $t_-, t_+$ with $t_-<0<t_+$ as
\begin{equation*}
t_-=\inf A, \;\; t_+=\sup A, \quad \text{where} \; A:=\{t\in\mathbb{R}:\min\{|N_t v_1|,|N_t v_2|\}>1\text{ and }\varphi(t)<0\}. 
\end{equation*}
Since $|N_t v_1|^2-1, |N_tv_2|^2-1$ are quadratic functions of $t$, and so is $\varphi$, the set $A$ is a union of a finite number of open intervals, and is not empty because $0 \in A$. 

We see that $t_-=\inf A$ implies that $\min\{|N_{t_-} v_1|,|N_{t_-} v_2|\}=1$ or $\varphi(t_{-})=0$ holds but, in fact, in either case $\min\{|N_{t_-} v_1|,|N_{t_-} v_2|\}=1$ holds. Indeed, when $\varphi(t_-)=0$, the vectors $N_{t_{-}}v_1$ and $N_{t_{-}}v_2$ are orthogonal, which in view of $\det N_{t_{-}}=1$ implies $\min\{|N_{t_-} v_1|,|N_{t_-} v_2|\}\leq 1$ and so $\min\{|N_{t_-} v_1|,|N_{t_-} v_2|\}=1$. From this and the definition of $\mathcal{M}$, we have that $N_{t_-}\in\mathcal{M}$ and an analogous argument leads to $N_{t_+}\in\mathcal{M}$.

Now we define $F_{\pm}:=N_{t_{\pm}}$. Since $N=N_0$ is the point that internally divides the line segment $N_t, \, t\in [t_-,t_+]$ connecting $F_-$ and $F_+$ into $-t_-:t_+$, setting $\mu=\frac{-t_-}{t_+-t_-}$ yields \eqref{NFpFm}. In addition, the property (i) is satisfied because $N_t(v_1+v_2)$ is constant in $t$. 

The case (ii) is proved in an analogous way, considering the function
\begin{equation*}
\varphi(t)=|N_t(v_1+v_2)|^2-|N_t(v_1-v_2)|^2,\quad \text{where} \;\; N_t=N(I+t(v_1-v_2)\otimes(v_1+v_2)).
\end{equation*}
The argument for cases (iii) and (iv) can be reduced to the setting of single slip, which is proved in Lemma 3.3 of \cite{Christowiak}. In particular, when $N\in\mathcal{M}$, we may take $F_+=F_-=N$. 
\end{proof}
We notice that by the definition of $t_-,t_+$, the function $t \mapsto W_{hom}(N_{t})$ is constant in $[t_-,t_+]$. Moreover, for $F\in \mathcal{M}$ we have $W_{hom}(F)=|Fv_1|^2+|Fv_2|^2-2=W(F)$. Hence, when $\min\{|Nv_1|,|Nv_2|\}>1$ we have
\begin{equation}
\label{energy-equation}
W_{hom}(N)=W(F_+)=W(F_-).
\end{equation}
This is true also when $\min\{|Nv_1|,|Nv_2|\}\leq 1$, as shown in \cite{Christowiak}, Lemma 3.3. 

The following two theorems are taken from \cite{ContiTheil2005} and \cite{MullerSverak1999}, respectively. In combination, they allow us to approximate a matrix $N\in \mathcal{N}$ by a simple laminate of finite energy, while preserving boundary values.

\begin{thm}[\cite{ContiTheil2005}, Theorem 4] 
\label{rank-one-approximation}
Let $F_+,F_-\in \mathbb{R}^{2\times 2}$ with $\det F_+ = \det F_- = 1$ and $\operatorname{rank} (F_{+}-F_{-}) = 1$ and $p\in \mathbb{R}^2$ be such that $|F_{+}p| = |F_{-}p|$ and $F_{+}p \neq F_{-}p$. Moreover, let \(\,\Omega\subset \mathbb{R}^2\) be a bounded domain and fix $\mu\in(0,1).$

Then for any $\delta>0,$ there are $h_{\delta}>0$ and $\Omega_{\delta}\subset\Omega$ with $|\Omega\backslash\Omega_{\delta}|<\delta$ such that the restriction to $\Omega_{\delta}$ of any simple laminate between the gradients $F_{+}$ and $F_{-}$ with weights $\mu$ and $1-\mu$ and period  $h<h_{\delta}$, can be extended to a finitely piecewise affine function $v_{\delta}:\Omega \rightarrow \mathbb{R}^2$ so that $\nabla v_{\delta}=\mu F_{+}+(1-\mu)F_{-}$ on $\partial \Omega$, $\det\nabla v_{\delta}=1$, $|(\nabla v_{\delta})p|\leq|F_{+}p|=|F_{-}p|$, and $\operatorname{dist} (\nabla v_{\delta},[F_{+},F_{-}])\leq\delta\,$ on $\Omega$, where $[F_{+},F_{-}]=\{tF_{+}+(1-t)F_{-}:t\in[0,1]\}.$
\end{thm}

Note that the vector $p$ appearing in the theorem exists and is unique up to scaling, cf. the proof in \cite{ContiTheil2005}. 

\begin{thm}[\cite{MullerSverak1999}, Theorem 1.3]
\label{convex-integration}
Let $\mathcal{K}\subset\{F\in\mathbb{R}^{2\times 2}: \det F=1\}.$ Suppose that $(U_i)_i$ is an in-approximation of $\mathcal{K}$, i.e., the sets $U_i$ are open in $\{F\in \mathbb{R}^{2\times 2}: \det F=1\}$ and
uniformly bounded, $U_i$ is contained in the rank-one convex hull of $U_{i+1}$ for every $i\in \mathbb{N}$, and $(U_i)_i$ converges to $\mathcal{K}$ in the following sense: if $F_i \in U_i$ and $|F_i - F|\rightarrow 0$ as $i \rightarrow \infty$ then $F\in \mathcal{K}.$

Then for any $F \in U_{1}$ and open domain $\Omega \subset \mathbb{R}^2,$ there exists $u\in W^{1,\infty}(\Omega;\mathbb{R}^2)$ such that $\nabla u\in\mathcal{K}\,a.e.\,in\,\Omega$ and $u=Fx\,$on$\,\partial \Omega.$
\end{thm}

Using Lemma \ref{rank-one-connection} and the above two theorems we prove the following approximation result that will be the basis for our construction of recovery sequence.

\begin{cor}
\label{cor-section3}
Let \(\,\Omega\subset \mathbb{R}^2\) be a bounded domain. If $N\in\mathcal{N}\backslash \mathcal{M}$, let $F_{+},F_{-}\in\mathcal{M}$ and $\mu \in(0,1)$ be as in Lemma \ref{rank-one-connection}, otherwise let $F_{+}=F_{-}=N\in\mathcal{M}$ and $\mu \in(0,1)$.

Then for every $\delta >0$ there exists $\Omega_{\delta}\subset\Omega$ with $|\Omega\backslash\Omega_{\delta}|<\delta$ and $u_{\delta}\in W^{1,\infty}(\Omega;\mathbb{R}^2)$ such that $u_{\delta}$ coincides with a simple laminate between $F_{+}$ and $F_{-}$ with weights $\mu$ and $1-\mu$ and
period $h_{\delta}<\delta \,$ in $\Omega_{\delta}$ , $\,\nabla u_{\delta}\in \mathcal{M}$ a.e. in $\Omega$, 
and $u_{\delta}=Nx\,on\, \partial\Omega$. Moreover, there is a constant $c$ depending only on $N$, such that for any $\delta \in (0,1)$,
\begin{equation}
\label{udbound}
|\nabla u_{\delta}|\leq c \quad \text{a.e. in} \;\; \Omega .
\end{equation}
In particular, $\nabla u_{\delta}\rightharpoonup N$ in $L^2(\Omega;\mathbb{R}^{2\times 2})$ as $\delta \rightarrow 0.$
\end{cor}

\begin{proof}
The proof requires addressing the four cases for $N$ in Lemma \ref{rank-one-connection} separately. Here we deal with the case (i) only, since a similar argument applies to case (ii), while cases (iii) and (iv) were treated in Corollary 3.7 of \cite{Christowiak}.
For the given $F_+,F_-$ and $\delta >0$, Theorem \ref{rank-one-approximation} with this $\delta$ and $p$ taken as $v_1+v_2$ yields a finitely piecewise affine function $v_{\delta}$ and a set $\Omega_{\delta}$. 
Since $\operatorname{dist} (\nabla v_{\delta},[F_{+},F_{-}])\leq\delta$ and $0<\delta<1$, there is $c>0$ independent of $\delta$, such that
\begin{equation*} 
|\nabla v_{\delta}q|<c \quad \text{a.e. in} \; \Omega \quad \text{for} \;\;  q=v_1+v_2,v_1-v_2,v_1,v_2 . 
\end{equation*}

The desired function $u_{\delta}$ is obtained as a result of applying Theorem \ref{convex-integration} to modify $v_{\delta}$ in the finitely many subsets of $\Omega\setminus\Omega_{\delta}$ where it is affine. Take any such subset $S$ and first assume that $\nabla v_{\delta}$ satisfies the conditions of case (i) in Lemma \ref{rank-one-connection}, i.e, $\min\{|\nabla v_{\delta}v_1|,|\nabla v_{\delta}v_2|\} >1$ and $\nabla v_{\delta}v_1\cdot \nabla v_{\delta}v_2>0$ on $S$. We apply Theorem \ref{convex-integration} to the in-approximation $(U^{\delta}_{i})_i$ of $\mathcal{K}:=\mathcal{M}\cap\{F\in\mathbb{R}^{2\times 2}:|F(v_1+v_2)|\leq c,\,Fv_1\cdot Fv_2\geq 0\}$ defined as 
\begin{align*}
&U_i^{\delta}:=\left\{F\in\mathbb{R}^{2\times 2}: \det F=1, |F(v_1+v_2)|<c, |Fv_1|>1,|Fv_2|>1, Fv_1\cdot Fv_2>0\right\}\\
&\qquad\cap \left\{F\in\mathbb{R}^{2\times 2}:|Fv_1|<1+2^{-(i-1)}\;\text{or}\;|Fv_2|<1+2^{-(i-1)}\right.\},\quad i\in\mathbb{N}.
\end{align*}
By shifting the index $i$ if necessary, we may assume that $\nabla v_{\delta}\in U_1^{\delta}$ in $S$.

In order to apply Theorem \ref{convex-integration}, it remains to show that $(U^{\delta}_{i})_i$ is indeed an in-approximation of $\mathcal{K}$. 
When $\,|F_i-F|\rightarrow 0$ for $F_i \in U^{\delta}_i$, it is easy to see that $\det F=1$ and $|F(v_1+v_2)|\leq c$ and $Fv_1\cdot Fv_2\geq 0$. 
In addition, we see that either $|Fv_1|=1$ or $|Fv_2|=1$ holds. Indeed, assuming that there is a subsequence $(F_{i_k})_k$ that satisfies $1<|F_{i_k}v_1|<1+2^{-(i_k-1)}$ for all $k$, we obtain $|Fv_1|=1$. If this assumption is false and there is only a finite set of values $k$ for which $1<|F_{i_k}v_1|<1+2^{-(i_k-1)}$ holds, from the definition of $U_i^{\delta}$ there must be an infinite subsequence $(F_{i_k})_k$ satisfying $1<|F_{i_k}v_2|<1+2^{-(i_k-1)}$ for all $k$. But then $|Fv_2|=1$.

Next, we prove that $U^{\delta}_i$ is contained in the rank-one convex hull of $U^{\delta}_{i+1}$.
Let $G\in U^{\delta}_i$ and set $G_t=G(I+t(v_1+v_2)\otimes (v_1-v_2))$. Then we have $|G_t (v_1+v_2)|=|G(v_1+v_2)|<c$ for any $t\in\mathbb{R}$. Defining $t_-, t_+$ as in Lemma \ref{rank-one-connection}, in view of the continuity of $|G_{t}v_1|, |G_{t}v_2|$ as functions of $t$, one can find $t_{-}^{i+1}>t_-$, $t_{+}^{i+1}<t_+$ close enough to $t_{-}$ , $t_{+}$ so that $\min\{|G_{t^{i+1}_-}v_1|, |G_{t^{i+1}_-}v_2| \} < 1+2^{-i}$ and $\min\{|G_{t^{i+1}_+}v_1|, |G_{t^{i+1}_+}v_2| \} < 1+2^{-i}$. Also, due to the definitions of $t_-$ and $t_+$, $G_{t^{i+1}_+}v_1\cdot G_{t^{i+1}_+}v_2>0$. Then $G_{t_{-}^{i+1}}, \, G_{t_{+}^{i+1}}\in U^{\delta}_{i+1}$, and $G$ can be written as a rank-one convex combination of $G_{t_{-}^{i+1}}$ and $G_{t_{+}^{i+1}}$. 

Now Theorem \ref{convex-integration} allows us to modify $v_{\delta}$ in $S$ to obtain a function $u_{\delta}$, which satisfies $\nabla u_{\delta} \in \mathcal{M}$ and $|\nabla u_{\delta}(v_1+v_2)|\leq c$ in $S$. 

The case when $\min\{|\nabla v_{\delta}v_1|,|\nabla v_{\delta}v_2|\} >1$ and $\nabla v_{\delta}v_1\cdot \nabla v_{\delta}v_2<0$ is handled analogously. 
One just needs to modify $\mathcal{K}, U^{\delta}_i$, and $G_t$ appropriately by interchanging the roles of $v_1+v_2$ and $v_1-v_2$, and switching the inequality $Fv_1\cdot Fv_2>0$ to $Fv_1\cdot Fv_2<0$ in the definition of $U_i^{\delta}$.

The case when $\min\{|\nabla v_{\delta}v_1|,|\nabla v_{\delta}v_2|\} <1$ in a subset $S$ can be reduced to the proof of Lemma 2 of \cite{ContiTheil2005}. For example, when $|\nabla v_{\delta} v_1| <1$, this yields $\nabla u_{\delta} \in \mathcal{M}_1$ satisfying $|\nabla u_{\delta} v_2| \leq c$ in $S$. Similarly, when $|\nabla v_{\delta} v_2| <1$, we have $\nabla u_{\delta} \in \mathcal{M}_2$ satisfying $|\nabla u_{\delta} v_1| \leq c$ in $S$.

Finally, to prove \eqref{udbound}, we recall that $\nabla u_{\delta}$ coincides with $\nabla v_{\delta}$ in $\Omega_{\delta}$ and thus satisfies 
$ |(\nabla u_{\delta})(v_1+v_2)|=|N(v_1+v_2)|$ a.e. in $\Omega_{\delta}$. 
On the other hand, $\nabla u_{\delta}$ on subsets $S$ of $\Omega\setminus\Omega_{\delta}$ was constructed so that it fulfills one of the following four conditions: (i) $\nabla u_{\delta}\in\mathcal{M}$ and 
$|(\nabla u_{\delta})(v_1+v_2)|\leq c$, (ii) $\nabla u_{\delta}\in\mathcal{M}$ and 
$|(\nabla u_{\delta})(v_1-v_2)|\leq c$, (iii) $\nabla u_{\delta}\in\mathcal{M}_1$ and $|\nabla u_{\delta}v_2|\leq c$, (iv) $\nabla u_{\delta}\in\mathcal{M}_2$ and $|\nabla u_{\delta}v_1|\leq c$. 
Therefore, for $\delta<1$, in either of the four cases both $|\nabla u_{\delta}v_2|$ and $|\nabla u_{\delta}v_1|$ are uniformly bounded with respect to $\delta$ in terms of $N$, from which \eqref{udbound} follows.
\end{proof}

\subsubsection{Recovery sequence for constant $\gamma$ }
\label{sec_rs3}

Let $\gamma \in \mathbb{R}$ be constant and define $N \in \mathcal{N}$ through 
\begin{equation}
\label{definition-of-N}
\lambda N+(1-\lambda)R=R(I+\gamma e_1\otimes e_2).
\end{equation}
With this definition, $Ne_1=Re_1$ holds, and thus laminates parallel to the direction of $e_1$ can be constructed.
For $\epsilon\in (0,1)$ let $\varphi^1_{\epsilon}\in W^{1,\infty}((0,1)\times(0,\lambda);\mathbb{R}^2)$ be obtained by applying Corollary \ref{cor-section3} to $\Omega=(0,1)\times(0,\lambda)\subset\mathbb{R}^2$ and $N$ with $\delta=\epsilon$. We then define
\begin{equation*}
\varphi _{\epsilon}(x) :=\sum_{i\in\mathbb{Z}^2}\left(\varphi^1_{\epsilon}(x-i)-N(x-i) \right)\mathbbm{1}_{i+(0,1)\times(0,\lambda)},\quad x\in\mathbb{R}^2,
\end{equation*}
so that $\varphi_{\epsilon}$ is $Y$-periodic.
Next we set $z_{\epsilon}=w+\varphi_{\epsilon}$, where $w\in W^{1,\infty}_{loc}\cap L^{\infty}_0(\mathbb{R}^2;\mathbb{R}^2)$ is such that
$\nabla w =R\mathbbm{1}_{Y_{rig}}+N\mathbbm{1}_{Y_{soft}}$ (see Fig. \ref{simplelaminate}).
Because $\nabla \varphi_{\epsilon}\rightharpoonup 0$ in $L^2_{loc}(\mathbb{R}^2;\mathbb{R}^{2\times 2})$ as $\epsilon\rightarrow 0$,
\begin{equation}
\label{nablazepsilon}
\nabla z_{\epsilon}\rightharpoonup \nabla w\quad\text{in}\,L^2_{loc}(\mathbb{R}^2;\mathbb{R}^{2\times 2}).
\end{equation}
If $u_{\epsilon}\in W^{1,2}(\Omega;\mathbb{R}^2)$ with mean value zero is determined by
\begin{equation*}
\nabla u_{\epsilon}(x)=\nabla z_{\epsilon}\left(\frac{x}{\epsilon}\right),\;\;x\in\Omega,
\end{equation*}
the functions $u_{\epsilon}$ are admissible for $E_{\epsilon}$ in view of $\nabla u_{\epsilon}=R\in SO(2)$ in $\epsilon Y_{rig}\cap \Omega$ and $\nabla u_{\epsilon}\in \mathcal{M}$ a.e. in $\Omega$.

\begin{figure}[!ht]
\begin{center}
\includegraphics[width=0.85\textwidth]{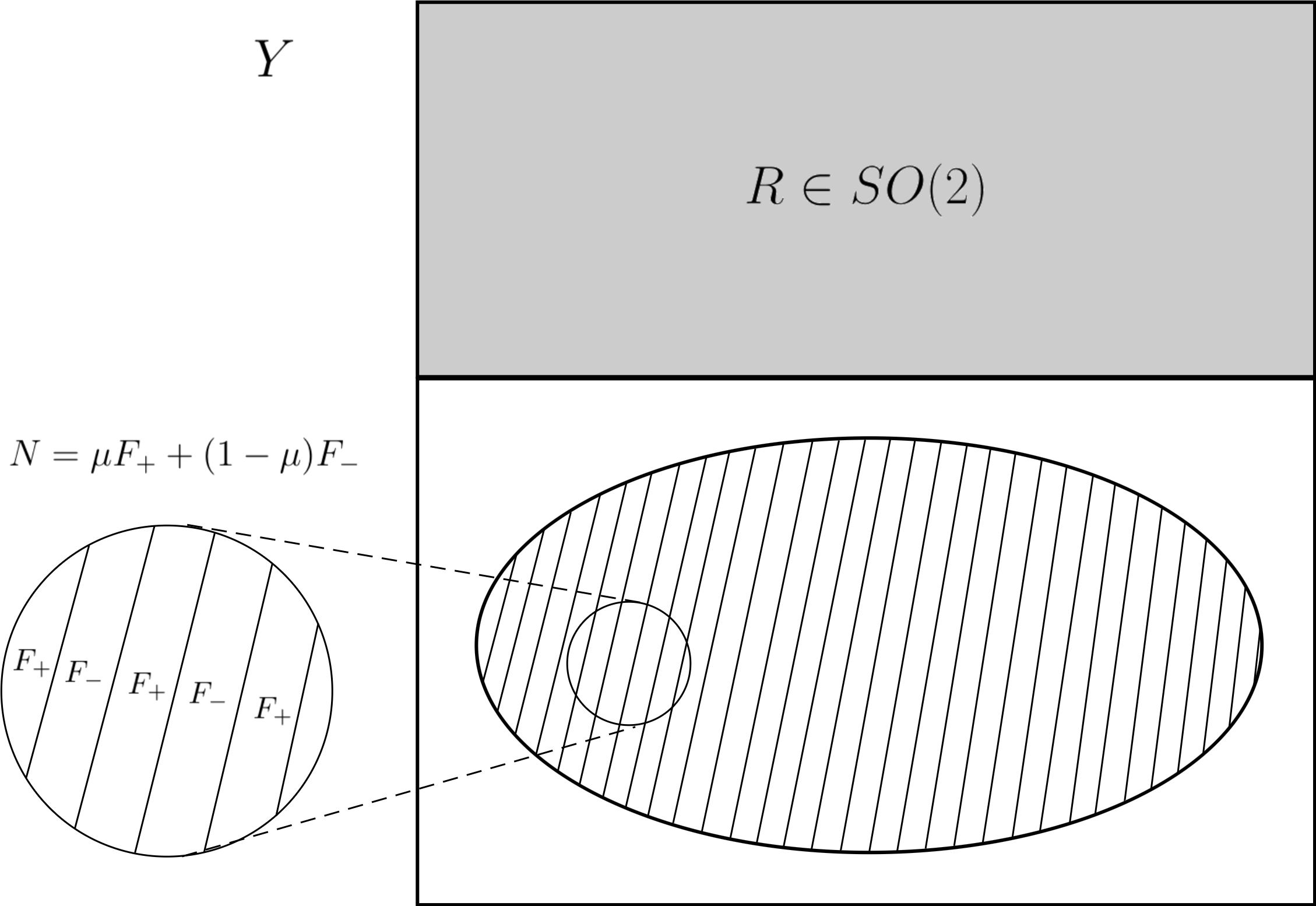}
\end{center}
\caption{Structure of $\nabla z_{\epsilon}$ in the unit cell $Y$. Corollary \ref{cor-section3} is used to graft a simple laminate into the soft layer, so that the surrounding area is smaller than $\epsilon$. Here, $\mu\in (0,1)$ and $F_{+}, F_{-}$ are as in Lemma \ref{rank-one-connection}}.
\label{simplelaminate}
\end{figure}

By \eqref{nablazepsilon}, when $\epsilon\to 0$,
\begin{equation}
\label{nablauweak}
\nabla u_{\epsilon}\rightharpoonup \int_{Y}\nabla w\,dx=\lambda N+(1-\lambda)R=\nabla u\quad\text{in}\;\; L^2(\Omega;\mathbb{R}^{2\times 2}).
\end{equation}
This weak convergence is proved using a nonstandard averaging lemma, see Appendix \ref{sec_averaginglemma} for details. 
Recalling \eqref{udbound} yields uniform boundedness of $\|\nabla z_{\epsilon}\|_{L^{\infty}}$ for $0<\epsilon<1$.
It follows from the continuity of $W$ on $\mathcal{M}$ that $|W(\nabla u_{\epsilon})|=|W(\nabla z_{\epsilon}(\epsilon^{-1}x))|\leq c\,$ a.e. in $\Omega$ for some $c>0$ and for all $0<\epsilon<1$.
We wish to investigate the limit 
\begin{equation}
\label{convergesEeue}
\lim_{\epsilon\rightarrow 0}E_{\epsilon}(u_{\epsilon})=\lim_{\epsilon\rightarrow 0}\int_{\Omega}W(\nabla z_{\epsilon}(\epsilon^{-1}x))\mathbbm{1}_{\epsilon Y_{soft}\cap\Omega}\, dx.
\end{equation}
Let $\Omega_{\epsilon}\subset \Omega$ be the open set where $\nabla u_{\epsilon}$ is a simple laminate. Then $\nabla u_{\epsilon}$ takes value in $\{F_{+},F_{-}\}$ on $\Omega_{\epsilon}$ and in $\mathcal{M}$ on $\Omega_{\epsilon}^0:=(\epsilon Y_{soft}\cap\Omega)\backslash\Omega_{\epsilon}.$ Using \eqref{energy-equation},
\begin{align*}
\eqref{convergesEeue}
=\lim_{\epsilon\rightarrow 0} \left( \int_{\epsilon Y_{soft}\cap\Omega}W_{hom}(N) \, dx+\int_{\Omega_{\epsilon}^0} \left( W(\nabla z_{\epsilon}(\epsilon^{-1}x))-W_{hom}(N) \right) dx \right). 
\end{align*}
The first term converges to $\lambda |\Omega|W_{hom}(N)$ when $\epsilon\rightarrow 0$ due to the averaging lemma. On the other hand, the second term converges to 0 when $\epsilon\rightarrow 0$ due to $|\Omega_{\epsilon}^0|\rightarrow 0$ and the uniform boundedness of the $L^{\infty}$-norm of the integrand. We conclude that
\begin{equation*}
\lim_{\epsilon\rightarrow 0}E_{\epsilon}(u_{\epsilon})=\lambda |\Omega|W_{hom}(N)=E(u).
\end{equation*}

\subsubsection{Recovery sequence for piecewise constant $\gamma$ }
\label{sec_rs4}

In this step, we prove the existence of recovery sequence for (finitely) piecewise affine $u$. To begin with, assume that $\Omega=(0,l)^2\subset\mathbb{R}^2$ with $l > 0$.
Since $\gamma$ is essentially a function of $x_2$ only, $\gamma$ can be expressed as a one-dimensional simple function
\begin{equation*}
\gamma(x_1,t)=\sum_{i=1}^{n}\gamma_i\mathbbm{1}_{(t_{i-1},t_i)}(t),\quad t\in(0,l),
\end{equation*}
with $\gamma_i \in \mathbb{R},\,i=1,\ldots,n$ and $0=t_0<t_1<\cdots<t_n=l.$

Let $N_i\in \mathcal{N}$ be the matrix corresponding to $\gamma_i$ via \eqref{definition-of-N}. We define $u_{\epsilon}\in W^{1,2}(\Omega;\mathbb{R}^2)$ with mean value zero by
\begin{equation*}
\nabla u_{\epsilon} :=R+\sum_{i=1}^{n}(\nabla u_{i,\epsilon}-R)\mathbbm{1}_{\epsilon Y_{soft}\cap\Omega}\mathbbm{1}_{\mathbb{R}\times(\lceil\epsilon^{-1}t_{i-1}\rceil\epsilon,\lfloor\epsilon^{-1}t_i\rfloor\epsilon)},
\end{equation*}
where $(u_{i,\epsilon})_{\epsilon}\subset W^{1,2}((0,l)\times(t_{i-1},t_i);\mathbb{R}^2)$ for $i\in\{1,\ldots,n\}$ are the recovery sequences corresponding to $\gamma_i$ constructed in Section \ref{sec_rs3}. 
Note that $u_{\epsilon}$ is well defined since $\nabla u_{i,\epsilon}$ and $R$ rank-one connect along all segments $[0,l] \times j \epsilon$, $j=1, \dots, \lfloor l/\epsilon \rfloor$ (also see Fig. \ref{recovery-sequence}).

For $\varphi\in L^2(\Omega),$
\begin{align*}
\int_{(0,l)\times (t_{i-1},t_i)}(\nabla u_{\epsilon}-\nabla u)\varphi\,&dx=\int_{(0,l)\times(\lceil\epsilon^{-1}t_{i-1}\rceil\epsilon,\lfloor\epsilon^{-1}t_i\rfloor\epsilon)}(\nabla u_{i,\epsilon}-\nabla u)\varphi\,dx\\
&+\int_{(0,l)\times (t_{i-1},t_i)] \backslash (0,l)\times(\lceil\epsilon^{-1}t_{i-1}\rceil\epsilon,\lfloor\epsilon^{-1}t_i\rfloor\epsilon)}(R-\nabla u)\varphi\,dx
\end{align*}
By \eqref{nablauweak} the first term on the right-hand side converges to $0$ as $\epsilon\rightarrow 0$, while the second term converges to $0$ due to vanishing measure of integration domain. It follows that
\begin{equation*}
\nabla u_{\epsilon}\rightharpoonup \sum_{i=1}^n(\lambda N_i+(1-\lambda)R)\mathbbm{1}_{(0,l)\times (t_{i-1},t_i)}=\nabla u\quad\text{in} \, L^2(\Omega;\mathbb{R}^{2\times 2}).
\end{equation*}
Similarly, for the energy contributions, it follows that
\begin{align*}
\int_{(0,l)\times (t_{i-1},t_i)}W(\nabla u_{\epsilon})\,&dx=\int_{(0,l)\times(\lceil\epsilon^{-1}t_{i-1}\rceil\epsilon,\lfloor\epsilon^{-1}t_i\rfloor\epsilon)}W(\nabla u_{i,\epsilon})\, dx\\
        &+\int_{(0,l)\times (t_{i-1},t_i) \backslash (0,l)\times(\lceil\epsilon^{-1}t_{i-1}\rceil\epsilon,\lfloor\epsilon^{-1}t_i\rfloor\epsilon)}W(R)\, dx.
\end{align*}
As $\epsilon\rightarrow 0$ the first term converges to $\lambda(t_i-t_{i-1})W_{hom}(N_i)$ in the same way as in Section \ref{sec_rs3}, and the second term is equal to $0$. This implies
\begin{equation*}
\lim_{\epsilon\rightarrow 0}\int_{\Omega}W(\nabla u_{\epsilon})\, dx =\lambda\sum_i^{n}(t_i-t_{i-1})W_{hom}(N_i) =E(u).
\end{equation*}

To generalize to a simply-connected Lipschitz domain $\Omega$, just fill $\Omega$ with overlapping cubes and glue the gradients properly.

\section{Envelopes for general slip systems}
\label{section_envelopes}

Our homogenization result in the previous Section relies on the construction of generalized convex envelopes. For the problem with a single slip direction or the problem with two orthogonal slip directions, the rank-one convex envelope of energy density in the setting of rigid plasticity have in common the usage of first-order laminates, i.e., rank-one convex combinations of matrices with finite energy. 
When there are two orthogonal slip systems, first-order laminates yield the rank-one convex envelope \cite{Conti2013}, which allows for the full identification of the corresponding $\Gamma$-limit.
In this case, this rank-one convex envelope in fact matches the polyconvex envelope and the quasiconvex envelope as well. 
This is not true for the spatial dimension $3$ \cite{Conti2013}.
In this section, we partially extend the results from \cite{Conti2013} to arbitrary two-slip systems, pointing out the related results in \cite{Conti2016}. 
The results are partial in the sense that the envelopes are identified only in a strict subset of $\mathbb{R}^{2\times 2}$. 

To fix the notation, let unit vectors $v_1,v_2$ be such that $(v_1, \,v_2)$ is a right-handed system and the angle between $v_1$ and $v_2$ is $2\theta$ with $\pi/4<\theta<\pi/2$. 
Further, let 
\begin{equation}
\label{defv3}
v_3:=-\frac{v_1+v_2}{|v_1+v_2|}
\end{equation}
denote the unit vector dividing in half the larger angle between $v_1$ and $v_2$.
For later use we note that
\begin{equation}
\label{v1.v2.v3}
v_1\cdot v_2 = \cos 2\theta, \;\; v_1\cdot v_3 = v_2\cdot v_3 = - \cos \theta, \;\; v_1 \cdot v_3^{\perp} = -v_2 \cdot v_3^{\perp} =  \sin \theta .
\end{equation}

Considering the energy density
\begin{equation*}
W(F)
=\begin{cases}
|F|^2-2 & \text{if}\; \;F\in \mathcal{M}_1\cup \mathcal{M}_2,\\
    \infty & \text{otherwise},
\end{cases}
\end{equation*}
where
\begin{equation*}
\mathcal{M}_i:=\left\{F\in \mathbb{R}^{2\times 2}: \; \det F=1, \, |Fv_i|=1\right\}, \qquad i=1,2,
\end{equation*}
the main result of this section reads as follows.

\begin{prop}
\label{prop_for_theorem2}
The rank-one convex envelope $W^{rc}$, the polyconvex envelope $W^{pc}$ and the quasiconvex envelope $W^{qc}$ of {W} fulfill
\begin{align*}
W^{rc}(F)=W^{pc}(F)=W^{qc}(F) &= \begin{cases} 
h(|Fv_3|) \quad & \text{if} \; F\in \overline{\mathcal{A}}\cup \overline{\mathcal{N}_1\cap \mathcal{N}_2}, \\ 
h^{\perp}(|Fv_3^{\perp}|) \quad & \text{if} \; F\in \overline{\mathcal{A}}_{\perp}  
\end{cases} \\
&= \begin{cases} 
|Fv_1^{\perp}|^2-1 & \text{if}\; \det F=1,|Fv_1|=1,\\
|Fv_2^{\perp}|^2-1 & \text{if} \; \det F=1,|Fv_2|=1,\\
h (|Fv_3|) & \text{if} \; F\in \mathcal{A}\cup (\mathcal{N}_1\cap \mathcal{N}_2),\\
h^{\perp} (|Fv_3^{\perp}|) & \text{if} \; F\in \mathcal{A}_{\perp}. 
\end{cases} 
\end{align*}
Furthermore,
\begin{equation*} W^{rc}(F)=W^{pc}(F)=+\infty \qquad \text{if} \;\, \det F \neq 1.\end{equation*}
Here the sets $\mathcal{N}_1, \mathcal{N}_2, A, \mathcal{A}_{\perp}$ are defined by
\begin{align*}
\mathcal{N}_i&:=\left\{F\in \mathbb{R}^{2\times 2} : \; \det F=1,\, |Fv_i|< 1\right\}, \qquad i=1,2, \\
\mathcal{A}&:=\left\{F\in \mathbb{R}^{2\times 2} : \; \det F=1,\,\min_{i=1,2}|Fv_i|>1,\, Fv_1\cdot Fv_2>0 \right\},\\
\mathcal{A}_{\perp}&:=\left\{F\in \mathbb{R}^{2\times 2} : \; \det F=1,\,\min_{i=1,2}|Fv_i|>1,\, Fv_1\cdot Fv_2<0 \right\}.
\end{align*}
\end{prop}

We recall that $h$ and $h^{\perp}$ were defined in Theorem \ref{maintheorem2}.
For the purposes of this Section, we extend them as (see Fig. \ref{fig_graphofh}),
\begin{align*}
h^*(z)&:=\frac{1+z^2-2\cos\theta\sqrt{z^2-\sin^2\theta}}{\sin^2\theta}-2\qquad \text{for} \; z\in[\sin\theta, \infty),\\
h^{\perp *}(z)&:=\frac{1+z^2-2\sin\theta\sqrt{z^2-\cos^2\theta}}{\cos^2\theta}-2\qquad \text{for} \; z\in[\cos\theta, \infty).
\end{align*}
Then one easily checks that 
\begin{equation}
\label{newdefh}
h(z) = \begin{cases}
0 & \text{if} \; z\in [0,1] \\
h^*(z) & \text{if} \; z \geq 1 
\end{cases}, \qquad
h^{\perp}(z) = \begin{cases}
0 & \text{if} \; z\in [0,1] \\
h^{\perp *}(z) & \text{if} \; z \geq 1 
\end{cases},
\end{equation}
and that
\begin{equation}
\label{defofh1}
h(z)\leq h^*(z)  \;\; \text{for} \; z\in[\sin\theta, \infty), \quad \text{and} \quad h^{\perp}(z)\leq h^{\perp*}(z) \;\; \text{for} \; z\in[\cos\theta, \infty).
\end{equation}

\begin{figure}[!ht]
\begin{center}
\includegraphics[width=0.4\textwidth]{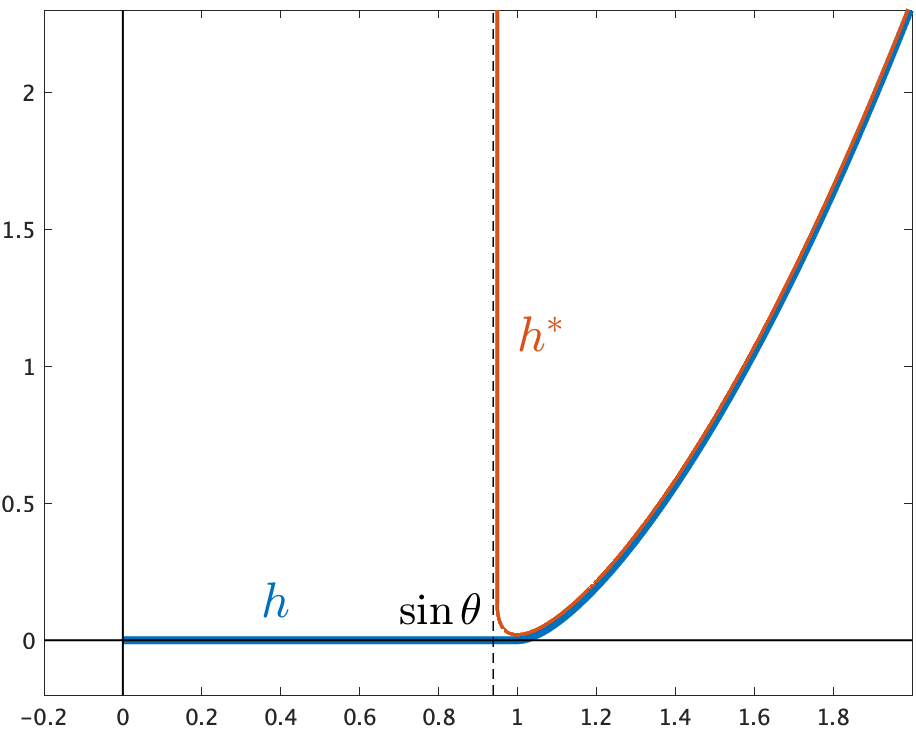} \hspace{1cm}
\includegraphics[width=0.4\textwidth]{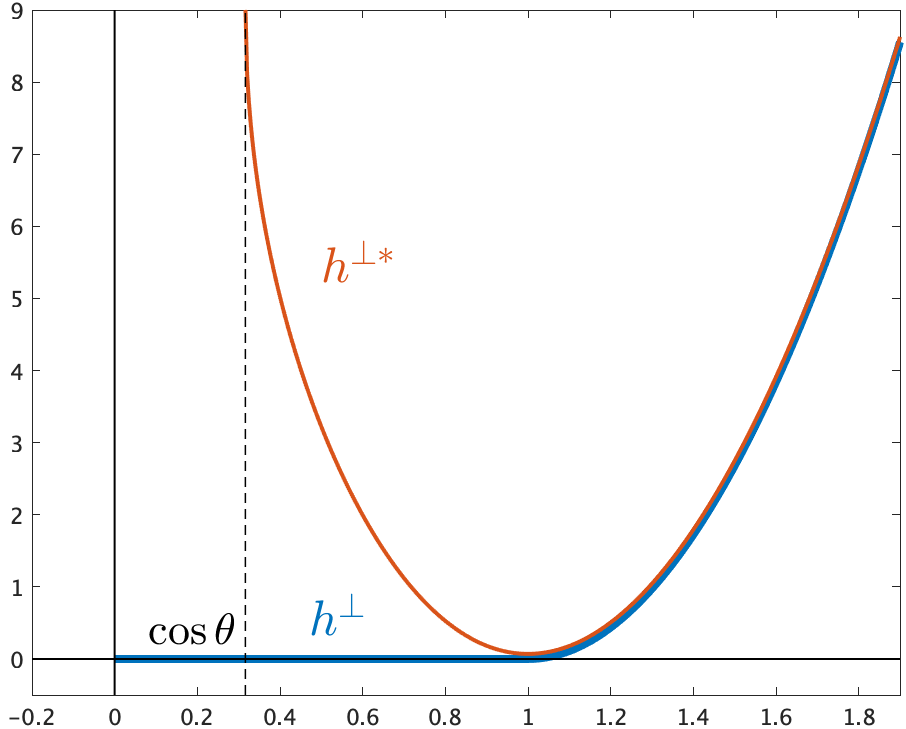}
\end{center}
\caption{The graphs of functions $h,h^*$ and $h^{\perp}, h^{\perp *}$ for the particular case of slip directions $v_1=(1,0), v_2=(-0.8,0.6)$.}
\label{fig_graphofh}
\end{figure}

\begin{rem}
\sloppy The condition $Fv_1\cdot Fv_2 \gtrless 0$ in the definition of $\mathcal{A}$ and $\mathcal{A}_{\perp}$ appears already in Theorem \ref{maintheorem} and serves the purpose of dividing the set $\left\{F: \; \det F=1,\, |Fv_1|> 1, |Fv_2|>1 \right\}$ into two parts $\mathcal{A}$ and $\mathcal{A}_{\perp}$, as shown in Figure \ref{figdefAN}. 
It is equivalent to $|Fv_3|/|Fv_3^{\perp}| \gtrless \sin\theta/\cos\theta$, as shown in Lemma \ref{exact} below.
\end{rem}

\begin{figure}[!ht]
\begin{center}
\includegraphics[width=0.65\textwidth]{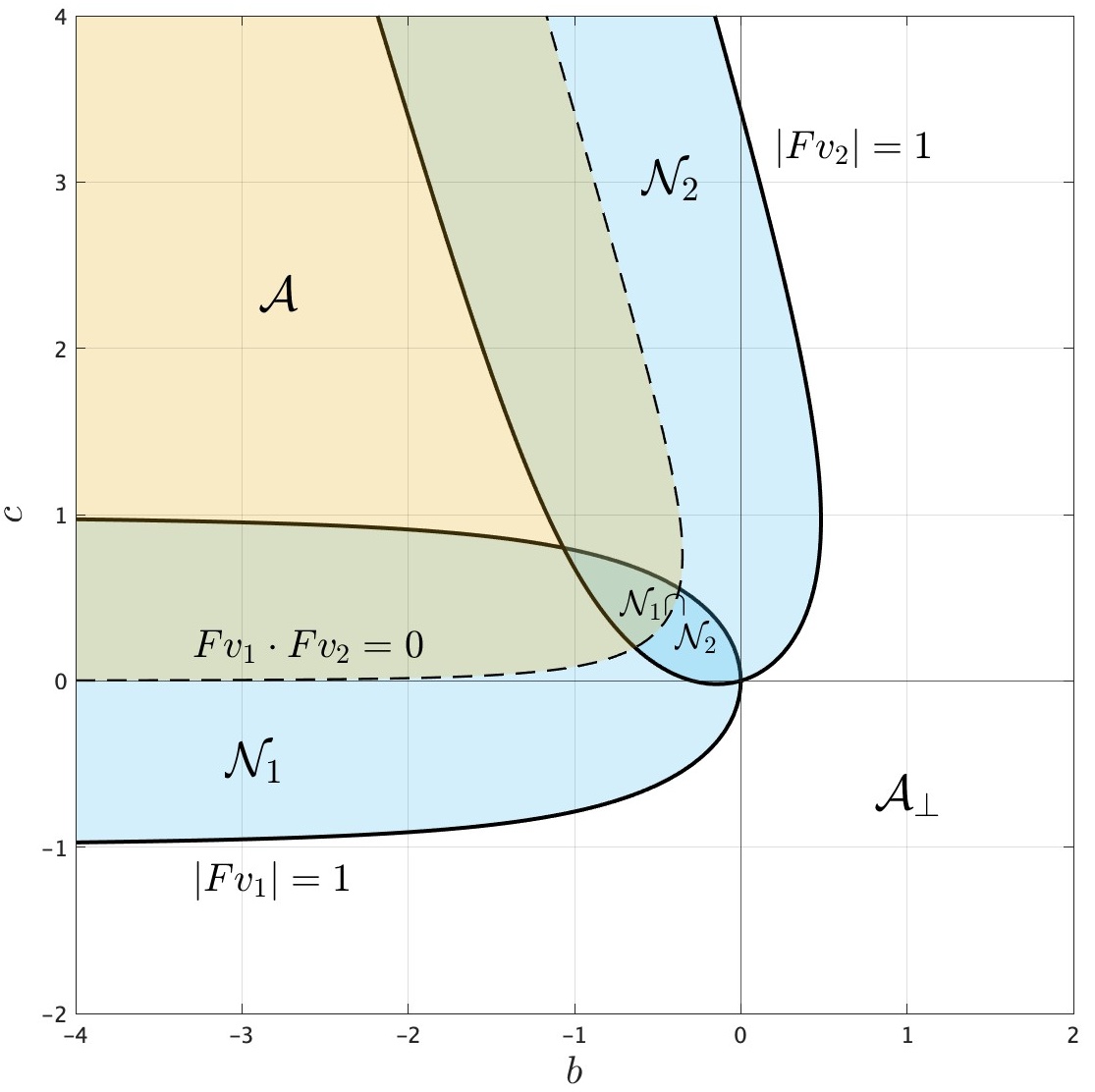}
\end{center}
\caption{The sets $\mathcal{N}_1, \mathcal{N}_2, \mathcal{A}, \mathcal{A}_{\perp}$ shown for $v_1=(1,0), v_2=(-0.8,0.6)$ using the $b,c$-coordinate system defined by $\left(F^T F\right)^{1 / 2}=\left(\begin{array}{cc}
a+b & c \\ c & a-b \end{array}\right)$ with $a=\sqrt{1+b^2+c^2}$. Each point in the $(b, c)$ plane corresponds to the orbit of one matrix under $SO(2)$ on matrices with determinant 1.}
\label{figdefAN}
\end{figure}

We introduce two formulas that will be useful in the sequel. First, for $F\in \mathbb{R}^{2\times 2}$ and $a,b\in \mathbb{R}^2$, one has
\begin{equation}
\label{F.1}
 |Fa|^2|Fb|^2=|Fa\cdot Fb|^2+(\det F)(a^{\perp}\cdot b)^2.
\end{equation}
Second, for $F\in \mathbb{R}^{2\times 2}$ and unit vectors $a,b$ with $a^{\perp}\cdot b\neq 0$, one has
\begin{equation}
\label{F.2}
 |F|^2=\frac{|Fa|^2+|Fb|^2-2(a\cdot b)(Fa\cdot Fb)}{(a^{\perp}\cdot b)^2}.
\end{equation}

\begin{lem}
\label{exact}
The conditions $Fv_1\cdot Fv_2\gtrless 0$ and $|Fv_3|/|Fv_3^{\perp}|\gtrless\sin\theta/\cos\theta$ are equivalent.
In addition, the following holds:
\begin{align*} 
& |Fv_3|\geq 1 \qquad \text{for} \;\; F\in \overline{\mathcal{A}} \cup (\overline{\mathcal{N}_1 \cap \mathcal{N}_2}), \\
& |Fv_3^{\perp}|\geq 1 \qquad \text{for} \;\;  F\in \overline{\mathcal{A}_{\perp}}.
\end{align*}
\end{lem}
\begin{proof}
First, we calculate
\begin{equation}
   \label{|Fv_3|} 
\begin{aligned}
|Fv_3|^2 &= \frac{|Fv_1+Fv_2|^2}{|v_1+v_2|^2} = \frac{|Fv_1|^2+|Fv_2|^2 + 2Fv_1\cdot Fv_2}{2+2v_1\cdot v_2}, \\
|Fv_3^{\perp}|^2 &= \frac{|Fv_1-Fv_2|^2}{|v_1-v_2|^2} = \frac{|Fv_1|^2+|Fv_2|^2 - 2Fv_1\cdot Fv_2}{2-2v_1\cdot v_2}, 
\end{aligned} 
\end{equation}
which together with \eqref{v1.v2.v3} implies that, if $Fv_3^{\perp}\neq 0$,
\begin{equation*}\frac{|Fv_3|^2}{|Fv_3^{\perp}|^2} \cdot\frac{\cos^2\theta}{\sin^2\theta}=\frac{|Fv_1|^2+|Fv_2|^2 + 2Fv_1\cdot Fv_2}{|Fv_1|^2+|Fv_2|^2 - 2Fv_1\cdot Fv_2}.\end{equation*}
Hence, the conditions $Fv_1\cdot Fv_2\gtrless 0$ and $|Fv_3|/|Fv_3^{\perp}|\gtrless\sin\theta/\cos\theta$ are equivalent.

To prove the second part of the Lemma, we use \eqref{F.1} to get 
\begin{equation}
\label{Fvandv}
|Fv_1\cdot Fv_2|^2 = |Fv_1|^2|Fv_2|^2 - (v_1^{\perp}\cdot v_2)^2 =  |Fv_1|^2|Fv_2|^2 - 1 + (v_1 \cdot v_2)^2.
\end{equation}
Now if $F\in\overline{\mathcal{A}}$ then $|Fv_1|,|Fv_2| \geq 1$, $Fv_1\cdot Fv_2\geq 0$, and thus in view of $v_1\cdot v_2=\cos 2\theta<0$ \eqref{|Fv_3|} implies
\begin{equation*}
|Fv_3|^2 \geq \frac{|Fv_1|^2+|Fv_2|^2}{2} \geq 1.
\end{equation*}
On the other hand, if $F\in \overline{\mathcal{N}_1 \cap \mathcal{N}_2}$, \eqref{|Fv_3|} yields
\begin{equation*} 
|Fv_3^{\perp}|^2 \leq \frac{1-Fv_1\cdot Fv_2}{1-v_1\cdot v_2} \leq 1,
\end{equation*}
because $v_1\cdot v_2< 0$ and $ |Fv_1\cdot Fv_2|^2\leq (v_1\cdot v_2)^2$ by \eqref{Fvandv}, so irrespective of the sign of $Fv_1\cdot Fv_2$, it is true that $v_1 \cdot v_2 \leq Fv_1\cdot Fv_2$. Therefore, $|Fv_3|\geq1$.

Finally, if $F\in \overline{\mathcal{A}^{\perp}}$ then $Fv_1\cdot Fv_2\leq 0$, and 
\begin{equation*} 
|Fv_3^{\perp}|^2 \geq \frac{1-Fv_1\cdot Fv_2}{1-v_1\cdot v_2} \geq 1,
\end{equation*}
because $v_1\cdot v_2< 0$ and $|Fv_1\cdot Fv_2|^2 \geq (v_1\cdot v_2)^2$, again by \eqref{Fvandv}.
\end{proof}

We shall prove Proposition \ref{prop_for_theorem2} in a series of lemmas. 
There are two ways to construct first order laminates: between two matrices on $\mathcal{M}_i$ and between a matrix on $\mathcal{M}_1$ and another matrix on $\mathcal{M}_2$. 
In the first case, the optimal laminate is in the direction $v_i$, which follows from the analysis of the single-slip problem \cite{ContiTheil2005}. In the latter case, the following lemma provides the stepping stone for identifying the optimal directions for first order laminates in the set $\mathcal{A}\cup \mathcal{A}_{\perp}=\left\{F\in\mathbb{R}^{2\times 2}: \; \det F=1,\, |Fv_1|> 1, |Fv_2|>1 \right\}.$

\begin{lem}
\label{leminters}
If $F \in \mathcal{A}_{\perp}$ we define for $t \in \mathbb{R}$,
 \begin{equation}
 \label{lami1}
 F_{t} = F(I +t\,v^{\perp}_3 \otimes v_3), 
 \end{equation}
and if $F\in \mathcal{A}$ we define for $t \in \mathbb{R}$,
 \begin{equation}
 \label{lami2}
 F_{t} = F(I +t\,v_3 \otimes v^{\perp}_3).  
 \end{equation}
Then the equations
 \begin{equation}
 \label{lami.1}
 |F_t v_1|=1
 \end{equation}
and 
 \begin{equation}
 \label{lami.2}
 |F_t v_2|=1
 \end{equation}
have each two distinct solutions. The two solutions of each of these two equations have the same sign, and the sign of the two solutions of \eqref{lami.1} is opposite to the sign of the two solutions of \eqref{lami.2}.
\end{lem}

\begin{figure}[!ht]
\begin{center}
\includegraphics[width=0.47\textwidth]{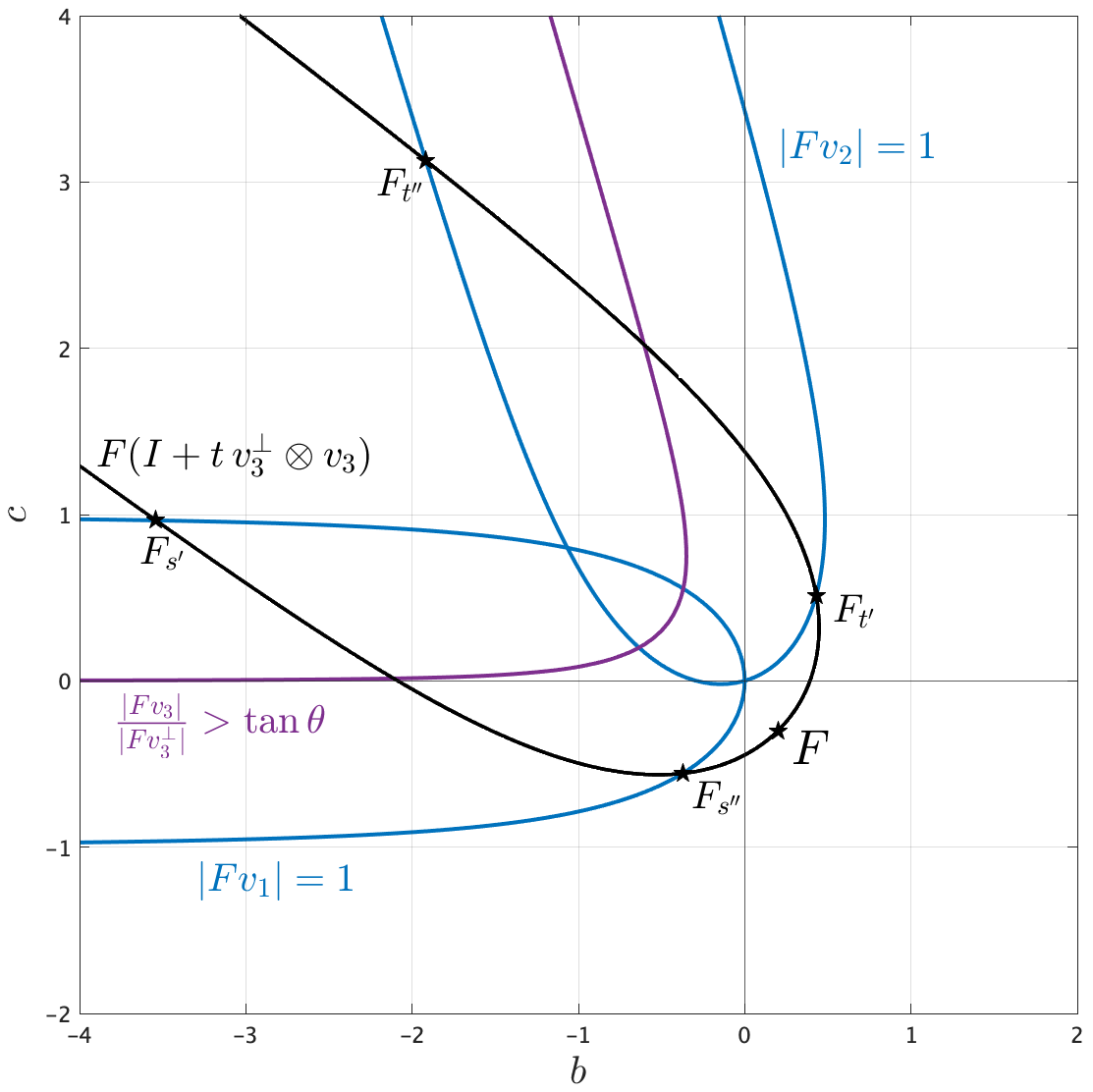}
\includegraphics[width=0.47\textwidth]{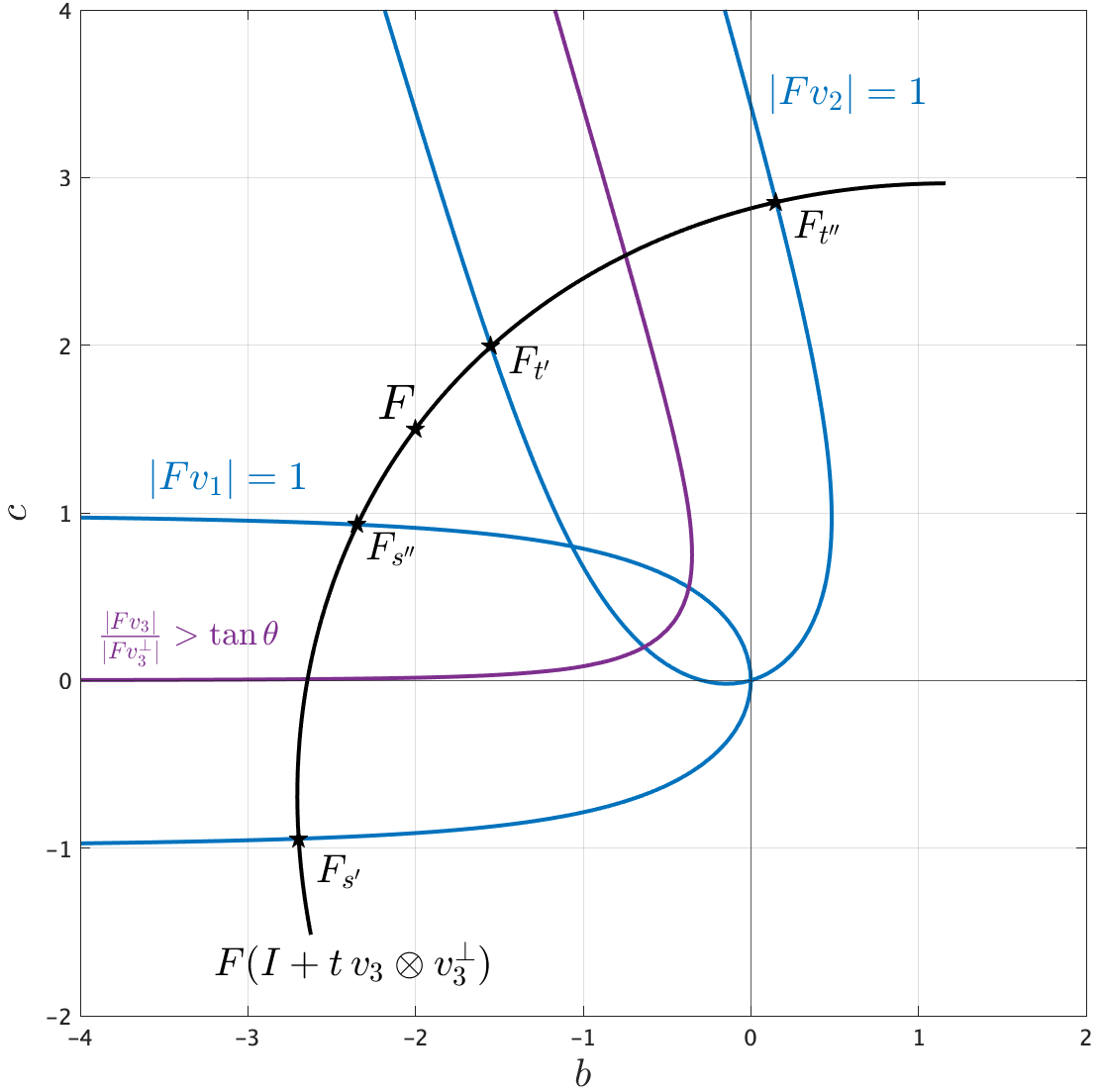}
\caption{Two cases from Lemma \ref{leminters}: $F\in \mathcal{A}_{\perp}$ (left) and $F\in \mathcal{A}$ (right).}
\end{center}
\label{figrankone}
\end{figure}

\begin{proof}
We first prove that there are two distinct solutions to each of \eqref{lami.1}, \eqref{lami.2} and that these solutions have the same sign.
We focus on the first case \eqref{lami1} since the second case \eqref{lami2} follows by switching the roles of $v_3$ and $v_3^{\perp}$.
Equation \eqref{lami.1}, when squared, is equivalent to
\begin{equation*}
|Fv_1|^2-1+2t(v_3\cdot v_1)(Fv_3^{\perp}\cdot Fv_1)+t^2(v_3\cdot v_1)^2|Fv_3^{\perp}|^2=0 .
\end{equation*}
In view of $v_3 \cdot v_1=-\cos\theta \neq 0$ we see that two distinct real solutions exist
provided that
\begin{equation*}
\left(Fv_3^{\perp}\cdot Fv_1\right)^2>|Fv_3^{\perp}|^2(|Fv_1|^2-1). 
\end{equation*}
In view of \eqref{F.1} and $\det F = 1$, this is equivalent to
\begin{equation*}
 |Fv_3^{\perp}|^2>(-v_3\cdot v_1)^2 = \cos^2\theta ,
\end{equation*}
which immediately follows from the assumption $|Fv_3|/|Fv_3^{\perp}|<\sin\theta/\cos\theta$ since due to $\det F=1$ we have $|Fv_3||Fv_3^{\perp}|\geq 1$ and thus
\begin{equation*}
|Fv_3^{\perp}|^2 = \frac{|Fv_3^{\perp}|}{|Fv_3|} |Fv_3^{\perp}| \, |Fv_3| > \frac{\cos\theta}{\sin\theta} > \cos^2\theta .
\end{equation*}
In addition, $|Fv_1|^2-1 > 0$ by assumption, and hence the solutions have the same sign. 
The proof for \eqref{lami.2} is obtained by simply replacing $v_1$ by $v_2$.

We next look at the sign of the solutions for \eqref{lami.1} and \eqref{lami.2}, again only in the first case of \eqref{lami1}. 
It suffices to show that the average of the two solutions of the first equation has a different sign than the average of the two solutions of the second one. The average of the two solutions of \eqref{lami.1} is
\begin{equation*}
 -\frac{Fv_3^{\perp}\cdot Fv_1}{(v_3\cdot v_1)|Fv_3^{\perp}|^2},
\end{equation*}
and, replacing $v_1$ by $v_2$, the average of the two solutions of \eqref{lami.2} is
\begin{equation*}
 -\frac{Fv_3^{\perp}\cdot Fv_2}{(v_3\cdot v_2)|Fv_3^{\perp}|^2}.
\end{equation*}
Since $(v_3 \cdot v_1)(v_3 \cdot v_2) = \cos^2\theta > 0$, it suffices to show that
\begin{equation*}
 \left(Fv_3^{\perp}\cdot Fv_1\right)\left(Fv_3^{\perp}\cdot Fv_2\right)<0.
\end{equation*}
We observe that
\begin{equation*}
v_1=-(\cos \theta) v_3 + (\sin \theta) v_3^{\perp} \quad  v_2=-(\cos \theta) v_3 - (\sin \theta) v_3^{\perp}
\end{equation*}
and thus
\begin{align*}
 \left(Fv_3^{\perp} \cdot Fv_1 \right)\left(Fv_3^{\perp} \cdot Fv_2\right) 
=\cos^2\theta |Fv_3^{\perp} \cdot Fv_3|^2 - \sin ^2\theta |Fv_3^{\perp}|^4 .
\end{align*}
\sloppy The assertion follows in view of Cauchy-Schwarz inequality and the assumption $|Fv_3|/|Fv_3^{\perp}|<\sin\theta/\cos\theta$.
\end{proof}
 
For $F\in\mathbb{R}^{2\times 2}$, we define $W^{lc}(F)$ as follows: 
\begin{multline}
\label{lamcon} 
W^{lc}(F)=\inf\Big\{ \mu W(F_0)+(1-\mu)W(F_1): \\
\mu\in[0,1], \, \mu F_0+(1-\mu)F_1=F, \,
\operatorname{rank}(F_0-F_1)\leq 1 \Big\}. 
\end{multline}
The following two lemmas evaluate $W^{lc}$ from above inside the sets $\mathcal{A}\cup \mathcal{A}_{\perp}$
and $\mathcal{N}_1\cap \mathcal{N}_2$, respectively, in terms of the functions $h, h^{\perp}$. The proof will reveal that the corresponding optimal laminates connect matrices on $\mathcal{M}_1$ and $\mathcal{M}_2$ with equal energies.

\begin{lem}
\label{evaluation.1}
If $F\in \mathcal{A}$ then
\begin{equation*}
W^{lc}(F)\leq h(|Fv_3|),
\end{equation*}
and if $F\in \mathcal{A}_{\perp}$ then
\begin{equation*}
W^{lc}(F)\leq h^{\perp}(|Fv_3^{\perp}|).
\end{equation*}
\end{lem}
\begin{proof}
If $F\in \mathcal{A}$, Lemma \ref{leminters} yields four mutually different real numbers $s', s'', t', t''$ such that $s',s''$ and $t',t''$ have opposite signs, and satisfy
\begin{equation*} |F_{s'}v_1|=|F_{s''}v_1|=|F_{t'}v_2|=|F_{t''}v_2|=1.
\end{equation*}
Here $F_t$ is the rank-one line with direction $Fv_3\otimes v_3^{\perp}$, i.e., $F_t =F(I+tv_3\otimes v_3^{\perp})$.
Note that $z=|F_tv_3|=|Fv_3|$ is independent of $t$. Using the formula \eqref{F.1} with $a=v_i$, $i=1,2$ and $b=v_3$, we get for $t\in\{s',s'',t',t''\}$,
\begin{equation*} F_tv_i \cdot F_tv_3 = \pm\sqrt{|F_tv_i|^2 |F_tv_3|^2 - (v_i^{\perp} \cdot v_3)^2} = \pm \sqrt{z^2 - \sin^2 \theta}. \end{equation*}
Hence by \eqref{F.2},
\begin{equation}
\label{Ft2}
|F_t|^2=\frac{1+z^2\pm2\cos\theta\sqrt{z^2-\sin^2\theta}}{\sin^2\theta},\quad t\in\{s',s'',t',t''\}.
\end{equation}

We may assume that $s'< s'', t'<t''$. Then either $s'<s''<0<t'<t''$ or $t'<t''<0<s'<s''$ holds. 
Since $|F_t|^2$ is a quadratic function of $t$, it never takes the same value more than twice. 
Moreover, there are only two possible values of $|F_t|^2$ for $t = s',s'',t',t''$, and hence we conclude that $|F_{s'}|=|F_{t''}|$ and $|F_{s''}|=|F_{t'}|$. 
Because the quadratic term $t^2$ appearing in $|F_t|^2$ has positive coefficient, we also see that the value with minus sign in \eqref{Ft2} is taken for $t=s'',t'$ (when  $s'<s''<0<t'<t''$) or for $t=t'',s'$ (when $t'<t''<0<s'<s''$).

Therefore, we can choose $s_{\star}\in\{s',s''\}$ and $t_{\star}\in\{t',t''\}$ such that
\begin{equation*}
|F_{s_{\star}}|^2=|F_{t_{\star}}|^2=\frac{1+z^2-2\cos\theta\sqrt{z^2-\sin^2\theta}}{\sin^2\theta}.
\end{equation*}
Note that $s_{\star}$ and $t_{\star}$ have different signs. Since $F$ is a rank-one convex combination of $F_{s_{\star}}$ and $F_{t_{\star}}$ and 
\begin{equation*} W(F_{s_{\star}}) = 
|F_{s_{\star}}|^2-2 =|F_{t_{\star}}|^2-2 = W(F_{t_{\star}}), \end{equation*} 
we conclude 
\begin{equation*}
W^{lc}(F) \leq W(F_{s_{\star}})=|F_{s_{\star}}|^2-2=h^*(z) = h(z),
\end{equation*}
the last equality following from Lemma \ref{exact}.

The case $F\in \mathcal{A}_{\perp}$ is shown analogously, interchanging the roles of $v_3$ and $v_3^{\perp}$ and the roles of $\cos\theta$ and $\sin\theta$.
\end{proof}

\begin{lem}
\label{N1N2}
If $F\in \mathcal{N}_1\cap \mathcal{N}_2$ then
\begin{equation*}
W^{lc}(F)\leq h(|Fv_3|) .
\end{equation*}
\end{lem}
\begin{proof}
Consider the rank-one line given by
$F_{t} = F(I +t\,v_3 \otimes v^{\perp}_3)$.
Since $|F_tv_1|\rightarrow \infty$ as $t \rightarrow \pm\infty$ and $|F_0v_1|<1$, by continuity we find $t_{1,\pm}$ with $t_{1,-} < 0 < t_{1,+}$ and $|F_{t_{1,\pm}} v_1| = 1$. 
Since $F_tv_1 \cdot Fv_3 = Fv_1 \cdot Fv_3 + t|Fv_3|^2(v^{\perp}_3 \cdot v_1)$ and $v^{\perp}_3 \cdot v_1 = \sin\theta >0 $,
it follows that $F_{t_{1,+}} v_1 \cdot Fv_3 > F_{t_{1,-}} v_1 \cdot Fv_3$.
Combining this with \eqref{F.1}, we have 
\begin{equation*} F_{t_{1,+}} v_1 \cdot Fv_3 = \sqrt{z^2 - \sin^2\theta} , \qquad F_{t_{1,-}} v_1 \cdot Fv_3 = - \sqrt{z^2 - \sin^2\theta} ,\end{equation*}
where we set $z = |Fv_3|$.
By \eqref{F.2} with $a = v_1, b = v_3$, 
\begin{equation*}
|F_{t_{1,-}}|^2=\frac{1+z^2-2\cos\theta\sqrt{z^2-\sin^2\theta}}{\sin^2\theta}.
\end{equation*}
Repeating the same argument with $v_2$ instead of $v_1$ leads to two values $t_{2,-} < 0 < t_{2,+}$, such that $|F_{t_{2,\pm}} v_2| = 1$. Since $v^{\perp}_3 \cdot v_2 = -\sin\theta<0$ we obtain as above, 
\begin{equation*}
|F_{t_{2,+}}|^2=\frac{1+z^2-2\cos\theta\sqrt{z^2-\sin^2\theta}}{\sin^2\theta}.
\end{equation*} 
Employing a laminate between $F_{t_{1,-}}$ and $F_{t_{2,+}}$ as in the previous lemma and noting that $|Fv_3| \geq 1$ by Lemma \ref{exact} finishes the proof.
\end{proof}

\begin{rem}
Using the rank-one line $F(I +t\,v^{\perp}_3 \otimes v_3)$) one can similarly show that $W^{rc}(F)\leq h^{\perp*}(|Fv_3^{\perp}|)$ for $F\in \mathcal{N}_1\cap \mathcal{N}_2$. However, this fact does not provide any additional information since for such $F$ one has $h(|Fv_3|) \leq h^{\perp*}(|Fv_3^{\perp}|)$. 
\end{rem}

We next turn to estimating the envelopes of $W$ from below in terms of $h$ and $h^{\perp}$.

\begin{lem}
\label{psiphi}
For any $F\in\mathbb{R}^{2\times 2}$,
\begin{equation*}  
\max\{h(|Fv_3|),h^{\perp}(|Fv_3^{\perp}|)\}\leq W^{pc}(F), W^{qc}(F), W^{rc}(F).
\end{equation*}
In addition, if $\det F\neq1$ then 
$$ W^{pc}(F)=W^{rc}(F)=+\infty . $$
\end{lem}

\begin{proof}
Since for $z\geq 0$, $h(z)$ and $h^{\perp}(z)$ are nondecreasing and convex functions taking finite values, the function $\max\{h(|Fv_3|),h^{\perp}(|Fv_3^{\perp}|)\}$ is convex and takes finite values for $F\in \mathbb{R}^{2\times 2}$. Thus, this function is also polyconvex, quasiconvex and rank-one convex.
Therefore, it suffices to show $h(|Fv_3|)\leq W(F)$ and $h^{\perp}(|Fv_3^{\perp}|)\leq W(F)$, where $F=F_{\gamma}^i=I+\gamma v_i\otimes v_i^{\perp}$, for any $\gamma\in\mathbb{R}$ and each $i=1,2$. In view of $\det F_{\gamma}^i =1$ and the relation $W(F_{\gamma}^i)=\gamma^2$, it is sufficient to check that $h(|F_{\gamma}^iv_3|) \leq \gamma^2$ and $h^{\perp}(|F_{\gamma}^iv_3^{\perp}|) \leq \gamma^2$ for any $\gamma\in\mathbb{R}$ and each $i=1,2$. To begin with, we fix $i=1$. Since $|F_{\gamma}^1v_3|^2=1+2\gamma \cos\theta \sin\theta + \gamma^2\sin ^2\theta$, we see that $|F_{\gamma}^1v_3|^2 - \sin^2\theta \geq 0$ and from \eqref{defofh1}
\begin{align*}
h(|F_{\gamma}^1v_3|) \leq h^*(|F_{\gamma}^1v_3|)
&=\frac{2\cos^2\theta+2\gamma\cos\theta\sin\theta+\gamma^2\sin^2\theta-2\cos\theta\sin\theta\sqrt{(\gamma+\frac{\cos\theta}{\sin\theta})^2}}{\sin^2\theta}\\
&=\left( \frac {\cos\theta-\sin\theta
\sqrt{(\gamma + \frac{\cos\theta}{\sin\theta})^2}}{\sin\theta} \right)^2.
\end{align*}
Hence, if $\gamma\geq -\frac{\cos \theta}{\sin \theta}$, we obtain $h(|F_{\gamma}^1v_3|) \leq h^*(|F_{\gamma}^1v_3|) = \gamma^2$. 
This inequality holds even if $\gamma < -\frac{\cos \theta}{\sin \theta}$. Indeed, 
\begin{equation*}
h(|F_{\gamma}^1v_3|)\leq h^*(|F_{\gamma}^1v_3|) \leq \frac{1+|F_{\gamma}^1v_3|^2+2\cos\theta\sqrt{|F_{\gamma}^1v_3|^2-\sin^2\theta}}{\sin^2\theta}-2
\end{equation*}
and a simple calculation shows that the last expression is equal to $\gamma^2$.

The same is true for $h^{\perp}$, where we find that $h^{\perp}(|F_{\gamma}^1v_3^{\perp}|) \leq h^{\perp*}(|F_{\gamma}^1v_3^{\perp}|)=\gamma^2$ for $\gamma \leq \frac{\sin\theta}{\cos\theta}$ and $h^{\perp}(|F_{\gamma}^1v_3^{\perp}|) \leq h^{\perp*}(|F_{\gamma}^{1}v_3^{\perp}|)\leq \gamma^2$ otherwise.
The proof for $i=2$ is analogous. 

To show the statement about infinite values, we follow \cite{Conti2013} and define the function
\begin{equation*}
I(F)=
\begin{cases}
0 & \text{if}\;\; \det F =1, \\
+\infty & \text{otherwise},
\end{cases}
\end{equation*}
and the functions
\begin{equation*}
g(F)=h(|Fv_3|)+I(F), \qquad g^{\perp}(F)=h^{\perp}(|Fv_3^{\perp}|)+I(F).
\end{equation*}
Since $I(F)$ is a convex function of $\det F$ and thus polyconvex, the functions $g$ and $g^{\perp}$, being sums of a convex and a polyconvex function, are polyconvex functions, and so is $\max \{g,g^{\perp}\}$. The function $\max \{g,g^{\perp}\}$ takes the value $+\infty$ for $F$ such that $\det F\neq 1$, and the first part of this proof shows that $\max\{g,g^{\perp}\}\leq W$.
Since, in general, $W^{pc}\leq W^{rc}$ even for extended-valued function $W$ according to \cite{Muller1999}, we deduce $\max\{g,g^{\perp}\}\leq W^{pc}\leq W^{rc}$. Therefore $W^{pc}(F)=W^{rc}(F)=+\infty$ if $\det F\neq1$.
\end{proof}

\begin{rem}
\label{extension}
 Based on the ideas in the proof of Lemma \ref{psiphi}, we may extend the results of Lemmas \ref{evaluation.1} and \ref{N1N2} up to the boundaries of the involved sets, i.e., Lemma \ref{evaluation.1} holds with $\mathcal{A}$, $\mathcal{A}_{\perp}$ replaced by $\overline{\mathcal{A}}$, $\overline{\mathcal{A}_{\perp}}$, respectively, and Lemma \ref{N1N2} holds with $\mathcal{N}_1\cap \mathcal{N}_2$ replaced by $\overline{\mathcal{N}_1\cap \mathcal{N}_2}$. 
To see why, first note that, using the notation of Lemma \ref{psiphi}, if a matrix $F_{\gamma}^1$ belongs to the boundary of the set $\mathcal{A} \cup (\mathcal{N}_1 \cap \mathcal{N}_2)$ then $\gamma \geq - \frac{\cos\theta}{\sin\theta}$.
Assuming on the contrary $\gamma<-\frac{\cos \theta}{\sin \theta} (<0)$ leads after some calculation to $F^1_{\gamma}v_1 \cdot F^1_{\gamma}v_2=\cos 2\theta +\gamma\sin 2\theta<0$, and thus $F^1_{\gamma}\in \partial \mathcal{A}_{\perp}$. 
But the only matrices belonging to both $\partial (\mathcal{A} \cup (\mathcal{N}_1 \cap \mathcal{N}_2))$ and $\partial \mathcal{A}_{\perp}$ are pure rotations, which shows that for $F_{\gamma}^1 \in \partial (\mathcal{A} \cup (\mathcal{N}_1 \cap \mathcal{N}_2))$ we indeed have $\gamma \geq - \frac{\cos\theta}{\sin\theta}$, and the proof of Lemma \ref{psiphi} in turn implies $h^*(|F^1_{\gamma}v_3|)=\gamma^2=W(F^1_{\gamma})=|F_{\gamma}^1v_1^{\perp}|^2-1$. Moreover, since for such $F_{\gamma}^1$ Lemma \ref{exact} yields $|F_{\gamma}^1 v_3|\geq 1$, we also have $h^*(|F^1_{\gamma}v_3|)=h(|F^1_{\gamma}v_3|)$.
The corresponding identity $h^{\perp}(|F^1_{\gamma}v_3^{\perp}|)=h^{\perp *}(|F^1_{\gamma}v_3^{\perp}|)=\gamma^2=W(F^1_{\gamma})=|F_{\gamma}^1v_1^{\perp}|^2-1$ holds for $F_{\gamma}^1 \in \overline{\mathcal{A}_{\perp}}$, and a similar argument can be made about $F^2_{\gamma}$. 
The statement in the beginning of this remark then follows from the relation $W^{lc}\leq W$. 
\end{rem}

Unfortunately, in the remaining set $(\mathcal{N}_1 \setminus \mathcal{N}_2) \cup (\mathcal{N}_2 \setminus \mathcal{N}_1)$, the estimate of the envelopes from above that is derived from optimal first order laminates does not match with the estimate from below obtained in Lemma \ref{psiphi}. We include the result for completeness and because the boundedness of $W^{lc}$, which it implies, will be needed later. 

\begin{lem}
\label{NiNi+1}
For each $i=1,2$ and $F \in \overline{\mathcal{N}_i}$, 
\begin{equation}
\label{oneslip.eq}
W^{lc}(F)\leq |Fv_{i}^{\perp}|^2 -1 .
\end{equation} 
If additionally $F\in (\mathcal{N}_1\backslash \mathcal{N}_2) \cup (\mathcal{N}_2\backslash \mathcal{N}_1)$ then
\begin{equation*}
W^{lc}(F)\leq \min \left\{ h_+(|Fv_3|), h^{\perp}_+(|Fv^{\perp}_3|) \right\} . 
\end{equation*}
Here $h_+,h^{\perp}_+$ are given by
\begin{equation*}
 h_+(z):=\frac{1+z^2+2\cos\theta\sqrt{z^2-\sin^2\theta}}{\sin^2\theta}-2,
\end{equation*}
and
\begin{equation*}
 h^{\perp}_+(z):=\frac{1+z^2+2\sin\theta\sqrt{z^2-\cos^2\theta}}{\cos^2\theta}-2.
\end{equation*}
\end{lem}

\begin{proof}
Let $F_{t} = F(I +t\,v^{\perp}
_i \otimes v_i)$. Since $|F_0v_i|=|Fv_i| \leq 1$ and $|F_t v_i|\rightarrow \infty$ as
$t \rightarrow \pm\infty$, there are two values \,$t_{-}\leq 0\leq t_{+}$ such that $|F_{t_{\pm}} v_i| = 1$. We have 
\begin{equation*} W(F_{t_{-}})=|F_{t_{-}}|^2-2 = |F_{t_{-}}v_i^{\perp}|^2+|F_{t_{-}}v_i|^2-2=|Fv_i^{\perp}|^2-1, \end{equation*} 
and similarly $W(F_{t_{+}})=|Fv_i^{\perp}|^2-1$. Since $F$ is a rank-one convex combination of $F_{t_-}$ and $F_{t_+}$, the relation \eqref{oneslip.eq} follows from the definition \eqref{lamcon} of $W^{lc}$.

Let us next prove that $ W^{lc}(F)\leq h_+(|Fv_3|)$ for $F \in \mathcal{N}_1\backslash \mathcal{N}_2$. 
We take the rank-one line $F_t=F(I +t\,v_3 \otimes v^{\perp}_3)$ and construct laminates between its intersections with $\mathcal{M}_1, \mathcal{M}_2$. 
To do so, we consider the transformation $\xi:\mathbb{R}^{2\times 2}\rightarrow \mathbb{R}^{2\times 2}$ given by $F\mapsto RFR^{-1}=RFR$, where $R$ is the reflection about $v_3$. In particular, taking $v_3$ and $v_3^{\perp}$ as the coordinate system, 
$R=\begin{pmatrix}
   1 & 0 \\
   0 & -1 \\
  \end{pmatrix}$.
Note that if $\det F=1$, then $\det \xi (F)=1$. Moreover, since $Rv_1=v_2$ and $Rv_2 =v_1$, we have $\xi(\mathcal{N}_1)=\mathcal{N}_2,\, \xi(\mathcal{M}_1)=\mathcal{M}_2$ and $\xi(\mathcal{N}_2)=\mathcal{N}_1,\, \xi(\mathcal{M}_2)=\mathcal{M}_1$. 

Since $|(UF)v_i|=|Fv_i|$ for any $U\in SO(2)$, we can assume that $F$ is an upper triangular matrix $F = \begin{pmatrix} \alpha & \beta \\ 0 & \gamma \end{pmatrix}$ with $\alpha>0$. 
Then $\{F_t : \;  t\in\mathbb{R}\}=\{\xi(F_t) : \; t\in\mathbb{R}\}$ holds because
\begin{equation*} 
\xi(F_t) = RFR + tRFv_3\otimes v_3^{\perp} R = \begin{pmatrix} \alpha & -\beta \\ 0 & \gamma \end{pmatrix} - t \begin{pmatrix} 0 & \alpha \\ 0 & 0 \end{pmatrix} = F - (\alpha t + 2\beta) v_3\otimes v_3^{\perp} . 
\end{equation*}
From the proof of Lemma \ref{N1N2}, we see that there exist $t_{\pm}$ such that $t_{-} < 0 < t_{+}$ and $|F_{t_{\pm}} v_1| = 1$. 
Since the line $F_t$ is invariant with respect to the transformation $\xi$, one finds that it intersects the set $\mathcal{M}_1$ at $F_{t_{-}}$ and $F_{t_{+}}$, and the set $\mathcal{M}_2$ at $\xi(F_{t_{-}})$ and $\xi(F_{t_{+}})$. 
Since the intersection of $\mathcal{M}_1$ and $\mathcal{M}_2$ is the set $SO(2)$ which is not crossed by the line $F_t$ more than once, at least three of the matrices $F_{t_{-}},\,F_{t_{+}},\,\xi (F_{t_{-}}),\,\xi (F_{t_{+}})$ are different.

Regarding the intersections of $F_t$ with $\mathcal{M}_2$, there exist $t_{-}'$ and $t'_{+}$ such that $\xi (F_{t_{-}})=F_{t'_{-}}$ and $\xi (F_{t_{+}})=F_{t'_{+}}$. In view of the form of $\xi (F_t)$ and $\alpha>0$, we can see that $t'_{+}<t'_{-}$. Furthermore, since $|F_tv_2|^2$ is a quadratic function with a positive coefficient of the second-order term and $|Fv_2|\geq 1$, $t'_{-}$ and $t'_{+}$ have the same sign. 

We consider the two cases for the sign of $t_+'$ and $t_-'$ separately, beginning with the case $0<t'_{+}<t'_{-}$.
Note that then $t'_{-}\leq t_{+}$ cannot happen because then $t_-<0<t_+'<t_-'\leq t_+$, while $|F_t|^2$ is a quadratic function of $t$ and $|F_{t_{+}}|^2 =|\xi(F_{t_{+}})|^2 =|F_{t'_{+}}|^2$ along with $|F_{t_{-}}|^2 =|\xi(F_{t_{-}})|^2 =|F_{t'_{-}}|^2$. Therefore, $t_{-}<0<t'_{+},t_{+}<t'_{-}$, i.e., $t_{-}=\min\{t_{-},t_{+},t'_{-},t'_{+}\}$ and $t'_{-}=\max\{t_{-},t_{+},t'_{-},t'_{+}\}$. 
In this case we show that we can construct optimal laminates between $F_{t_{-}}$ and $F_{t'_{-}}$.
Let $F^{\pm}_{\gamma}=R_{\pm}(I+\gamma_{\pm} v_1\otimes v_1^{\perp})$ for $\gamma_{\pm}\in\mathbb{R}$ and $R_{\pm}\in SO(2)$ be the matrices on $\mathcal{M}_1$ such that $F_{t_{\pm}} = F^{\pm}_{\gamma}$. The calculations in Lemma \ref{psiphi} show that $W(F_{t_{\pm}})=|F_{t_{\pm}}|^2-2=W(F_{\gamma}^{\pm})=|F_{\gamma}^{\pm}|^2-2=\gamma_{\pm}^2$ is equal to either $h(|F_{\gamma}^{\pm}v_3|)=h(|Fv_3|)$ or to $h_+(|F_{\gamma}^{\pm}v_3|)=h_+(|Fv_3|)$, depending on the value of $\gamma_{\pm}$. 
We notice that $|F_{t_{+}}|\neq|F_{t_{-}}|$.
Indeed, if it is not the case, the quadratic function $|F_t|^2$ would take the same value for more than two values of $t$, in view of $|F_{t_{+}}|^2 =|F_{t'_{+}}|^2$ and $|F_{t_{-}}|^2 =|F_{t'_{-}}|^2$. 
This implies $\{W(F_{t_{+}}),W(F_{t_{-}})\}=\{h(|Fv_3|),h_+(|Fv_3|)\}$.
Now, $W(F_t)=|F_t|^2-2$ is a quadratic function of $t$ with positive coefficient of the quadratic term, $W(F_{t_+})=W(F_{t_+}')$, and $t_-<t_+',t_+<t_-'$. In view of $h(|Fv_3|)\leq h_+(|Fv_3|)$, we conclude that $W(F_{t_{-}})=W(F_{t'_-})=h_+(|Fv_3|)$, which implies $W^{lc}(F) \leq h_+(|Fv_3|)$.
The second case $t'_{+}<t'_{-}<0$ is analogous, leading to laminates between $F_{t_{+}}$ and $F_{t'_{+}}$.

This finishes the proof of $W^{lc}(F)\leq h_+(|Fv_3|)$ for $F\in \mathcal{N}_1\setminus \mathcal{N}_2$. The proof for $F \in \mathcal{N}_2\setminus \mathcal{N}_1$ is identical except for exchanging the roles of $v_1$ and $v_2$. Finally, the proof of $W^{rc}(F)\leq h_+^{\perp}(|Fv_3^{\perp} |)$ is done in the same way replacing the rank-one line $F_t=F(I +t\,v_3 \otimes v_3^{\perp})$ by $F_t=F(I +t\,v_3^{\perp} \otimes v_3)$. 
\end{proof}

\noindent
\textsl{Proof of Proposition \ref{prop_for_theorem2}}\\
First we confirm that for $F\in \mathcal{N}$, i.e., for $F\in\mathbb{R}^{2\times 2}$ with $\det F=1$, we have $W^{rc}(F) \leq W^{lc}(F)$ and $W^{pc}(F) \leq W^{lc}(F)$.
Indeed, the first inequality follows directly from the definition of the envelopes, while the second one follows by noting that we may restrict $F_0,F_1$ to matrices with determinant 1 in the definition \eqref{lamcon} of $W^{lc}$.
Then if $F=\mu F_0+ (1-\mu)F_1$ for $\mu \in [0,1]$ and $F_0,F_1\in \mathcal{N}$, we have $\det F=\mu \det F_0 + (1-\mu)\det F_1$. 
Since $W^{pc}(F)$ can be expressed as a convex function of $F$ and $\det F$, we get $W^{pc}(F) \leq \mu W^{pc}(F_0) + (1-\mu) W^{pc}(F_1) \leq \mu W(F_0) + (1-\mu) W(F_1)$.
Taking infimum over all rank-one connected $F_0,F_1$ proves the second inequality.

Lemmas \ref{evaluation.1}, \ref{N1N2} and \ref{psiphi}, together with Remark \ref{extension} imply 
\begin{equation*}
W^{pc}(F)=W^{rc}(F)=W^{lc}(F)=\max\{h(|Fv_3|),h^{\perp}(|Fv_3^{\perp}|)\} \;\; \text{for} \; F \in \overline{\mathcal{A}} \cup  (\overline{\mathcal{N}_1\cap \mathcal{N}_2}) \cup \overline{\mathcal{A}_{\perp}}.
\end{equation*}
Since, in view of Lemmas \ref{evaluation.1}, \ref{N1N2} and \ref{NiNi+1}, $W^{lc}$ is finite on $\mathcal{N}$, Theorem 3.1 in \cite{Conti2016} implies that $W^{qc}\leq W^{lc}$ on $\mathcal{N}$. 
Hence, taking into account Lemma \ref{psiphi}, all the envelopes $W^{pc}, W^{qc}, W^{rc}, W^{lc}$ are equal on $\overline{\mathcal{A}} \cup  (\overline{\mathcal{N}_1\cap \mathcal{N}_2}) \cup \overline{\mathcal{A}_{\perp}}$.
The remaining claims of Proposition \ref{prop_for_theorem2}, pertaining to values of envelopes on the boundary and to their infinite values, follow from Lemma \ref{psiphi} and the accompanying Remark \ref{extension}. 
\qed

\section{Homogenization for general slips}

In this section we assume that the angle between slip directions is arbitrary and, based on the analysis of convex envelopes in Section \ref{section_envelopes}, prove the partial homogenization result in Theorem \ref{maintheorem2}, assuming that the $\Gamma$-limit is an integral functional. This assumption is needed because the general theory (see, e.g., \cite{Braides1999}) that guarantees $\Gamma$-limits to be integral functionals applies only to sequences of functionals with standard growth conditions, while in our setting we deal with functionals whose admissible functions form a nonconvex subset of the underlying Sobolev space.

First, we present generalized versions of the two pillars of the proof, namely, Proposition \ref{liminf_affine} and Corollary \ref{cor-section3}. 
\begin{cor}
\label{liminf_affine2}
Let $\Omega = (0,l)^2$ for $l > 0$ be a cube, and let $(u_{\epsilon})_{\epsilon}\subset W^{1,2}(\Omega;\mathbb{R}^2)$ 
be such that 
$E_{\epsilon}(u_{\epsilon})\leq C$ for all $\epsilon>0$ and $u_{\epsilon}\rightharpoonup u$ in $W^{1,2}(\Omega;\mathbb{R}^2)$ for $u \in W^{1,2}(\Omega;\mathbb{R}^2)$ with gradient of the form \eqref{lim-form.1}.  If, in addition, $u$ is piecewise affine, then
\begin{equation}
\label{liminfaffgen}
\liminf_{\epsilon \rightarrow 0} E_{\epsilon}(u_{\epsilon}) 
 \geq  \lambda\int_{\Omega}f\left( \frac{1}{\lambda} \left(\nabla u-(1-\lambda)R \right) \right) \,dx,
\end{equation}
where
\begin{equation*} 
f(F) = \max \left\{ h(|Fv_3|), h^{\perp} (|Fv_3^{\perp}|) \right\}, 
\end{equation*}
with $v_3$ defined by \eqref{defv3} and $h,h^{\perp}$ by \eqref{newdefh}. 
In particular, for $N:=\frac{1}{\lambda}(\nabla u-(1-\lambda)R)$,
\begin{equation*}
\liminf_{\epsilon \rightarrow 0} E_{\epsilon}(u_{\epsilon}) 
\geq \begin{cases}
\lambda|\Omega|\; h\left(|Nv_3 |\right) & \text{if}\;\; N\in \overline{\mathcal{A}}\cup \overline{\mathcal{N}_1\cap \mathcal{N}_2},\\
\lambda|\Omega|\; h^{\perp}\left(|Nv_3^{\perp}| \right) & \text{if}\;\; N\in \overline{\mathcal{A}_{\perp}}.
\end{cases}
\end{equation*}
\end{cor}
\begin{proof}
As noted in the proof of Lemma \ref{psiphi}, $f$ is a convex function on $\mathbb{R}^{2\times 2}$ and hence continuous. 
Moreover, Lemma \ref{psiphi} and Remark \ref{extension} show that $f(\nabla u_{\epsilon})=W(\nabla u_{\epsilon})$ a.e. in $\epsilon Y_{soft}\cap \Omega$ for $u_{\epsilon}$ with finite energy.
For such $u_{\epsilon}$ we additionally have $f(\nabla u_{\epsilon})=0$ a.e. in $\epsilon Y_{rig}\cap \Omega$.
Therefore, beginning with
\begin{equation*} 
\liminf_{\epsilon \rightarrow 0}E_{\epsilon}(u_{\epsilon}) 
= \liminf_{\epsilon \rightarrow 0}\int_{\epsilon Y_{soft}\cap \Omega} \;W(\nabla u_{\epsilon})\; dx
= \liminf_{\epsilon \rightarrow 0}\int_{\Omega} f(\nabla u_{\epsilon})\,dx , 
\end{equation*}
the arguments in the proof of Proposition \ref{liminf_affine} are still valid, yielding the relation \eqref{liminfaffgen}. 

Finally, the analysis in Section \ref{section_envelopes} implies
\begin{equation*}
\lambda|\Omega|\; f\left(N\right)=
\begin{cases}
\lambda|\Omega|\; h\left(|Nv_3| \right) & \text{if}\;\; N\in  \overline{\mathcal{A}}\cup \overline{\mathcal{N}_1\cap \mathcal{N}_2},\\
\lambda|\Omega|\; h^{\perp}\left(|Nv_3^{\perp} \right|) & \text{if}\;\; N\in \overline{\mathcal{A}_{\perp}}.
\end{cases}
\end{equation*}
\end{proof}

\begin{cor}
\label{cor-section3.2}
Let $\Omega\subset \mathbb{R}^2$ be a bounded domain. If $N\in \overline{\mathcal{A}}\cup \overline{\mathcal{N}_1\cap \mathcal{N}_2}\cup \overline{\mathcal{A}_{\perp}}$, let $F_{+},F_{-}\in \mathcal{M}_1\cup \mathcal{M}_2$ and $\mu \in(0,1)$ be such that $N=\mu F_+ + (1-\mu) F_-$ and $\mu \in(0,1)$. Specifically,
\begin{itemize}
\item[(i)] if $N\in \mathcal{A}$ or $N\in \mathcal{A}_{\perp}$, let $F_-=F_{s_\star}$, $F_+=F_{t_{\star}}$ with $F_{s_{\star}}, F_{t_{\star}}$ defined in the proof of Lemma \ref{evaluation.1}, and a suitable $\mu$;
\item[(ii)] if $N\in \mathcal{N}_1\cap \mathcal{N}_2$, let $F_-=F_{t_{1,-}}$, $F_+=F_{t_{2,+}}$ with $F_{t_{1,-}}, F_{t_{2,+}}$ defined in the proof of Lemma \ref{N1N2}, and a suitable $\mu$;
\item[(iii)] if $N\in \mathcal{M}_1\cup \mathcal{M}_2$, let $F_+=F_-=N$ and $\mu \in(0,1)$ be arbitrary. 
\end{itemize}
Then for every $\delta >0$ there exists $\Omega_{\delta}\subset\Omega$ with $|\Omega\backslash\Omega_{\delta}|<\delta$ and $u_{\delta}\in W^{1,\infty}(\Omega;\mathbb{R}^2)$ such that $u_{\delta}$ coincides with a simple laminate between $F_{+}$ and $F_{-}$ with weights $\mu$ and $1-\mu$ and
period $h_{\delta}<\delta \,$ in $\Omega_{\delta}$ , $\,\nabla u_{\delta}\in \mathcal{M}_1\cup \mathcal{M}_2$ a.e. in $\Omega$, 
and $u_{\delta}=Nx\,on\, \partial\Omega$. Moreover, there is a constant $c$ depending only on $N$, such that for any $\delta \in (0,1)$,
\begin{equation}
\label{udbound2}
|\nabla u_{\delta}|\leq c \quad \text{a.e. in} \;\; \Omega .
\end{equation}
In particular, $\nabla u_{\delta}\rightharpoonup N$ in $L^2(\Omega;\mathbb{R}^{2\times 2})$ as $\delta \rightarrow 0.$
\end{cor}
\begin{proof}
The proof is almost the same as for Corollary \ref{cor-section3}. Let us only deal with the case that $N\in \mathcal{A}$, since a similar argument applies to the cases $N\in \mathcal{A}_{\perp}$ and $N\in (\mathcal{N}_1\cap \mathcal{N}_2)\backslash (\mathcal{M}_1\cup \mathcal{M}_2)$.
For the given $F_+,F_-$ and $\delta >0$, 
Theorem \ref{rank-one-approximation} with $p$ taken as $v_3$ yields a finitely piecewise affine function $v_{\delta}$ and a set $\Omega_{\delta}$. 
Since $\operatorname{dist} (\nabla v_{\delta},[F_{+},F_{-}])\leq\delta$ and $0<\delta<1$, there is $c>0$ independent of $\delta$, such that
\begin{equation*} |\nabla v_{\delta}q|<c \qquad \text{for} \;  q=v_3,v_3^{\perp},v_1,v_2 . \end{equation*}

The desired function $u_{\delta}$ is obtained as a result of applying Theorem \ref{convex-integration} to modify $v_{\delta}$ in the finitely many subsets of $\Omega\setminus\Omega_{\delta}$ where it is affine. Taking any such subset $S$, there are three possible cases for the value of $\nabla v_{\delta}$ in $S$: either $\nabla v_{\delta}\in \mathcal{A}$ or $\nabla v_{\delta}\in \mathcal{A}_{\perp}$ or $\min\{|\nabla v_{\delta}v_1|,|\nabla v_{\delta}v_2|\} <1$. We first investigate the case $\nabla v_{\delta}\in \mathcal{A}$ in $S$. Recalling that the condition $|Fv_3|/|Fv_3^{\perp}|>\sin\theta/\cos\theta$ is equivalent to $Fv_1\cdot Fv_2 >0$, we can apply Theorem \ref{convex-integration} to the in-approximation $(U^{\delta}_{i})_i$ of $\mathcal{K}:=(\mathcal{M}_1\cap \mathcal{M}_2)\cap\{F\in\mathbb{R}^{2\times 2}:|Fv_3|\leq c,\,Fv_1\cdot Fv_2 \geq 0\}$ defined as 
\begin{align*}
&U_i^{\delta}:=\left\{F\in\mathbb{R}^{2\times 2}: \det F=1, |Fv_3|<c, |Fv_1|>1,|Fv_2|>1, Fv_1\cdot Fv_2>0\right\}\\
&\qquad\cap \left\{F\in\mathbb{R}^{2\times 2}:|Fv_1|<1+2^{-(i-1)}\,\text{or}\,|Fv_2|<1+2^{-(i-1)}\right.\},\quad i\in\mathbb{N}.
\end{align*}
It can be shown similarly to Corollary \ref{cor-section3} that $(U^{\delta}_{i})_i$ is an in-approximation of $\mathcal{K}$. The case when $\nabla v_{\delta}\in \mathcal{A}_{\perp}$ in $S$ is handled analogously. The case when $\min\{|\nabla v_{\delta}s|,|\nabla v_{\delta}m|\} <1$ in a subset $S$ can be reduced to the proof of Lemma 2 of \cite{ContiTheil2005}.
The estimate \eqref{udbound2} can also be proved in the same way as in the proof of Corollary \ref{cor-section3}.
\end{proof}

We now turn to the proof of Theorem \ref{maintheorem2}.
Noting that the set of admissible functions for the $\Gamma$-limit $E$ coincides with the set of weak limits $\lim_{j\to\infty} u_j$ of admissible sequences $(u_j)_j$ for $(E_j)_j$, we deduce from Proposition \ref{lim-form.1} that $E$ is finite exactly for functions $u$ with the form \eqref{lim-form.1}.\\

\noindent
\textsl{Step 1: Liminf inequality for affine $u$}\\
Let $u$ be affine and let $(u_j)_j$ be a sequence  converging to $u$ in $L^2(\Omega;\mathbb{R}^2)$ such that $(E_{\epsilon_j}(u_j))_j$ is bounded, where $\epsilon_j\rightarrow 0+$ as $j\rightarrow \infty$. 
Then, taking a subsequence if necessary, $u_j\rightharpoonup u$ in $W^{1,2}(\Omega;\mathbb{R}^2)$ and $\nabla u=R(I+\gamma e_1\otimes e_2)$ for some $R\in SO(2)$ and $\gamma\in L^2(\Omega)$. Similarly to the proof of Theorem \ref{maintheorem}, it is sufficient to consider the case when $\Omega$ is a square. 
If $N=\frac{1}{\lambda}(\nabla u-(1-\lambda)R) \in \overline{\mathcal{A}}\cup \overline{\mathcal{N}_1\cap \mathcal{N}_2}\cup \overline{\mathcal{A}_{\perp}}$, we obtain from Corollary \ref{liminf_affine2} 
\begin{equation}
\label{liminf-constant}
\liminf_{j\rightarrow \infty}E_{\epsilon_j}(u_j)\geq \lambda |\Omega|\, f(N).
\end{equation}

\noindent
\textsl{Step 2: Recovery sequence for affine $u$}\\
Let $u$ be affine. If $N=\frac{1}{\lambda}(\nabla u-(1-\lambda)R)\in \overline{\mathcal{A}}\cup \overline{\mathcal{N}_1\cap \mathcal{N}_2}\cup \overline{\mathcal{A}_{\perp}}$, Corollary \ref{cor-section3.2} can be used to perform convex integration as in the proof of Theorem \ref{maintheorem} (see Section \ref{sec_rs3}), and to show the existence of a recovery sequence $(u_j)_j$. Then 
\begin{equation}
\label{recovery-constant}
\lim_{j\rightarrow \infty}E_{\epsilon_j}(u_j)= \lambda |\Omega| \, f(N).
\end{equation}

\noindent
\textsl{Step 3: $W_{hom}=f$ on $\overline{\mathcal{A}}\cup \overline{\mathcal{N}_1\cap \mathcal{N}_2}\cup \overline{\mathcal{A}_{\perp}}$}\\
Let $N\in \overline{\mathcal{A}}\cup \overline{\mathcal{N}_1\cap \mathcal{N}_2}\cup \overline{\mathcal{A}_{\perp}}$ and take $u$ affine such that $\frac{1}{\lambda}(\nabla u-(1-\lambda)R)=N$.
By the assumption that the $\Gamma$-limit is an integral functional with density $\lambda W_{hom}$, there is a recovery sequence $(u_j)_j$ such that $\lim_{j\rightarrow \infty} E_{\epsilon _j}(u_j)=\lambda |\Omega|W_{hom}(N)$. From this and \eqref{liminf-constant}, $f(N)\leq W_{hom}(N)$. On the other hand, from \textsl{Step 2}, there is a recovery sequence $(u_j)_j$ such that \eqref{recovery-constant} holds. Using the liminf inequality in the definition of $\Gamma$-limit, we get $\liminf_{j\rightarrow \infty}E_{\epsilon_j}(u_j)\geq \lambda |\Omega|W_{hom}(N)$, and hence $W_{hom}(N)\leq f(N)$. These two inequalities yield $W_{hom}(N)=f(N)$.
\qed

\newpage
\appendix
\section{Appendix}
\subsection{Averaging Lemma}
\label{sec_averaginglemma}
The standard averaging lemma is a special case of the following claim.
\begin{lem}[Averaging lemma for ``doubly oscillating'' sequences]
\label{averaging} $\;$\\
Let $\Omega$ be a bounded open set in $\mathbb{R}^n$, and let $g_{\epsilon}\in L^2_{loc}(\mathbb{R}^n)
$ be $Y$-periodic functions, where $Y$ is an $n$-cube. If $g_{\epsilon}\rightharpoonup g$ in $L^2(Y)$, then
\begin{equation*}
g_{\epsilon}\left(\frac{x}{\epsilon}\right)  \rightharpoonup  \langle g\rangle :=\frac{1}{|Y|}\int_{Y}g(x) \, dx \quad \text{in }L^2(\Omega) .
\end{equation*}
\end{lem}
\begin{proof}
We prove this theorem by generalizing the proof for the standard averaging lemma in \cite{Berlyand2018}, see also \cite{Lukkassen2002}. 
Since $C^{\infty}_0(\Omega)$ is dense in $L^2(\Omega)$, we only need to prove the following for every $\theta\in C^{\infty}_0(\Omega)$:
\begin{equation}
\label{weak-form}
\int_{\Omega}g_{\epsilon}\left(\tfrac{x}{\epsilon}\right)\theta(x)\, dx \rightarrow \langle g \rangle \int_{\Omega}\theta(x) \, dx.
\end{equation}
First, we show the uniform boundedness of $g_{\epsilon}(x/\epsilon)$ in $L^2(\Omega)$.
To this end, we introduce the notation $\epsilon Y_i :=x^{\epsilon}_i+\epsilon Y$ , where $x^{\epsilon}_i\in \epsilon\mathbb{Z}^n\cap \Omega$, i.e., $\{x^{\epsilon}_i, \, i\in\mathbb{Z} \}$ is the set of grid points separated by the distance $\epsilon$. Then
\begin{equation*}
\int_{\Omega}\left|g_{\epsilon}\left(\tfrac{x}{\epsilon}\right)\right|^2 dx=\sum_i \int_{\epsilon Y_i\cap\Omega}\left|g_{\epsilon}\left(\tfrac{x}{\epsilon}\right)\right|^2dx\leq \frac{|\Omega|}{\epsilon^n}\int_{Y}|g_{\epsilon}(y)|^2 \epsilon^n \, dy= |\Omega|\|g_{\epsilon}\|^2_{L^2(Y)}.
\end{equation*}
Since $g_{\epsilon}\rightharpoonup g$, there is $C>0$ such that
\begin{equation}
\label{uniform-boundedness-average}
\left\|g_{\epsilon}\left(\tfrac{x}{\epsilon}\right)\right\|_{L^2(\Omega)} \leq C.
\end{equation}
We define a piecewise constant interpolation of $\theta$ by
\begin{equation*}
\theta_{\epsilon}(x)=\theta(x^{\epsilon}_i),\quad x\in \epsilon Y_i \cap \Omega.
\end{equation*}
By Cauchy-Schwarz inequality and \eqref{uniform-boundedness-average},
\begin{equation}
\label{2.24}
\int_{\Omega}\left|g_{\epsilon}\left(\tfrac{x}{\epsilon}\right)(\theta(x) - \theta_{\epsilon}(x))\right| \, dx \leq \|g_{\epsilon}\left( \tfrac{x}{\epsilon}\right)\|_{L^2(\Omega)}\|\theta - \theta_{\epsilon}\|_{L^2(\Omega)}
\leq C\|\theta - \theta_{\epsilon}\|_{L^2(\Omega)} .
\end{equation}
Since $\theta$ has a compact support in $\Omega$, making $\epsilon$ small enough we can ignore cubes $\epsilon Y_i$ intersecting the boundary $\partial \Omega$. Thus,
\begin{equation}
\label{2.25}
\int_{\Omega}g_{\epsilon}\left(\tfrac{x}{\epsilon}\right) \theta_{\epsilon}(x) \, dx=\sum_i\int_{\epsilon Y_i}g_{\epsilon}\left(\tfrac{x}{\epsilon}\right) \theta_{\epsilon}(x) \, dx=
\langle g_{\epsilon} \rangle \int_{\Omega}\theta_{\epsilon}(x) \, dx .
\end{equation}
In view of $\|\theta - \theta_{\epsilon}\|_{L^2(\Omega)}\rightarrow 0$ and $\langle g_{\epsilon} \rangle \int_{\Omega}\theta_{\epsilon}(x) \, dx \rightarrow \langle g \rangle \int_{\Omega}\theta(x) \, dx$ as $\epsilon \rightarrow 0$, the convergence \eqref{weak-form} follows from \eqref{2.24} and \eqref{2.25}.
\end{proof}

\subsection{Proof of Lemma \ref{lemFAD}}
\label{prooflemFAD}

\noindent
{\bf Lemma 2.5.} {\it
Let $v_2=v_1^{\perp}$ and $F=A+D$, where $D=R\gamma_1 e_1\otimes e_2$ and either $A=R(I+\gamma_2 v_2\otimes v_1)$ or $A=R(I+\gamma_2 v_1\otimes v_2)$ for some $\gamma_1, \gamma_2$ and $R\in SO(2)$.
Then there exists a constant $c$, such that
\begin{equation*} f(F) \leq |F|^2 - 2 +c\left(\sqrt{|D|}+|D|\right)\left( \sqrt{|A|}+ |A|+\sqrt{|D|}+|D|\right) .\end{equation*}
}

\begin{proof}
We focus on the case $|Av_2|=1$, i.e., $A=R(I+\gamma_2 v_2\otimes v_1)$ since the remaining case $|Av_1|=1$ is analogous. We can also assume $R=I$ since rotations do not change any of the terms appearing in the inequality to be proved.
First, we denote the components of $v_i$ by $(v_i^{(1)},v_i^{(2)})$ and do a few preliminary calculations:
\begin{align*}
|Fv_2|^2 &= |v_2+\gamma_1v_2^{(2)}e_1|^2 = 1 + 2\gamma_1v_2^{(1)}v_2^{(2)}+\gamma_1^2(v_2^{(2)})^2 \\
|Fv_1|^2 &= |v_1+\gamma_1v_1^{(2)}e_1+\gamma_2v_2|^2 = 1 + \gamma_2^2 + \gamma_1^2 (v_1^{(2)})^2 + 2\gamma_1v_1^{(1)}v_1^{(2)}+2\gamma_1\gamma_2v_1^{(2)}v_2^{(1)} \\
|F|^2 &= 2+\gamma_1^2+\gamma_2^2 +2\gamma_1\gamma_2 v_2^{(1)} v_1^{(2)} \\
|A|^2 &= 2+ \gamma_2^2 \\
|D|^2 &= \gamma_1^2 \\
\det F &= 1- \gamma_1\gamma_2 v_2^{(2)}v_1^{(2)}
\end{align*}
Recalling that
\begin{equation*} f(F) = \max \left\{ (|Fv_1|^2-1)_+, \; (|Fv_2|^2-1)_+, \;  \chi\left(\max\{|Fv_3|,|Fv_3^{\perp}|\}\right) \right\}, \end{equation*}
it is enough to prove the following 3 inequalities:
\begin{align}
\label{FAD1}
(|Fv_1|^2-1)_+ &\leq |F|^2 - 2 +c(|D|+|A||D|) \\
\label{FAD2}
(|Fv_2|^2-1)_+ &\leq |F|^2 - 2 +c|A||D| \\
\label{FAD3}
 \chi\left(\max\{|Fv_3|,|Fv_3^{\perp}|\}\right) &\leq |F|^2 - 2 + c(A,D) 
\end{align}
where $c(A,D)$ stands for $c(\sqrt{|D|}+|D|)( \sqrt{|A|}+ |A|+\sqrt{|D|}+|D|)$. 
When $|Fv_1|\geq 1$, since $|F|^2=|Fv_2|^2+|Fv_1|^2$, inequality \eqref{FAD1} is equivalent to
\begin{equation*} -|Fv_2|^2+ 1 \leq  c\left( |D|+ |A||D|\right), \end{equation*}
which is true in view of the preliminary calculations because
\begin{equation*} -|Fv_2|^2+ 1 = -2\gamma_1v_2^{(1)}v_2^{(2)}-\gamma_1^2(v_2^{(2)})^2 \leq 2|\gamma_1|\leq 2|D|. \end{equation*}
On the other hand, when $|Fv_1|<1$, we have to prove the inequality
\begin{equation*} -|F|^2 +2 \leq c\left( |D|+|A||D|\right), \end{equation*}
which is again true due to
\begin{equation*}
-|F|^2+ 2 =-\gamma_1^2-\gamma_2^2 -2\gamma_1\gamma_2 v_2^{(1)} v_1^{(2)}  \leq 2|\gamma_1| |\gamma_2| \leq 2 |D||A|. 
\end{equation*}
The reasoning for the second inequality \eqref{FAD2} is similar and we omit it.

To address the third inequality \eqref{FAD3}, we first notice that 
\begin{equation*} \max\{|Fv_3|^2,|Fv_3^{\perp}|^2\} = \frac{1}{2}(|Fv_1|^2 + |Fv_2|^2+ 2|Fv_1\cdot Fv_2|) \geq 1 \end{equation*}
because 
\begin{equation*} |Fv_1|^2 + |Fv_2|^2 =|F|^2 = 2+\gamma_1^2+\gamma_2^2 +2\gamma_1\gamma_2 v_2^{(1)} v_1^{(2)} \geq 2+(|\gamma_1|-|\gamma_2|)^2 . \end{equation*}
Hence we can omit the plus subscripts in $\chi(t)=((2z^2-1)^{1/2}_+-1)_+^2$ when $z^2=\max\{|Fv_3|^2,|Fv_3^{\perp}|^2\}$ and write
\begin{align*}
 \chi &\left(\max\{|Fv_3|,|Fv_3^{\perp}|\}\right)\\
&=\left(\sqrt{|Fv_1|^2+|Fv_2|^2+2\sqrt{|Fv_1|^2|Fv_2|^2-|\det F|^2}-1} -1\right)^2 .
\end{align*}
Setting $x := |Fv_1|^2$, $\delta_1 := |Fv_2|^2-1$, $\delta_2 := |Fv_1|^2(|Fv_2|^2-1) +1-|\det F|^2$ for brevity, this transforms into 
\begin{equation*} 
\chi \left(\max\{|Fv_3|,|Fv_3^{\perp}|\}\right) =\left(\sqrt{x+\delta_1 +2\sqrt{x-1+\delta_2}} -1\right)^2 .
\end{equation*}
For now we assume that $|Fv_1|^2\geq 1$ to estimate
\begin{align*}
& \left(\sqrt{x+\delta_1 +2\sqrt{x-1+\delta_2}} -1\right)^2 - \left(\sqrt{x+2\sqrt{x-1}} -1\right)^2 \\
& \quad \leq  \left(\sqrt{x+|\delta_1| +2\sqrt{x-1}+2\sqrt{|\delta_2|}} -1\right)^2 - \left(\sqrt{x+2\sqrt{x-1}} -1\right)^2 \\
& \quad =  \left(\sqrt{x+|\delta_1| +2\sqrt{x-1}+2\sqrt{|\delta_2|}}  -\sqrt{x+2\sqrt{x-1}} \right) \\
& \qquad \times \left(\sqrt{x+|\delta_1| +2\sqrt{x-1}+2\sqrt{|\delta_2|}} + \sqrt{x+2\sqrt{x-1}} -2\right).
\end{align*}
Denoting the content of the last bracket by $b$, we continue
\begin{align*}
& \quad = \frac{|\delta_1| + 2\sqrt{|\delta_2|}}{\sqrt{x+|\delta_1| +2\sqrt{x-1}+2\sqrt{|\delta_2|}}+\sqrt{x+2\sqrt{x-1}}} \times b \\
& \quad \leq \frac{ |\delta_1| + 2\sqrt{|\delta_2|}}{\frac{1+\sqrt{2}}{\sqrt{2}}(\sqrt{x-1}+1)+\frac{1}{\sqrt{2}}\sqrt{|\delta_1| + 2\sqrt{|\delta_2|}}} \times b \\
& \quad \leq \frac{\left( 2\sqrt{x-1} + \sqrt{|\delta_1| + 2\sqrt{|\delta_2|}} \right) \left( |\delta_1| + 2\sqrt{|\delta_2|}\right)}{\frac{2}{\sqrt{2}}\sqrt{x-1}+\frac{1}{\sqrt{2}}\sqrt{|\delta_1| + 2\sqrt{|\delta_2|}}} \\
& \quad \leq \sqrt{2} \left( |\delta_1| + 2\sqrt{|\delta_2|}\right),
\end{align*}
where the last three inequalities follow from $\sqrt{a+b}\geq \frac{1}{\sqrt{2}}(\sqrt{a}+\sqrt{b})$ for $a,b\geq 0$, from the fact that $\sqrt{x+2\sqrt{x-1}} = \sqrt{x-1}+1$ and that
\begin{equation*}
b \leq 2\sqrt{x+2\sqrt{x-1}} + \sqrt{|\delta_1| +2\sqrt{|\delta_2|}} -2 = 2 \sqrt{x-1} +\sqrt{|\delta_1| + 2\sqrt{|\delta_2|}} .
\end{equation*}
The above estimates imply
\begin{equation*} \chi \left(\max\{|Fv_3|,|Fv_3^{\perp}|\}\right) \leq |Fv_1|^2 -1 +\sqrt{2} \left( |\delta_1| + 2\sqrt{|\delta_2|}\right), \end{equation*}
and since we have already estimated $|Fv_1|^2-1$, it remains to show that $\delta_1, \delta_2$ have the right order.
This follows from
\begin{align*}
|\delta_1| & = \left| 2\gamma_1v_2^{(1)}v_2^{(2)}+\gamma_1^2(v_2^{(2)})^2 \right| \leq 2|D|+|D|^2 ,\\
|\delta_2| & = \left| \left(  1 + \gamma_2^2 + \gamma_1^2 (v_2^{(2)})^2 + 2\gamma_1v_1^{(1)}v_1^{(2)}+2\gamma_1\gamma_2v_1^{(2)}v_2^{(1)} \right)\left( 2\gamma_1v_2^{(1)}v_2^{(2)}+\gamma_1^2(v_2^{(2)})^2 \right) \right. \\
 & \qquad + \left. 1 -\left( 1-\gamma_1\gamma_2 v_2^{(2)}v_1^{(2)}\right)^2\right| \\
& \leq 4(|D|+|D|^2)^2+2|D||A|(1+|D|)^2 + 2|D||A|^2+2|D|^2|A|^2 .
\end{align*}

Next, when $|Fv_1|^2<1$, we set $x := |Fv_2|^2 >1$, $
\delta_1 := |Fv_1|^2-1 <0$ and 
\begin{equation*}
\delta_2 := |Fv_2|^2(|Fv_1|^2-1) +1-|\det F|^2 \leq 1-|\det F|^2 \leq 2\gamma_1 \gamma_2 v_2^{(2)} v_1^{(2)},
\end{equation*}
and an analogous estimation process leads to
\begin{equation*} 
\chi \left(\max\{|Fv_3|,|Fv_3^{\perp}|\}\right) \leq |Fv_2|^2 -1 + 4\sqrt{|\gamma_1 \gamma_2 v_2^{(2)} v_1^{(2)}|}, 
\end{equation*}
which is the desired estimate in view of having already estimated $|Fv_2|^2-1$.
\end{proof}

\subsection{Proof of Remark \ref{rem_estimateWhom}}
\label{appendix_Whom}

{\it There is a constant $c$ depending only on $\lambda$ such that
\begin{equation*} 
|W_{hom}(\gamma_1) - W_{hom}(\gamma_2)| \leq c(1+|\gamma_1|+|\gamma_2|)  |\gamma_1-\gamma_2| ,
\end{equation*}
where we write $W_{hom}(\gamma)$ instead of $W_{hom}(\nabla u)$ with $\nabla u = R(I+\gamma e_1\otimes e_2)$.}\\

\begin{proof}
Let $(a,b)$ stand for the components of the vector $v_1$. Then $v_2$ has components $(-b,a)$. For $\nabla u = R(I+\gamma e_1\otimes e_2)$ we have $N=R(I+\frac{\gamma}{\lambda} e_1\otimes e_2)$ and thus
\begin{gather*}
|Nv_2|^2 = 1 - 2ab \frac{\gamma}{\lambda} + a^2 \frac{\gamma^2}{\lambda^2}, \qquad
|Nv_1|^2 = 1 + 2ab \frac{\gamma}{\lambda} + b^2 \frac{\gamma^2}{\lambda^2}, \\
|N(v_1 \pm v_2)|^2 = 2 \pm 2\frac{\gamma}{\lambda}(a^2-b^2) + \frac{\gamma^2}{\lambda^2}(1 \pm 2ab). 
\end{gather*}
Then for $ab >0$,
\begin{align*}
|Nv_2|\leq 1 \qquad &\Leftrightarrow \qquad 0 \leq \gamma \leq 2\lambda \frac{b}{a} \\
|Nv_1|\leq 1 \qquad &\Leftrightarrow \qquad -2\lambda \frac{a}{b} \leq \gamma \leq 0 \\
|Nv_2|, |Nv_1| \geq 1 \qquad &\Leftrightarrow \qquad \gamma \geq 2\lambda \frac{b}{a} \;\; \text{or} \;\; \gamma \leq -2\lambda \frac{a}{b} \\
|N(v_1+v_2)| \geq |N(v_1-v_2)| \qquad &\Leftrightarrow \qquad \gamma \geq \max \left\{ 0, \lambda \frac{b^2-a^2}{ab} \right\} \\
& \qquad\qquad \text{or} \;\; \gamma \leq \min \left\{ 0, \lambda \frac{b^2-a^2}{ab} \right\},
\end{align*}
where the last inequality $|N(v_1+v_2)| \geq |N(v_1-v_2)|$ always holds when $|Nv_2|\geq 1$ and $|Nv_1|\geq 1$.
An analogous calculation is done for the case $ab<0$ leading to the following expression for $W_{hom}$ when $ab >0$:
\begin{equation*} W_{hom}(N) = \left\{ \begin{array}{l}
2ab \frac{\gamma}{\lambda} + b^2 \frac{\gamma^2}{\lambda^2} \qquad \quad  \text{when} \;\; 0 \leq \gamma \leq 2\lambda \frac{b}{a} \\
- 2ab \frac{\gamma}{\lambda} + a^2 \frac{\gamma^2}{\lambda^2} \qquad \; \text{when} \;\; -2\lambda \frac{a}{b} \leq \gamma \leq 0 \\
\left(\sqrt{(1+2(a^2-b^2)\frac{\gamma}{\lambda} +(1+2ab) \frac{\gamma^2}{\lambda^2})_+} -1 \right)_+^2 \quad  \text{otherwise}, 
\end{array} \right. \end{equation*}
and when $ab <0$:
\begin{equation*} W_{hom}(N) = \left\{ \begin{array}{l}
2ab \frac{\gamma}{\lambda} + b^2 \frac{\gamma^2}{\lambda^2} \qquad \quad \text{when} \;\; 2\lambda \frac{b}{a} \leq \gamma \leq 0 \\
- 2ab \frac{\gamma}{\lambda} + a^2 \frac{\gamma^2}{\lambda^2} \qquad \; \text{when} \;\; 0 \leq \gamma \leq -2\lambda \frac{a}{b} \\
\left(\sqrt{(1+2(b^2-a^2)\frac{\gamma}{\lambda} +(1-2ab) \frac{\gamma^2}{\lambda^2})_+} -1 \right)_+^2 \quad  \text{otherwise} .
\end{array} \right. \end{equation*}
The $+$ subscript can be removed in the given ranges because, for example in the setting $ab >0$, for $ \gamma \geq 2\lambda \frac{b}{a}$ one can estimate
\begin{equation*} 
1+ 2(a^2-b^2)\frac{\gamma}{\lambda} +(1+2ab) \frac{\gamma^2}{\lambda^2} = \left(1+\frac{\gamma}{\lambda} \right)^2 +2ab\frac{\gamma}{\lambda} \left( \frac{\gamma}{\lambda} - 2 \frac{b}{a} \right)\geq 1,
\end{equation*}
and similarly for the remaining cases.

\begin{figure}[htbp]
\begin{center}
\includegraphics[width=0.6\textwidth]{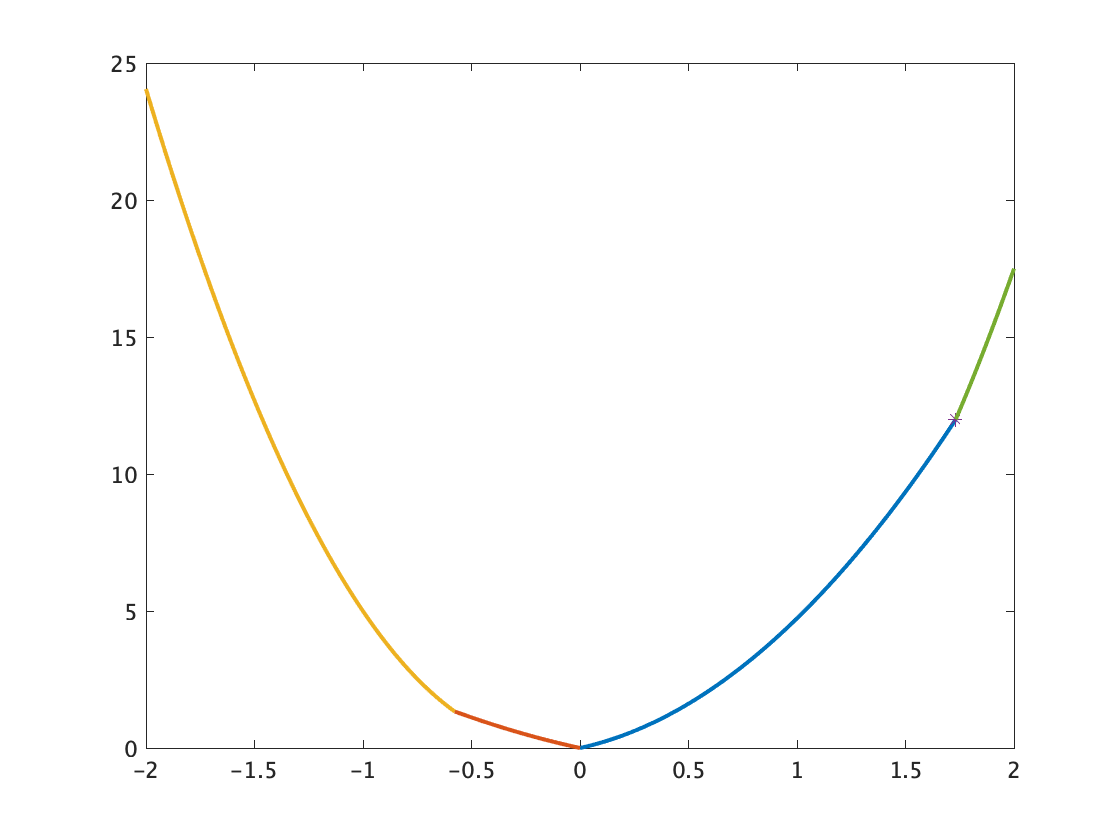}
\caption{Graph of $W_{hom}$ as function of $\gamma$ for $v_1=(\frac{1}{2},\frac{\sqrt{3}}{2})$ and $\lambda = 0.5$.}
\end{center}
\label{figWhom}
\end{figure}

Estimating the derivative with respect to $\gamma$ when $ab>0$, we see that
\begin{itemize}
\item for $0 \leq \gamma \leq 2\lambda \frac{b}{a}$: 
$\displaystyle \quad \left| \frac{\partial W_{hom}}{\partial \gamma} \right| = \left| 2\frac{ab}{\lambda} + 2\frac{b^2}{\lambda^2} \gamma \right| \leq \frac{2}{\lambda^2} (1+|\gamma|) $
\item for $-2\lambda \frac{a}{b} \leq \gamma \leq 0$:
$\displaystyle \quad \left| \frac{\partial W_{hom}}{\partial \gamma} \right| = \left| -2\frac{ab}{\lambda} + 2\frac{a^2}{\lambda^2} \gamma \right| \leq \frac{2}{\lambda^2} (1+|\gamma|) $
\item otherwise: setting $R=\sqrt{1+2(b^2-a^2)\frac{\gamma}{\lambda} +(1-2ab) \frac{\gamma^2}{\lambda^2}} \geq 1$,
\begin{equation*} \left| \frac{\partial W_{hom}}{\partial \gamma} \right| =
\left| 2 \left(R -1 \right)\frac{\frac{b^2-a^2}{\lambda} +\frac{1-2ab}{\lambda^2} \gamma}{R} \right| \leq \frac{2}{\lambda^2} (1+|\gamma|) .\end{equation*}
\end{itemize}
The computation is analogous for $ab<0$. This means that there is a constant $c(\lambda)$ such that for all $\gamma_1,\gamma_2\in\mathbb{R}$,
\begin{equation*} 
|W_{hom}(\gamma_1) - W_{hom}(\gamma_2)| \leq c(1+|\gamma_1|+|\gamma_2|) |\gamma_1-\gamma_2| .
\end{equation*}
\end{proof}

\vspace{1cm}

\noindent{\bf Acknowledgement}: This research was supported by JSPS Kakenhi Grant numbers 19K03634 and 18H05481.

\end{document}